\documentclass[reqno,12pt]{amsart}
 \usepackage{amsmath}
 \usepackage{amsthm}
 \usepackage{amssymb}
 \usepackage{latexsym,longtable}
 \usepackage{graphicx}
 \usepackage{multicol}
 \usepackage{mathrsfs}
 \usepackage{young}
 \usepackage[vcentermath,enableskew,stdtext]{youngtab}
 \usepackage[left=1.1in,right=1.1in,top=1.05in,bottom=1.08in]{geometry}
 \usepackage[all]{xy}
    \SelectTips{cm}{10}     
    \everyxy={<2.5em,0em>:} 
 \usepackage{fancyhdr}      
 \usepackage{multicol}
 \usepackage{multirow}
 \usepackage{enumitem}
 \usepackage{bm}
 \usepackage{stmaryrd}
 \usepackage{etex}
 \usepackage{tikz}
 \usepackage{float}
 \usepackage{array}
 \usepackage{colortbl}
 \usepackage{mathtools}
 \restylefloat{figure}
\numberwithin{equation}{section}
 \usetikzlibrary{shapes}
 \usetikzlibrary{arrows}
 \usetikzlibrary{snakes}
\usepackage{url}

\usepackage{amsfonts,epsfig,color}
\makeatletter
\newcommand*{\centerfloat}{%
  \parindent \z@
  \leftskip \z@ \@plus 1fil \@minus \textwidth
  \rightskip\leftskip
  \parfillskip \z@skip}
\makeatother
\newcounter{ctr}
\theoremstyle{plain}
\newtheorem{theorem}{Theorem}[section]
\newtheorem{lemma}[theorem]{Lemma}
\newtheorem{corollary}[theorem]{Corollary}
\newtheorem{proposition}[theorem]{Proposition}
\newtheorem{conjecture}[theorem]{Conjecture}
\newtheorem{problem}[theorem]{Problem}

\theoremstyle{definition}
\newtheorem{definition}[theorem]{Definition}

\newtheorem{remark}[theorem]{Remark}
\newtheorem{example}[theorem]{Example}

\newcommand{\ignore}[1]{}

\newcommand{\G}{\ensuremath{\mathcal{G}}}

\newcommand{\QQ}{\ensuremath{\mathbb{Q}}}

\newcommand{\NN}{\ensuremath{\mathbb{N}}}

\newcommand{\sgn}{\text{\rm sgn}}

\newcommand{\U}{\ensuremath{\mathcal{U}}}

\newcommand{\ZZ}{\ensuremath{\mathbb{Z}}}

\newcommand{\be}{\begin{equation}}
\newcommand{\ee}{\end{equation}}

\renewcommand{\S}{\ensuremath{\mathcal{S}}}
\newcommand{\tsr}{\ensuremath{\otimes}}

\newcommand{\creading}{\text{\rm colword}}
\newcommand{\sqread}{\text{\rm sqread}}

\def\Tiny{\fontsize{6pt}{6pt}\selectfont}

\newcommand{\crc}[1]{\ensuremath{\overline{#1}}\vphantom{\underline{\overline{#1}}}}

\newcommand{\norotation}{\text{\hspace{0.5pt}\rotatebox{90}{$\!\varnothing$}\hspace{-1.3pt}}}

\newcommand{\Iplac}{\ensuremath{{{I_\text{\rm plac}}}}}
\newcommand{\Iplacperp}{\ensuremath{{{I_\text{\rm plac}^\perp}}}}

\newcommand{\Inplac}{\ensuremath{{{I_\text{\rm nplac}}}}}
\newcommand{\Inplacperp}{\ensuremath{{{I_\text{\rm nplac}^\perp}}}}
\newcommand{\Incox}{\ensuremath{{{I_\text{\rm nCox}}}}}
\newcommand{\Incoxperp}{\ensuremath{{{I_\text{\rm nCox}^\perp}}}}
\newcommand{\Iweakh}{\ensuremath{{{I_\text{\rm H}}}}}
\newcommand{\Iweakhperp}{\ensuremath{{{I_\text{\rm H}^\perp}}}}
\newcommand{\Icomm}{\ensuremath{{{I_{\text{\rm C}}}}}}
\newcommand{\Icommperp}{\ensuremath{{{I_{\text{\rm C}}^\perp}}}}

\newcommand{\Ifgp}{\ensuremath{{{I_{\text{\rm B}}}}}}
\newcommand{\Ifgpperp}{\ensuremath{{{I_{\text{\rm B}}^\perp}}}}
\newcommand{\Ifg}{\ensuremath{{{I_{\norotation}}}}}
\newcommand{\Ifgperp}{\ensuremath{{{I_{\norotation}^\perp}}}}
\newcommand{\Ilam}[1]{\ensuremath{{{I_{\text{\rm L}, #1}}}}}

\newcommand{\Ilamleperp}[1]{\ensuremath{{{I_{\text{\rm L}, \le #1}^\perp}}}}
\newcommand{\Jlam}[1]{\ensuremath{{{J_{\text{\rm L}, #1} }}}}
\newcommand{\Jlamperp}[1]{\ensuremath{{{J_{\text{\rm L}, #1}^\perp }}}}
\newcommand{\Iaba}[1]{\ensuremath{{{I_{\text{\rm R}, #1}}}}}
\newcommand{\Iabaperp}[1]{\ensuremath{{{I_{\text{\rm R}, #1}^\perp}}}}
\newcommand{\Ist}{\ensuremath{{{I_{\text{\rm st}}}}}}
\newcommand{\Irk}{\ensuremath{{{I_{\text{\rm S}}}}}}
\newcommand{\Irkperp}{\ensuremath{{{I_{\text{\rm S}}^\perp}}}}

\newcommand{\Iassaf}[1]{\ensuremath{{I_{\text{\rm A}, #1}}}}
\newcommand{\Iassafperp}[1]{\ensuremath{{I_{\text{\rm A}, #1}^\perp}}}

\newcommand{\Des}{\ensuremath{{\text{\rm Des}}}}
\newcommand{\invi}[1]{\ensuremath{\operatorname{inv}_{#1}'}}

\newcommand{\Wi}[1]{\ensuremath{{\text{\rm W$_{#1}'$}}}}

\newcommand{\qlam}{\ensuremath{{\hat{q}}}}

\newcommand{\e}{\mathsf}
\newcommand{\diagread}{\text{\rm diagread}}

\newcommand{\Dzero}{\operatorname{D}_0}

\newcommand{\SYT}{\text{\rm SYT}}

\newcommand{\RSST}{\text{\rm RSST}}

\newcommand{\SSYT}{\text{\rm SSYT}}

\newcommand{\Jnot}[1]{\ensuremath{{\hspace{1pt}\text{\rm:}~\raisebox{1pt}{$#1$}~\text{\rm:}\hspace{1pt}}}}

\newcommand*\encircle[1]{\tikz[baseline=(char.base)]{
    \node[shape=circle,draw,inner sep=.74pt] (char) {\ensuremath{#1}};}}

\newlength{\cellsize}
\newcommand\tableau[1]{
\vcenter{
\let\\=\cr
\baselineskip=-16000pt \lineskiplimit=16000pt \lineskip=0pt
\halign{&\tableaucell{##}\cr#1\crcr}}}

\newcommand{\tableaucell}[1]{{%
\def \arg{#1}\def \void{}%
\ifx \void \arg
\vbox to \cellsize{\vfil \hrule width \cellsize height 0pt}%
\else \unitlength=\cellsize
\begin{picture}(1,1)
\put(0,0){\makebox(1,1){$#1\vphantom{\crc{#1}}$}}
\put(0,0){\line(1,0){1}}
\put(0,1){\line(1,0){1}}
\put(0,0){\line(0,1){1}}
\put(1,0){\line(0,1){1}}
\end{picture}%
\fi}}

\newcommand\boldtableau[1]{
\vcenter{
\let\\=\cr
\baselineskip=-16000pt \lineskiplimit=16000pt \lineskip=0pt
\halign{&\boldtableaucell{##}\cr#1\crcr}}}

\newcommand{\boldtableaucell}[1]{{%
\def \arg{#1}\def \void{}%
\ifx \void \arg
\vbox to \cellsize{\vfil \hrule width \cellsize height 0pt}%
\else \unitlength=\cellsize
\begin{picture}(1,1)
\put(0,0){\makebox(1,1){$\mathbf{#1\vphantom{\crc{#1}}}$}}
\put(0,0){\line(1,0){1}}
\put(0,1){\line(1,0){1}}
\put(0,0){\line(0,1){1}}
\put(1,0){\line(0,1){1}}
\end{picture}%
\fi}}

\newcommand{\partition}[1]{{\setlength{\cellsize}{1ex}  \tableau{#1}}}
\title[Noncommutative Schur functions, switchboards, and Schur positivity]{Noncommutative Schur functions,\\[.05in] switchboards, and Schur positivity}

\keywords{Noncommutative Schur function, switchboard, D~graph, LLT polynomial, Schur positivity}

\begin{document}

\author{Jonah Blasiak}
\address{Department of Mathematics, Drexel University, Philadelphia, PA 19104}
\email{jblasiak@gmail.com}

\author{Sergey Fomin}
\address{Department of Mathematics, University of Michigan, Ann Arbor,
MI 48109}
\email{fomin@umich.edu}

\thanks{The authors were supported by NSF Grants DMS-14071174 (J.~B.)
and DMS-1361789 (S.~F.).}

\date{October~1, 2015}

\subjclass[2010]{Primary 05E05. 
Secondary 05A15, 
16S99. 
}

\begin{abstract}
We review and further develop a general approach to Schur positivity of
symmetric functions based on the machinery of noncommutative Schur
functions. This approach unifies ideas of Assaf \cite{SamiOct13, SamiForum},
Lam \cite{LamRibbon}, and Greene and the second author~\cite{FG}.
\end{abstract}%

\ \vspace{.2in}

\maketitle


\section{Introduction}
\label{s intro}

Schur functions form the most important basis in the ring of symmetric
functions,
and the problem of expanding various families of symmetric functions
with respect to this basis arises in many mathematical contexts.
One is particularly interested in obtaining manifestly positive
combinatorial descriptions for the coefficients in various Schur
expansions,
the celebrated Littlewood-Richardson Rule being the prototypical
example.

One powerful approach to this class of problems relies on the idea of
generalizing the notion
of a Schur function to the noncommutative case, or more precisely to
noncommutative rings whose generators (the ``noncommuting variables'')
satisfy some carefully chosen relations,
making it possible to retain many key
features of the classical theory of symmetric functions.
The first implementation of this idea goes back to the work of
Lascoux and Sch\"utzenberger in the 1970's
\cite{LS, Sch}. They showed that the plactic algebra
(i.e., the ring defined by the Knuth relations)
contains a subalgebra isomorphic to the ring of symmetric functions,
and moreover this subalgebra comes naturally equipped
with a basis of noncommutative versions of Schur functions.
In the 1990's, Greene and the second author~\cite{FG} generalized this
approach by replacing some of the plactic relations
by weaker four-term relations.
By studying noncommutative versions of Schur functions in the
algebra thus defined, they obtained, in a uniform fashion,
positive Schur expansions for a large class of
(ordinary) symmetric functions that includes the Stanley symmetric functions
and stable Grothendieck polynomials.
As explained below, the construction in~\cite{FG} can be further
generalized to include such important examples as
LLT and Macdonald polynomials.

\pagebreak[3]

LLT polynomials are a family of symmetric functions
introduced by
Lascoux, Leclerc, and Thibon~\cite{LLT}.
While LLT polynomials were shown to be Schur positive \cite{GH, LT00}
using Kazhdan-Lusztig theory, it remains a
fundamental open problem to find an explicit positive combinatorial
formula for the coefficients in their Schur expansions. \linebreak[3]
A~solution to this problem would also settle the analogous
problem for Macdonald polynomials, since
the Haglund-Haiman-Loehr
formula~\cite{HHL} expresses an arbitrary transformed Macdonald polynomial
as a positive sum of LLT polynomials.

We are optimistic that the noncommutative Schur function approach
can be successfully implemented
to obtain positive combinatorial formulas for the Schur
expansions of LLT and Macdonald polynomials.
The first step in this direction was made by Lam~\cite{LamRibbon}
who used a variation of this approach to produce
formulas---combinatorial but not manifestly positive---for the
coefficients appearing in the Schur expansions of LLT polynomials.
In a separate development, Assaf~\cite{SamiOct13}
introduced variants of
Knuth equivalence as a combinatorial tool to study the
Schur positivity phenomenon for LLT polynomials.

\medskip

This paper extends the setup of~\cite{FG}
to encompass Lam's work and give
an algebraic framework for Assaf's equivalences.
Specifically, we
\begin{itemize}[leftmargin=6.2mm,itemsep=3.5pt] 
\item
strengthen the main result of \cite{FG}
(Theorem~\ref{th:FG'-positivity};
proof in Section~\ref{sec:proof-of-FG'-positivity});
\item
introduce combinatorial gadgets called switchboards (Definition~\ref{def-switchboard}),
which were inspired by the
D~graphs of Assaf~\cite{SamiOct13};
\item
associate a symmetric function to each switchboard
(Definitions~\ref{def:FGamma}/\ref{def:cauchy-product}) and show that
the class of symmetric functions so defined
includes many important
families such as LLT polynomials and all examples from~\cite{FG}
(Sections~\ref{sec:switchboards}--\ref{s bijectivizations and
  examples});
\item
give a (counter)example of a switchboard whose symmetric function is
not Schur positive (Corollary \ref{cor-nopos});
\item
recall the main result of \cite{BLamLLT},
a positive combinatorial formula for the Schur expansion of an LLT
polynomial indexed by a $3$-tuple of skew shapes (Theorem~\ref{t
  sqread});
\item
investigate the phenomenon of monomial positivity of
noncommutative Schur functions
(see~\eqref{eq:ncschur-via-e})
in various rings
(Sections~\ref{s semimatched ideals}--\ref{s Schur positivity}).
\end{itemize}
We also obtain a ``tightness'' result (Theorem~\ref{t intro
  positivity}; proof in Section~\ref{s monomial positivity})
which states that in a given ring 
satisfying the basic requirements of our setup,
the noncommutative Schur functions
are monomial positive if and only
if  the symmetric functions naturally associated to this ring are Schur positive.
In other words,  the noncommutative Schur function
approach gives an equivalent reformulation of the Schur positivity problem.

Section~\ref{s main results} of the paper serves as its ``extended abstract.''

This paper can be viewed as a ``prequel'' to (and, partly, a review of)
the closely related papers~\cite{BLamLLT, BD0graph} by the
first author, and to his paper~\cite{BLKronecker} with R.~Liu.

\subsection*{Acknowledgments}
We thank Sami Assaf, Anna Blasiak, Thomas Lam, and Bernard Leclerc for helpful discussions,
and Elaine So and Xun Zhu for help typing and typesetting figures.
This project began while the first author was a postdoc at the
University of Michigan.
He is grateful to John Stembridge for his generous~advice and many detailed discussions.


\section{Noncommutative Schur functions and Schur positivity}
\label{s main results}

Let $\U= \ZZ\langle u_1,u_2,\dots,u_N \rangle$ be the free associative ring
 with generators $u_1, u_2, \dots, u_N$, which we regard as ``noncommuting variables.''
For $S\subset \{1,\dots,N\}$ and $k\in\mathbb{Z}$, we define
the \emph{noncommutative elementary symmetric function} $e_k(\mathbf{u}_S)$ by
\begin{align}\label{e ek def}
e_k(\mathbf{u}_S)=\sum_{\substack{i_1 > i_2 > \cdots > i_k \\ i_1,\dots,i_k \in S}}u_{i_1}u_{i_2}\cdots u_{i_k}\,;
\end{align}
by convention,
$e_0(\mathbf{u}_S)=1$ and $e_k(\mathbf{u}_S) = 0$ for $k<0$ or $k>|S|$
(cardinality of~$S$).
For $S=\{1,\dots,N\}$, we use the notation $e_k(\mathbf{u})$.

Throughout this paper an ideal will always mean a two-sided ideal.
\begin{theorem} 
\label{th:quad-Icomm}
For an ideal $I$ of\,~$\U$, the following are equivalent:
\begin{itemize}
\item the ideal $I$ includes the elements
\begin{alignat}{3}
&u_b^2 u_a + u_a u_b u_a - u_b u_a u_b - u_b u_a^2 & (a <
  b),\quad\ \ \,\label{quad uu 2vars} \\
&u_b u_c u_a + u_a u_c u_b - u_b u_a u_c - u_c u_a u_b &  \qquad(a < b < c), \label{quad uuu 3vars}\\
&u_c u_b u_c u_a + u_b u_c u_a u_c - u_c u_b u_a u_c - u_b u_c^2 u_a
  & \qquad(a < b < c); \label{quad uuuu 3vars}
\end{alignat}
\item
the noncommutative elementary symmetric functions $e_k(\mathbf{u}_S)$ and
$e_\ell(\mathbf{u}_S)$ commute with each other modulo~$I$, for any $k,\ell$ and
any~$S$:
\begin{equation}
\label{eq:e's-commute}
e_k(\mathbf{u}_S)\,e_\ell(\mathbf{u}_S)\equiv e_\ell(\mathbf{u}_S)\,e_k(\mathbf{u}_S) \bmod I.
\end{equation}
\end{itemize}
\end{theorem}

\begin{remark}
Theorem~\ref{th:quad-Icomm} is equivalent to a result by
A.~N.~Kirillov~\cite[Theorem~2.24]{Kirillov-Notes}.
It generalizes similar results obtained in \cite{BD0graph, FG, LamRibbon, NS}.
Specifically, \cite[Lemma~3.1]{FG},
\hbox{\cite[Theorem~3]{LamRibbon}},
\cite[Lemma~2.2]{BD0graph},
and \cite[Theorem~2]{NS}
assert, in the notation to be introduced below,
that the  $e_k(\mathbf{u})$ commute modulo the ideals $\Ifgp$,
$\Jlam{k}$, $\Irk+\Ist$, and the triples ideal with all rotation triples, respectively; cf.\ Definitions \ref{d Ifg'},
\ref{d Lam algebra}, \ref{d Irk}/\ref{def:Ist}, and \ref{d triple biject}.
\end{remark}

Theorem~\ref{th:quad-Icomm} has the following elegant reformulation.

\begin{theorem}
\label{t intro ed commute}
The commutation relations \eqref{eq:e's-commute}
hold for all
$k,\ell$, and $S \subset \{1,\dots N\}$ provided
they hold for $k=1$, $\ell\in\{2,3\}$, and $|S|\le 3$.
\end{theorem}

It is quite miraculous that the special cases $|S|\le 3$ of the
commutation relations~\eqref{eq:e's-commute}
(i.e., the cases involving two or three variables only)
are sufficient to get all of them.
An~additional miracle is that the relation
$e_2(\mathbf{u}_S)e_3(\mathbf{u}_S) \equiv e_3(\mathbf{u}_S)e_2(\mathbf{u}_S)$
for $|S|=3$ is not needed: all instances of \eqref{eq:e's-commute},
including the latter, are a consequence of
$e_1(\mathbf{u}_S)e_2(\mathbf{u}_S) \equiv
e_2(\mathbf{u}_S)e_1(\mathbf{u}_S)$ and
$e_1(\mathbf{u}_S)e_3(\mathbf{u}_S) \equiv
e_3(\mathbf{u}_S)e_1(\mathbf{u}_S)$ for
$|S|\le 3$.

We denote by $\Icomm$ the ideal in~$\U$ generated by the elements
\eqref{quad uu 2vars}--\eqref{quad uuuu 3vars} in
Theorem~\ref{th:quad-Icomm}.
All ideals~$I\subset\U$ considered in this paper contain~$\Icomm$.

Let $x$ be a formal variable that commutes with $u_1,\dots,u_N$.
The elements $e_k(\mathbf{u})$ can be packaged into the generating function
\begin{equation}
\label{eq:E(x)}
E(x)=\sum_{k = 0}^N x^k e_k(\mathbf{u})
=(1+x u_N)\cdots (1+x u_1)\in\U[x].
\end{equation}
The condition that the $e_k(\mathbf{u})$
commute pairwise modulo~$\Icomm$ is equivalent to the generating function congruence 
$E(x)E(y)\equiv E(y)E(x)$ modulo~$\Icomm$,
or more precisely modulo the corresponding ideal $\Icomm[x,y]\subset\U[x,y]$.
Here $y$ is another formal variable that commutes with~$x,
u_1,\dots,u_N$.

The \emph{noncommutative complete homogeneous
  symmetric functions} are defined by
\[
h_\ell(\mathbf{u})=\sum_{1 \le i_1 \le i_2 \le \cdots \le i_\ell \le N}u_{i_1}u_{i_2}\cdots u_{i_\ell},
\]
for $\ell>0$; by convention, $h_0(\mathbf{u})=1$
and $h_{\ell}(\mathbf{u}) = 0$ for $\ell<0$.
Similarly to~\eqref{eq:E(x)}, we set
\begin{equation}
\label{eq:H(x)}
H(x)=\sum_{\ell = 0}^\infty x^\ell h_\ell(\mathbf{u})
=(1-x u_1)^{-1}\cdots (1-x u_N)^{-1}\in\U[[x]].
\end{equation}

\begin{corollary}
\label{cor:h's-commute}
Let $I\subset\U$ be an ideal containing~$\Icomm$.
Then the elements $h_1(\mathbf{u}), h_2(\mathbf{u}),\dots$
commute pairwise modulo~$I$.
Consequently,
\[
H(x)H(y)\equiv H(y) H(x) \, \bmod I[[x,y]]
\]
where $x$ and $y$ are formal variables commuting with each other and with $u_1,\dots,u_N$.
\end{corollary}

\begin{proof}
The identity $H(x)E(-x)=1$ implies that for any~$m>0$,
\begin{align}\label{e E(-x)H(x)}
h_m(\mathbf{u})-h_{m-1}(\mathbf{u})\,e_1(\mathbf{u})+h_{m-2}(\mathbf{u})\,e_2(\mathbf{u})-\cdots
+(-1)^m e_m(\mathbf{u})=0.
\end{align}
Thus one can recursively express $h_1(\mathbf{u}), h_2(\mathbf{u}), \dots$ in terms
of $e_1(\mathbf{u}), e_2(\mathbf{u}), \dots$, 
and the claim follows from Theorem~\ref{th:quad-Icomm}.
\end{proof}

We denote by $\U^*$ the $\mathbb{Z}$-module freely spanned by the set of words
in the alphabet $\{1,\dots,N\}$.
There is a natural pairing $\langle\cdot ,\cdot  \rangle$
between the spaces $\U$ and~$\U^*$ in which
the basis of noncommutative monomials is dual to the basis of words.
More precisely, if $\mathbf{u}_{\e{w}}=u_{\e{w}_1}\cdots u_{\e{w}_n}\in\U$
is a monomial corresponding to the word
$\e{w} = \e{w}_1 \cdots \e{w}_n\in\U^*$,
and $\e{v}\in\U^*$ is another word, then
$\langle \mathbf{u}_{\e{w}},\e{v}\rangle =\delta_{\e{vw}}$. 

For an ideal  $I\subset\U$, we denote by $I^\perp$ the orthogonal complement
\[
I^\perp = \big\{ \gamma \in \U^* \mid \langle z, \gamma \rangle = 0 \text{ for all }z
\in I \big\}.
\]
Informally speaking, pairing with the elements of $I^\perp$ is tantamount to
working modulo~$I$:
for $f,g\in\U$, the congruence $f\equiv g \bmod I$ implies
$\langle f,\gamma \rangle=\langle g,\gamma \rangle$
for all $\gamma\in I^\perp$.
Put another way,
any element of $\U/I$ has a well-defined pairing with any
element of~$I^\perp$.


\pagebreak[3]

As shown in~\cite{FG}, noncommutative symmetric functions
$e_k(\mathbf{u})$ and $h_\ell(\mathbf{u})$ can be
used to define and study a large class of symmetric functions in the usual
sense, i.e., formal power series (in \emph{commuting} variables)
which are symmetric under permutations of these variables.
This is done as follows.

\begin{definition}
\label{def:cauchy-product}
Let $x_1,x_2,\dots$ be formal variables which commute with each other
and with $u_1,\dots,u_N$.
Define the \emph{noncommutative Cauchy product} $\Omega(\mathbf{x},
\mathbf{u})\in\U[[x_1,x_2,\dots]]$~by
\begin{equation}
\label{eq:cauchy-product}
\Omega(\mathbf{x}, \mathbf{u})
= H(x_1) H(x_2) \cdots
=\prod_{j=1}^\infty\prod_{i=1}^{N}(1-x_ju_i)^{-1}
=\prod_{j=1}^\infty \sum_{\ell = 0}^\infty x_j^\ell h_\ell(\mathbf{u})
\end{equation}
(cf.~\eqref{eq:H(x)}); here the notation $\prod$ always means that the factors are
multiplied left-to-right.
Now,~for $\gamma\in\U^*$, we define the formal
power series $F_\gamma(\mathbf{x})\in \mathbb{Z}[[x_1,x_2,\dots]]$ by
\begin{equation}
\label{eq:F-gamma}
F_\gamma(\mathbf{x})=\big\langle \Omega(\mathbf{x}, \mathbf{u}) ,\gamma\big\rangle,
\end{equation}
where  $\langle\cdot ,\cdot  \rangle$ is the extension of the pairing between~$\U$ and~$\U^*$ to
the bilinear pairing between $\U[[x_1,x_2,\dots]]$ and~$\U^*$ in which $\langle f \, \mathbf{u}_{\e{w}},\e{v}\rangle = f \, \delta_{\e{vw}}$ for
any  $f \in \ZZ[[x_1,x_2,\dots]]$.
\end{definition}

An alternative definition of~$F_\gamma(\mathbf{x})$ can be given in
terms of \emph{quasisymmetric functions}.
Collecting the terms in~\eqref{eq:cauchy-product} involving each noncommutative monomial~$\mathbf{u}_{\e{w}}$, we
obtain 
\[
\Omega(\mathbf{x}, \mathbf{u})=\sum_{\e{w}} Q_{\Des(\e{w})}(\mathbf{x}) \mathbf{u}_{\e{w}}\,,
\]
where $\Des(\e{w})= \{i \in \{1,\dots,n-1\} \mid \e{w}_i > \e{w}_{i+1} \}$ denotes  the \emph{descent set} of
a word $\e{w} = \e{w}_1 \cdots \e{w}_n$, and
\[
Q_{\Des(\e{w})}(\mathbf{x}) = \sum_{\substack{1 \le i_1 \le \,
    \cdots \, \le i_n\\j \in \Des(\e{w}) \implies i_j < i_{j+1} }}
x_{i_1}\cdots x_{i_n}
\]
is the corresponding \emph{fundamental quasisymmetric function} (I. Gessel \cite{GesselPPartition}).
Consequently, for a vector $\gamma=\displaystyle\sum_{\e{w}} \gamma_{\e{w}} \,\e{w}\in\U^*$, we have
\begin{equation}
\label{eq:Fgamma-via-Q}
F_\gamma(\mathbf{x})=\displaystyle\sum_{\e{w}} \gamma_{\e{w}}\,
Q_{\Des(\e{w})}(\mathbf{x}).
\end{equation}
Thus the map $\gamma\mapsto F_\gamma(\mathbf{x})$
is simply the linear map $\U^*\to \ZZ[[x_1, x_2, \dots]]$
that sends each word~$\e{w}$ to the fundamental
quasisymmetric function associated with the descent set of~$\e{w}$.

\begin{proposition}
\label{Icomm-symmetric}
If $\gamma\in \Icommperp$, then $F_\gamma(\mathbf{x})$
is symmetric in $x_1, x_2,\dots$
\end{proposition}

\begin{proof}
Let $\Omega_j(\mathbf{x}, \mathbf{u})=\cdots H(x_{j-1}) H(x_{j+1}) H(x_j)H(x_{j+2})\cdots$
denote the result of switching $x_j$ and $x_{j+1}$
in~$\Omega(\mathbf{x}, \mathbf{u})$.
By Corollary~\ref{cor:h's-commute},
$\Omega_j(\mathbf{x}, \mathbf{u})\equiv
\Omega(\mathbf{x}, \mathbf{u}) \bmod \Icomm[[x_1,x_2,\dots]]$.
Hence
$\langle \Omega_j(\mathbf{x}, \mathbf{u}),\gamma\rangle
=\langle\Omega(\mathbf{x}, \mathbf{u}),\gamma\rangle$, and the
claim follows.
%
\end{proof}

As explained in  Sections~\ref{s LLT} and~\ref{s bijectivizations and
  examples},
the class of symmetric functions $F_\gamma$ includes several important families
such as Stanley symmetric functions and LLT polynomials.
The following construction provides a powerful tool for studying
the coefficients of the expansions of the symmetric functions~$F_\gamma$
in the basis of Schur functions.

\pagebreak[3]

\begin{definition}[\emph{Noncommutative Schur functions}]
\label{d noncommutative Schur functions}
Let $\lambda=(\lambda_1,\lambda_2,\dots)$ be an integer partition.
Let $\lambda'$ be the conjugate partition,
and let  $t=\lambda_1$ be the  number of parts of~$\lambda'$.
The \emph{noncommutative Schur function}
$\mathfrak{J}_\lambda(\mathbf{u}) \in \U$ is given by the following
noncommutative version of the classical Kostka-Naegelsbach determinantal formula  $s_\lambda = \det(e_{\lambda'_{i}+j-i})$\,:
\begin{equation}
\label{eq:ncschur-via-e}
\mathfrak{J}_\lambda(\mathbf{u}) = \sum_{\pi\in \S_{t}}
\sgn(\pi) \, e_{\lambda'_1+\pi(1)-1}(\mathbf{u})\,
e_{\lambda'_2+\pi(2)-2}(\mathbf{u}) \cdots
e_{\lambda'_{t}+\pi(t)-t}(\mathbf{u}).
\end{equation}
\end{definition}

\begin{remark}
Alternatively, 
$\mathfrak{J}_\lambda(\mathbf{u})$  can be defined by the
noncommutative version of the Jacobi-Trudi formula:
\begin{equation}
\label{eq:ncschur-via-h}
\mathfrak{J}_\lambda(\mathbf{u}) = \sum_{\pi\in \S_{\ell}}
\sgn(\pi) \, h_{\lambda_1+\pi(1)-1}(\mathbf{u})
h_{\lambda_2+\pi(2)-2}(\mathbf{u}) \cdots
h_{\lambda_{\ell}+\pi(\ell)-\ell}(\mathbf{u});
\end{equation}
here $\ell=\lambda_1'$ is the number of parts of~$\lambda$.
The fact that the right-hand sides of \eqref{eq:ncschur-via-e} and~\eqref{eq:ncschur-via-h}
are congruent to each other modulo~$\Icomm$
can be established by the argument given in the proof of
Proposition~\ref{Icomm-Schur} below.

The closely related noncommutative Schur functions~$s_\lambda(\mathbf{u})$ introduced and
studied in~\cite{FG} were defined there
in terms of
semistandard Young tableaux, cf.\ Theorem~\ref{t FG}.
While 
$s_\lambda(\mathbf{u})\equiv \mathfrak{J}_\lambda(\mathbf{u})$
modulo the ideals considered in~\cite{FG}, the two definitions
may diverge for other ideals containing~$\Icomm$.
\end{remark}

The next result is a version of a simple, yet important, observation made in~\cite{FG}.

\begin{proposition}
\label{Icomm-Schur}
For any $\gamma\in \Icommperp$, we have
\begin{equation}
\label{eq:Delta-f-via-schurs}
F_\gamma(\mathbf{x}) 
= \sum_{\lambda} s_\lambda(\mathbf x) \big\langle \mathfrak
J_{\lambda}(\mathbf{u}) , \gamma \big\rangle.
\end{equation}
\end{proposition}

\begin{proof} 
Comparing \eqref{eq:Delta-f-via-schurs} to~\eqref{eq:F-gamma}, we see
that it suffices to show
\begin{equation}
\label{eq:f-gamma-schur-claim}
\Omega(\mathbf{x}, \mathbf{u})\equiv
\sum_{\lambda} s_\lambda(\mathbf x) \mathfrak
J_{\lambda}(\mathbf{u})\bmod \Icomm[[\mathbf{x}]].
\end{equation}
Let $\Lambda(\mathbf{y})$ denote the ordinary ring of symmetric polynomials in
$N$ commuting variables $\mathbf{y}=(y_1,\dots,y_N)$.
This is a polynomial ring with algebraically independent generators
$e_1(\mathbf{y}),\dots,e_N(\mathbf{y})$.
The ring homomorphism $\psi: \Lambda(\mathbf{y})\to \U/\Icomm$
defined by $e_k(\mathbf{y})\mapsto e_k(\mathbf{u})$ sends each Schur
polynomial~$s_\lambda(\mathbf{y})$ to
$\mathfrak{J}_\lambda(\mathbf{u})$ and
sends $h_m(\mathbf{y})$ to $h_m(\mathbf{u})$. To see the latter,
apply  $\psi$ to the classical identity  $\sum_{\ell+k=m} (-1)^k h_\ell(\mathbf{y}) e_k(\mathbf{y})=0$ and then subtract
\eqref{e E(-x)H(x)} to obtain
$\sum_{\ell+k=m} (-1)^k \left( \, \psi(h_\ell(\mathbf{y}))- h_\ell(\mathbf{u}) \, \right) e_k(\mathbf{u})=0$;
the desired fact $\psi(h_m(\mathbf{y})) = h_m(\mathbf{u})$ then follows by induction on  $m$.
Now applying $\psi$ to the classical Cauchy identity
\[
\prod_{j=1}^\infty \sum_{\ell = 0}^\infty x_j^\ell \,h_\ell(\mathbf{y})
=\sum_{\lambda} s_\lambda(\mathbf x) s_{\lambda}(\mathbf{y})
\]
yields~\eqref{eq:f-gamma-schur-claim}.
\end{proof}

As mentioned earlier, many important families of symmetric functions
arise from the above construction.
One begins by choosing a particular ideal~$I\subset\U$
which contains $\Icomm$ and is homogeneous with respect to the grading
$\operatorname{deg}(u_i)=1$.
Typically, one is interested in the subclass of symmetric functions
$F_\gamma(\mathbf{x})$ labeled by the elements
$\gamma\in I^\perp\subset \Icommperp$ which are
\emph{nonnegative} integer combinations of words.
The goal is to show that any such symmetric function $F_\gamma(\mathbf{x})$
is \emph{Schur positive}, i.e., its expansion in the basis of Schur functions
has nonnegative coefficients.
An even more ambitious goal is to obtain a manifestly positive combinatorial rule for these
coefficients.

\pagebreak[3]

Let $\U_{\ge 0}\subset\U$ (resp.,~$\U^*_{\ge 0}\subset\U^*$) denote
the $\ZZ_{\ge 0}$-cone (that is, the additive monoid)
generated by the noncommutative monomials in~$\U$ (resp., words in~$\U^*$).
To rephrase, these cones consist of all nonnegative integer
combinations of monomials (resp., words).
The symmetric functions $F_\gamma$ of interest to us are labeled by
$\gamma\in \U^*_{\ge 0}\cap I^\perp$.

An element $f\in\U$ is called \emph{$\ZZ$-monomial positive}
  modulo an ideal~$I$ if $f\in \U_{\ge0}+I$, i.e.,
if $f$ can be written, modulo~$I$, as a nonnegative integer
combination of noncommutative monomials.
We say that $f\in\U$ is  \emph{$\QQ$-monomial positive} modulo~$I$
if some positive integer multiple of $f$ is $\ZZ$-monomial positive
  modulo~$I$. 

\pagebreak[3]

%

\begin{theorem}
\label{t intro positivity}
For a homogeneous ideal $I\subset\U$ containing~$\Icomm$,
the following are equivalent:
\begin{itemize}
\item[\emph{(i)}]
all symmetric functions $F_\gamma(\mathbf{x})$, for $\gamma \in
\U^*_{\ge 0}\cap I^\perp$, are Schur positive;
\item[\emph{(ii)}]
all noncommutative Schur functions
$\mathfrak{J}_\lambda(\mathbf{u})$ are $\QQ$-monomial positive
modulo~$I$.
\end{itemize}
In fact, a stronger statement holds: given $I\supset\Icomm$ and an integer
partition~$\lambda$,
the  coefficient of $s_\lambda(\mathbf{x})$ in $F_\gamma(\mathbf{x})$
is nonnegative for all $\gamma \in \U^*_{\ge 0}\cap I^\perp$
if and only if $\mathfrak{J}_\lambda(\mathbf{u})$ is $\QQ$-monomial
positive modulo~$I$.
\end{theorem}

Theorem \ref{t intro positivity} can be viewed as a simplified version of
\cite[Theorem~1.4]{BD0graph}, which gave a slightly stronger
conclusion in a more restricted~setting.

\begin{proof}[Proof of the implication {\rm (ii)}$\Rightarrow${\rm
      (i)}]
By assumption~(ii), there is a monomial expansion
\begin{equation}
\label{eq:cJ=...}
c_\lambda\,\mathfrak{J}_\lambda(\mathbf{u})\equiv
\sum_{\e{w}
} a_{\lambda,\e{w}}\, \mathbf{u}_{\e{w}}
\,(\bmod \, I),
\end{equation}
for some positive integer~$c_\lambda$ and nonnegative integer coefficients~$a_{\lambda,\e{w}}$.
Let
\begin{equation}
\gamma=\sum_{\e{w}} \gamma_{\e{w}}\, \e{w} \in I^\perp,
\end{equation}
where
all coefficients $\gamma_{\e{w}}$ are nonnegative integers.
It follows that
\begin{equation}
\label{eq:cJ,gamma=...}
c_\lambda\,\big\langle \mathfrak J_{\lambda}(\mathbf{u}) , \gamma \big\rangle
=\Big\langle \,\sum_{\e{w}} a_{\lambda,\e{w}}\, \mathbf{u}_{\e{w}}
, \gamma \,\Big\rangle
=\sum_{\e{w}} a_{\lambda,\e{w}}\, \gamma_{\e{w}}
\ge 0.
\end{equation}
It remains to recall that by Proposition~\ref{Icomm-Schur},
the coefficient of~$s_\lambda(\mathbf{x})$ in the Schur expansion
of~$F_\gamma(\mathbf{x})$ is equal to $\langle \mathfrak J_{\lambda}(\mathbf{u}) , \gamma \rangle$.
\end{proof}

The converse implication {\rm (i)}$\Rightarrow${\rm (ii)} is proved
in Section~\ref{s monomial positivity}
using Farkas' Lemma from convex analysis.
%
Informally,
this implication shows that the noncommutative Schur function
approach to proving Schur positivity is a powerful one because it only
fails if the symmetric functions to which it is applied are not
actually all Schur positive.

\begin{remark}
\label{rem:manifestly-positive-formulas}
We note that $\QQ$-monomial positivity \eqref{eq:cJ=...} of
noncommutative Schur functions
$\mathfrak{J}_\lambda(\mathbf{u})$
modulo a given ideal~$I$ implies more than ``just'' Schur positivity
of the corresponding class of symmetric functions~$F_\gamma(\mathbf{x})$.
It yields a \emph{manifestly positive} formula for the Schur
expansion of~$F_\gamma(\mathbf{x})$,
cf.\ \eqref{eq:Delta-f-via-schurs} and~\eqref{eq:cJ,gamma=...}:
\begin{equation}
\label{eq:F-gamma=c^{-1}...}
F_\gamma(\mathbf{x})
=\sum_\lambda s_\lambda(\mathbf{x}) \,c_\lambda^{-1} \sum_{\e{w}}
a_{\lambda,\e{w}}\, \gamma_{\e{w}}\,.
\end{equation}
%
If moreover all 
$\mathfrak{J}_\lambda(\mathbf{u})$ are $\ZZ$-monomial positive
modulo~$I$, so we can take $c_\lambda=1$,
then~\eqref{eq:F-gamma=c^{-1}...} becomes a positive
\emph{combinatorial} rule for the Schur expansions of~$F_\gamma$.
(We imagine that a monomial positive expansion for
$\mathfrak{J}_\lambda(\mathbf{u})$ is established by an explicit
formula in~which the numbers $a_{\lambda,\e{w}}$ count combinatorial objects of some kind.)
We suspect that it can happen that
$\mathfrak{J}_\lambda(\mathbf{u})$ is  $\QQ$-monomial positive
modulo~$I$, but not  $\ZZ$-monomial positive
modulo~$I$.
\end{remark}

\pagebreak[3]

Unfortunately, and contrary to what extensive computational
experiments initially  led us to believe,
not all symmetric functions~$F_\gamma(\mathbf{x})$, for $\gamma\in  \U^*_{\ge 0}\cap \Icommperp$,
are Schur~positive.
(Equivalently, $\QQ$-monomial positivity of
$\mathfrak{J}_{\lambda}(\mathbf{u})$ may fail modulo~$\Icomm$.)
We~present a counterexample in Section~\ref{sec:switchboards},
cf.\ Corollary~\ref{cor-nopos}.
%
This negative result must not however discourage us from pursuing the current
line of inquiry:
even though Schur positivity may fail in the most general case
$I=\Icomm$,
it still holds for many important ideals $I\supset\Icomm$, \linebreak[3]
or equivalently for the nonnegative vectors $\gamma\in I^\perp\subset\Icommperp$.

In each application,
the challenge is to find a  monomial positive expression for
$\mathfrak{J}_{\lambda}(\mathbf{u})$ that holds modulo an
appropriately chosen ideal  $I$.
In doing so, the goal is to choose an ideal $I$ which is ``exactly the right size.''
First of all,  $I$ must be small enough so that the symmetric
functions of interest to us are of the form
 $F_\gamma$ for  $\gamma\in I^\perp$.
On the other hand, if  $I$ is too small, e.g.,  $I = \Icomm$, then
$\mathfrak{J}_{\lambda}(\mathbf{u})$ is not monomial positive modulo
$I$.
But there is yet the further subtle consideration:
if  $I$ is too large, then there are many ways to
write $\mathfrak{J}_\lambda(\mathbf{u})$ as a positive sum of
monomials (modulo~$I$),
making it extremely hard to single out a ``canonical one'' from this multitude;
this is what essentially happens in Lam's approach~\cite{LamRibbon}
to LLT polynomials.
In~conclusion, identifying the ``smallest'' ideal $I$ such
that~$\mathfrak{J}_\lambda(\mathbf{u})$ is monomial positive
modulo~$I$ is likely to be of help in finding (and proving)
the desired combinatorial rule for such monomial positive expression,
and would establish Schur positivity for a wider class of symmetric
functions.


Matters are complicated by the fact that the subset of
ideals~$I\supset \Icomm$ for which monomial positivity of
$\mathfrak{J}_\lambda(\mathbf{u})$ holds modulo~$I$ does not have a unique
minimal element, see Corollary~\ref{c no unique smallest}.
Hence it might be too optimistic to hope for an all-encompassing ``master
Schur positivity theorem,'' as different symmetric
functions~$F_\gamma$ may be Schur positive for different reasons.


On the bright side, monomial positivity of
$\mathfrak{J}_\lambda(\mathbf{u})$ can be established for many
important ideals~$I\supset\Icomm$, as demonstrated in~\cite{FG}.
(The first example of this phenomenon, namely the case of the plactic
algebra, goes back to Lascoux and Sch\"utzenberger~\cite{LS, Sch}.)
Our next theorem can be viewed as a strengthening of the main result of~\cite{FG}.

\begin{definition}
\label{d Ifg'}
Let $\Ifgp$ denote the ideal in~$\U$ generated by the elements
\begin{alignat}{2}
&u_b^2 u_a + u_a u_b u_a - u_b u_a u_b - u_b u_a^2 \qquad && (a <
  b), \label{quad uu 2vars again}\\
&u_b u_c u_a - u_b u_a u_c \qquad&& (a < b < c), \label{i knuth1}\\
&u_a u_c u_b - u_c u_a u_b \qquad&& (a < b < c). \label{i knuth2}
\end{alignat}
\end{definition}

\begin{lemma}
$\Ifgp\supset\Icomm$.
\end{lemma}

\begin{proof}
Note that \eqref{quad uu 2vars again} coincides with~\eqref{quad uu 2vars},
while the sum of \eqref{i knuth1} and~\eqref{i knuth2} is~\eqref{quad uuu 3vars}.
Finally, we get the elements~\eqref{quad uuuu 3vars} as follows:
$u_c u_b u_c u_a + u_b u_c u_a u_c - u_c u_b u_a u_c - u_b u_c^2 u_a
=
u_c(u_b u_c u_a-u_b u_a u_c)
-u_b(u_c^2 u_a + u_au_cu_a -u_c u_a u_c - u_c u_a^2) +(u_b u_au_c -
u_bu_c u_a)u_a\in\Ifgp$.
\end{proof}


The ideal $\Ifgp$ is the enlargement of~$\Icomm$ obtained
by replacing the generators \eqref{quad uuu 3vars} by
the ``3-letter Knuth elements'' \eqref{i knuth1}--\eqref{i
  knuth2}.
Curiously, this makes the degree~$4$ generators~\eqref{quad uuuu 3vars} superfluous, as the above calculation shows.

\pagebreak[3]

\begin{theorem}
\label{th:FG'-positivity}
The noncommutative Schur functions
$\mathfrak{J}_\lambda(\mathbf{u})$ are
\hbox{$\ZZ$-monomial} positive modulo~$\Ifgp$.
Consequently, the symmetric functions $F_\gamma(\mathbf{x})$
labeled by $\gamma\in  \U^*_{\ge 0}\cap \Ifgpperp$ are Schur positive.
\end{theorem}

Theorem~\ref{th:FG'-positivity} generalizes \cite[Lemma~3.2]{FG}, which used a
larger ideal~$\Ifg$ \linebreak[3]
obtained by adding ``Knuth elements''
involving two non-adjacent indices,
cf.~\eqref{eq:Ifg}.

Theorem~\ref{th:FG'-positivity} is proved in
Section~\ref{sec:proof-of-FG'-positivity}.
The $\ZZ$-monomial positive formula that we give for
$\mathfrak{J}_\lambda(\mathbf{u})$ (modulo~$\Ifgp$)
leads to a manifestly positive combinatorial formula for
the Schur expansions of the symmetric functions
$F_\gamma(\mathbf{x})$ labeled by $\gamma\in  \U^*_{\ge 0}\cap \Ifgpperp$,
extending the ``generalized Littlewood-Richardson rule''
of~\cite[Theorem~1.2]{FG} for $\gamma\in  \U^*_{\ge 0}\cap
\Ifgperp$.
As shown in~\cite{FG}, the latter subclass of $F_\gamma$'s
includes Stanley symmetric functions and, more generally,
homogeneous components of
stable Grothendieck polynomials.

\begin{definition}
\label{def:I_k}
For $k$ a positive integer, let $\Iaba{k}$ denote
the ideal in~$\U$ generated~by 
\begin{alignat}{3}
&u_a^2 &&\text{for all $a$,} \label{eq:aa} \\
&u_a u_b u_a &&\text{for all $a$ and $b$,} \label{eq:aba} \\
&u_bu_au_c - u_bu_cu_a \qquad &&\text{for $c-a > k$ and $a<b<c$,} \label{e Iassaf1} \\
&u_au_cu_b - u_cu_au_b \qquad &&\text{for $c-a > k$ and $a<b<c$,} \label{e Iassaf2}\\
&u_bu_au_c - u_au_cu_b \qquad &&\text{for $c-a \leq k$ and $a<b<c$,} \label{e Iassaf3}\\
&u_bu_cu_a - u_cu_au_b \qquad &&\text{for $c-a \leq k$ and $a<b<c$.} \label{e Iassaf4}
\end{alignat}
\end{definition}

It is straightforward to verify that
$\Iaba{k}\supset\Icomm$.

In Section~\ref{s LLT}, by recasting the work of Assaf~\cite{SamiOct13, SamiForum}
through the perspective outlined in this paper,
we explain that the class of symmetric
functions~$F_\gamma$ labeled by nonnegative vectors $\gamma\in \Iabaperp{k}$
includes the LLT polynomials of Lascoux, Leclerc, and
Thibon~\cite{LLT}; more precisely, it includes the coefficients of
powers of $q$ in LLT polynomials indexed by an arbitrary $k$-tuple of skew shapes.
Furthermore this class of symmetric functions~$F_\gamma$ includes the
(coefficients of  $q^it^j$ in) transformed
Macdonald polynomials (cf.~\cite{GarsiaHaiman}, \cite[Definition~3.5.2]{H3});
this is because Haglund, Haiman, and Loehr~\cite{HHL} showed that
any transformed Macdonald polynomial is a nonnegative integer
combination of LLT polynomials.
(Recall that the map $\gamma\mapsto F_\gamma$ is linear.)

The following statement is inspired by (and is implicit in) the work of
Assaf.

\begin{conjecture}
\label{conj:monomial-positivity-modIk}
For any integer partition~$\lambda$ and any positive integer~$k$,
the noncommutative Schur function
$\mathfrak{J}_\lambda(\mathbf{u})$ is
\hbox{$\ZZ$-monomial} positive modulo~$\Iaba{k}$.
Consequently, the symmetric functions $F_\gamma(\mathbf{x})$
labeled by $\gamma\in  \U^*_{\ge 0}\cap \Iaba{k}$ are Schur positive.
\end{conjecture}

As explained above, the statement in Conjecture~\ref{conj:monomial-positivity-modIk}
is stronger than Schur positivity of LLT and Macdonald polynomials.
More importantly (cf.\ Remark~\ref{rem:manifestly-positive-formulas}),
an explicit formula expressing $\mathfrak{J}_\lambda(\mathbf{u})$ as a
sum of noncommutative monomials modulo~$\Iaba{k}$
would immediately yield a Littlewood-Richardson-type rule for LLT
and Macdonald polynomials.

Thus far, the main new result arising from the noncommutative Schur
function approach to LLT positivity
is an explicit $\ZZ$-monomial positive expression for
$\mathfrak{J}_\lambda(\mathbf{u})$ modulo $\Ilam{3}$
which was obtained in~\cite{BLamLLT}.
(Here $\Ilam{k}\supset \Iaba{k}$ is the ideal associated with
Lam's algebra of ribbon Schur operators, cf.\ \cite{LamRibbon} and Definition~\ref{d Lam algebra}.)
This result yields a positive combinatorial formula for the Schur
expansion of LLT polynomials indexed by 3-tuples of skew shapes and
for the Schur expansion of
transformed Macdonald polynomials indexed by shapes with 3 columns;
see Theorem \ref{t sqread} and \eqref{e qLRcoefs}.

As the above discussions indicate,
our efforts to create a general framework for the study
of Schur positivity of symmetric functions~$F_\gamma$
have been driven by potential applications.
Since these symmetric functions are labeled by
nonnegative vectors~$\gamma$
chosen from the orthogonal complement~$I^\perp$
of some ideal~$I\supset\Icomm$,
our goal is to establish Schur positivity for $I^\perp$ as large as possible,
or equivalently for $I$~as small as possible.

It turns out that all ideals~$I\supset\Icomm$ arising in important applications known
to us contain the following ideal.
(The choice of its name will become clear soon.)


\begin{definition}
\label{d Irk}
The \emph{switchboard ideal} $\Irk$ is the ideal in~$\U$ generated by the elements
\begin{alignat}{2}
&u_b^2 u_a + u_a u_b u_a - u_b u_a u_b - u_b u_a^2  && (b-a=1), \label{i fg close}\\
&u_b^2 u_a - u_b u_a u_b && (b - a \ge 2), \label{i fg far1}\\
&u_b u_a^2 - u_a u_b u_a  && (b - a \ge 2), \label{i fg far2}\\
&u_b u_c u_a + u_a u_c u_b - u_b u_a u_c - u_c u_a u_b \qquad &&  (a < b
  < c). \label{quad uuu 3vars again}
\end{alignat}
\end{definition}

\begin{lemma}
$\Irk\supset\Icomm$.
\end{lemma}

\begin{proof}
The span of \eqref{i fg close}--\eqref{i fg far2}
contains \eqref{quad uu 2vars},
while \eqref{quad uuu 3vars again} is the same as~\eqref{quad uuu 3vars}.
Finally, the elements~\eqref{quad uuuu 3vars} are obtained as follows:
$u_c u_b u_c u_a + u_b u_c u_a u_c - u_c u_b u_a u_c - u_b u_c^2 u_a
=u_b(u_c u_a u_c - u_c^2 u_a)
+u_c (u_b u_c u_a + u_a u_c u_b - u_b u_a u_c - u_c u_a u_b)
+(u_c^2 u_a - u_c u_a u_c) u_b
\in \Irk$.
\end{proof}

\begin{remark}
We believe that the ideals $\Ifgp$ and~$\Irk$ are the smallest ``natural''
ideals containing~$\Icomm$ which are generated in degrees~$\le 3$.
(There is a precise statement along these lines, which we omit.)
Note that among the generators of~$\Icomm$,
the only ones of degree~$>3$ are those listed in~\eqref{quad uuuu 3vars}.
\end{remark}

The inclusions between the main ideals studied in this paper are
summarized in Figure~\ref{fig:some-ideals}
(several of these ideals are yet to be defined).
Except for $\Ifgp$ and  $\Icomm$, all these ideals contain the switchboard ideal~$\Irk$.
Our main focus henceforth will be on the study of the latter.

\begin{figure}[ht]
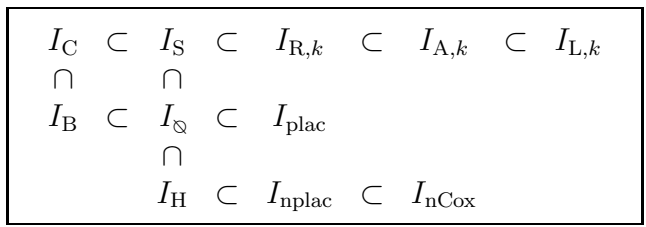

\begin{equation*}
\begin{array}{|ccccccccccc|}
\hline
\ &&&&&&&&&& \\[-.12in]
&\Icomm &\subset &\Irk & \subset & \Iaba{k} & \subset & \Iassaf{k} & \subset & \Ilam{k} &\\
&\cap      &             & \cap &&&&&&&\\
&\Ifgp      &\subset & \Ifg &\subset& \Iplac  &&&&& \\
&              &            & \cap &&&&&&&\\
&             &             & \Iweakh&\subset&\Inplac &\subset&\Incox &&&
\\[.05in]
\hline
\end{array}
\end{equation*}
\caption{\label{fig:some-ideals}The main
ideals studied in this
paper. All of them contain~$\Icomm$\,.
Schur positivity of the symmetric functions~$F_\gamma$, for $\gamma\in\U^*_{\ge 0}\cap I^\perp$,
is known to hold for~$I\supset\Ifgp$ and $I=\Ilam{k}$; is known to fail for $I\subset\Irk$;
and is conjectured for $\Iaba{k}$ and~$\Iassaf{k}$.
}
\end{figure}

\section{Switchboards}
\label{sec:switchboards}

In this section, we introduce a class of labeled graphs called
switchboards.
They serve as a combinatorial tool for studying
the symmetric functions~$F_\gamma$
associated to the switchboard ideal~$\Irk\,$, that is, those
labeled by vectors $\gamma\in\U^*_{\ge 0} \cap \Irkperp$.
This notion is crucial for unifying the perspectives of~\cite{SamiOct13}, \cite{FG},
and~\cite{LamRibbon}.

\begin{definition}[\emph{Switchboards}]
\label{def-switchboard}
Let $\e{w}=\e{w}_1\cdots \e{w}_n$ and $\e{w'}=\e{w}_1'\cdots
  \e{w}_n'$
be two words of the same length~$n$ in the
alphabet $\{1,\dots,N\}$.
We say that $\e{w}$ and $\e{w'}$  are related by a \emph{switch} in
position~$i$ if
$\e{w}_j=\e{w}_j'$ for $j\notin\{i-1,i,i+1\}$,
while the unordered pair 
$\{\e{w_{i-1}w_iw_{i+1}},
\e{w}'_{i-1}\e{w}'_i\e{w}'_{i+1}\}$
fits one of the following patterns (cf.\ Figure~\ref{f switches}):
\begin{itemize}
\item
$\{\e{bac}, \e{bca}\}$ or  $\{\e{acb}, \e{cab}\}$, with $a < b < c$
  (a \emph{Knuth switch});
\item
$\{\e{bac}, \e{acb}\}$ or  $\{\e{bca}, \e{cab}\}$, with $a < b < c$
  (a \emph{rotation switch});
\item
$\{\e{bab}, \e{bba}\}$ or $\{\e{aba}, \e{baa}\}$, with $a < b$
  (a \emph{Knuth switch});
\item
$\{\e{bab}, \e{aba}\}$ or $\{\e{bba}, \e{baa}\}$, with $b = a+1$
  (a \emph{braid/idempotent switch}).
\end{itemize}
A \emph{switchboard}
is an edge-labeled graph $\Gamma$ on a vertex set of words of fixed length~$n$
in the alphabet $\{1,\dots,N\}$
with edge labels from the set  $\{2,3,\dots,n-1\}$ such that
each $i$-edge (i.e., an edge labeled~$i$) corresponds to a switch in position~$i$,
and each vertex in~$\Gamma$ 
which has exactly one descent in positions $i-1$ and~$i$
belongs to exactly one $i$-edge.
\end{definition}

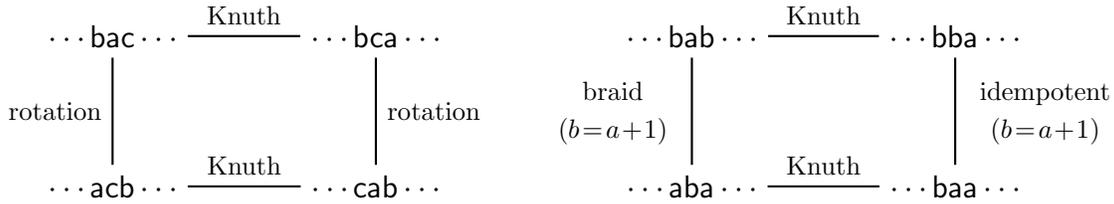
\begin{figure}[ht]
        \centerfloat
\ \\[-.1in]
\begin{tikzpicture}[xscale = 3.5,yscale = 2]
\tikzstyle{vertex}=[inner sep=0pt, outer sep=4pt]
\tikzstyle{framedvertex}=[inner sep=3pt, outer sep=4pt, draw=gray]
\tikzstyle{aedge} = [draw, thin, ->,black]
\tikzstyle{edge} = [draw, thick, -,black]
\tikzstyle{doubleedge} = [draw, thick, double distance=1pt, -,black]
\tikzstyle{hiddenedge} = [draw=none, thick, double distance=1pt, -,black]
\tikzstyle{dashededge} = [draw, very thick, dashed, black]
\tikzstyle{LabelStyleH} = [text=black, anchor=south]
\tikzstyle{LabelStyleHn} = [text=black, anchor=north]
\tikzstyle{LabelStyleV} = [text=black, anchor=east]
\tikzstyle{LabelStyleVw} = [text=black, anchor=west]

\node[vertex] (v1) at (0,1){$\e{\cdots bac\cdots} $};
\node[vertex] (v2) at (1,1){$\e{\cdots bca\cdots} $};
\node[vertex] (v3) at (0,0){$\e{\cdots acb\cdots} $};
\node[vertex] (v4) at (1,0){$\e{\cdots cab\cdots} $};
\draw[edge] (v1) to node[LabelStyleH]{\Small Knuth} (v2);
\draw[edge] (v3) to node[LabelStyleH]{\Small Knuth} (v4);
\draw[edge] (v1) to node[LabelStyleV]{\Small rotation} (v3);
\draw[edge] (v2) to node[LabelStyleVw]{\Small rotation} (v4);

\node[vertex] (v5) at (2.2,1){$\e{\cdots bab \cdots} $};
\node[vertex] (v6) at (3.2,1){$\e{\cdots bba \cdots} $};
\node[vertex] (v7) at (2.2,0){$\e{\cdots aba \cdots} $};
\node[vertex] (v8) at (3.2,0){$\e{\cdots baa \cdots} $};
\draw[edge] (v5) to node[LabelStyleH]{\Small Knuth} (v6);
\draw[edge] (v7) to node[LabelStyleH]{\Small Knuth} (v8);
\draw[edge] (v5) to node[LabelStyleV]{\begin{tabular}{cc}\Small
    braid\\ \Small ($b\!=\!a\!+\!1$)\end{tabular}} (v7);
\draw[edge] (v6) to node[LabelStyleVw]{\begin{tabular}{cc}\Small
    idempotent\\ \Small ($b\!=\!a\!+\!1$)\end{tabular}} (v8);
\end{tikzpicture}
\\[.1in]
\caption{\label{f switches}
Switches of different types.
}
\end{figure}

Examples of switchboards can be found in Figures \ref{f CKR graph}--\ref{f triples}.
In all these figures, we label the Knuth $i$-edges by~$i$,
and the non-Knuth $i$-edges by~$\tilde i$.

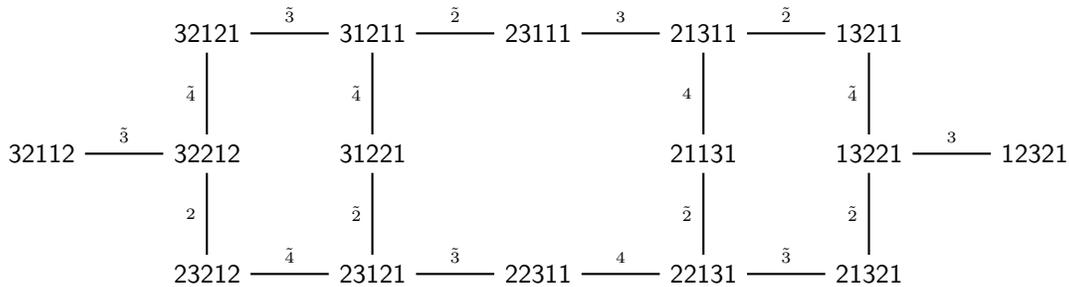
\begin{figure}[ht]
        \centerfloat
\begin{tikzpicture}[xscale = 2.2,yscale = 1.6]
\tikzstyle{vertex}=[inner sep=0pt, outer sep=4pt]
\tikzstyle{framedvertex}=[inner sep=3pt, outer sep=4pt, draw=gray]
\tikzstyle{aedge} = [draw, thin, ->,black]
\tikzstyle{edge} = [draw, thick, -,black]
\tikzstyle{doubleedge} = [draw, thick, double distance=1pt, -,black]
\tikzstyle{hiddenedge} = [draw=none, thick, double distance=1pt, -,black]
\tikzstyle{dashededge} = [draw, very thick, dashed, black]
\tikzstyle{LabelStyleH} = [text=black, anchor=south]
\tikzstyle{LabelStyleHn} = [text=black, anchor=north]
\tikzstyle{LabelStyleV} = [text=black, anchor=east]

\node[vertex] (v1) at (-2,-1){\footnotesize$\e{23212} $};
\node[vertex] (v2) at (-2,0){\footnotesize$\e{32212} $};
\node[vertex] (v3) at (-1,0){\footnotesize$\e{31221} $};
\node[vertex] (v4) at (1,-1){\footnotesize$\e{22131} $};
\node[vertex] (v5) at (0,-1){\footnotesize$\e{22311} $};
\node[vertex] (v6) at (2,0){\footnotesize$\e{13221} $};
\node[vertex] (v7) at (2,-1){\footnotesize$\e{21321} $};
\node[vertex] (v8) at (3,0){\footnotesize$\e{12321} $};
\node[vertex] (v9) at (-3,0){\footnotesize$\e{32112} $};
\node[vertex] (v10) at (-1,-1){\footnotesize$\e{23121} $};
\node[vertex] (v11) at (-2,1){\footnotesize$\e{32121} $};
\node[vertex] (v12) at (0,1){\footnotesize$\e{23111} $};
\node[vertex] (v13) at (1,0){\footnotesize$\e{21131} $};
\node[vertex] (v14) at (2,1){\footnotesize$\e{13211} $};
\node[vertex] (v15) at (-1,1){\footnotesize$\e{31211} $};
\node[vertex] (v16) at (1,1){\footnotesize$\e{21311} $};
\draw[edge] (v1) to node[LabelStyleV]{\Tiny$2 $} (v2);
\draw[edge] (v1) to node[LabelStyleH]{\Tiny$\tilde{4} $} (v10);
\draw[edge] (v2) to node[LabelStyleH]{\Tiny$\tilde{3} $} (v9);
\draw[edge] (v2) to node[LabelStyleV]{\Tiny$\tilde{4} $} (v11);
\draw[edge] (v3) to node[LabelStyleV]{\Tiny$\tilde{4} $} (v15);
\draw[edge] (v3) to node[LabelStyleV]{\Tiny$\tilde{2} $} (v10);
\draw[edge] (v4) to node[LabelStyleH]{\Tiny$\tilde{3} $} (v7);
\draw[edge] (v4) to node[LabelStyleH]{\Tiny$4 $} (v5);
\draw[edge] (v4) to node[LabelStyleV]{\Tiny$\tilde{2} $} (v13);
\draw[edge] (v5) to node[LabelStyleH]{\Tiny$\tilde{3} $} (v10);
\draw[edge] (v6) to node[LabelStyleV]{\Tiny$\tilde{4} $} (v14);
\draw[edge] (v6) to node[LabelStyleH]{\Tiny$3 $} (v8);
\draw[edge] (v6) to node[LabelStyleV]{\Tiny$\tilde{2} $} (v7);
\draw[edge] (v11) to node[LabelStyleH]{\Tiny$\tilde{3} $} (v15);
\draw[edge] (v12) to node[LabelStyleH]{\Tiny$3 $} (v16);
\draw[edge] (v12) to node[LabelStyleH]{\Tiny$\tilde{2} $} (v15);
\draw[edge] (v13) to node[LabelStyleV]{\Tiny$4 $} (v16);
\draw[edge] (v14) to node[LabelStyleH]{\Tiny$\tilde{2} $} (v16);
\end{tikzpicture}
\caption{\label{f CKR graph}
A switchboard with $N=3$ and $n=5$. 
}
\end{figure}

\begin{remark}
The notion of a switchboard was inspired by the D~graphs of
Assaf~\cite{SamiOct13, Sami2, SamiForum}.
More specifically, switchboards are a generalization of the D~graphs $\G^{(k)}_{c,D}$
of~\cite[\textsection4.2]{SamiOct13},
which Assaf introduced to study LLT polynomials; cf.\ Definition~\ref{def-assaf-llt-k}.
Switchboards also generalize the $\operatorname{D}_0$~graphs
of~\cite{BD0graph},
which were in turn motivated by~\cite{SamiOct13,Sami2, SamiForum}.
\end{remark}

The definition of a switchboard is justified by the following
statement, whose verification is straightforward
(cf.\ the proof of \cite[Proposition-Definition 3.2]{BD0graph}). 

\begin{proposition}
\label{pr:01Irk=switchboard}
Let $W$ be a set of words in~\,$\U^*$ of the same length.
Then $\sum_{\e{w} \in W} \e{w} \in \Irkperp$ if and only if $W$ is the vertex
set of a switchboard.
\end{proposition}

\begin{definition}
\label{def:FGamma}
For a switchboard~$\Gamma$ with vertex set~$W$, we set
$F_\Gamma(\mathbf{x})=F_\gamma(\mathbf{x})$ where $\gamma=\sum_{\e{w} \in  W} \e{w}$.
Thus $F_\Gamma(\mathbf{x})$ is a symmetric function in the variables
$\mathbf{x}\!=\!(x_1,x_2,\dots)$ defined as the sum of fundamental quasisymmetric
functions associated with descent sets of the vertices~of~$\Gamma$.
\end{definition}

\begin{example}
For the switchboard~$\Gamma$ in Figure~\ref{f CKR graph},
formula~\eqref{eq:Fgamma-via-Q} gives
\begin{align*}
F_\Gamma&=Q_{2}+Q_{3}+Q_{12}+3Q_{13}+2Q_{14}+2Q_{23}+3Q_{24}+Q_{34}+Q_{124}+Q_{134}\\
&=s_{32} + s_{311} + s_{221}\,.
\end{align*}
\end{example}

The rest of this paper is mostly concerned with the study of
switchboards and the associated symmetric functions~$F_\Gamma$.
Comparing Proposition~\ref{pr:01Irk=switchboard} to
Definitions~\ref{def:cauchy-product} and~\ref{def:FGamma},
we see that symmetric functions~$F_\Gamma$ associated to switchboards
are precisely the~$F_\gamma$'s
labeled by $(0,1)$-vectors $\gamma\in\Irkperp$.
(A~$(0,1)$-vector is simply a sum of a subset of words in~$\U^*$.)
Restricting the treatment to $(0,1)$-vectors is not much of an imposition
since in all important applications the labeling vectors~$\gamma$ have this form.

One advantage of switchboards 
is that one can consider their \emph{connected components}.
This notion involves the edges of a switchboard, not just its
vertices, and thus has no direct counterpart for the integer vectors~$\gamma$.
%
The following statement is easy to check.

\begin{proposition}
\label{p CKR set ops}
Connected components of a switchboard are switchboards.
If two switchboards have disjoint vertex sets consisting of words of the same length,
then their union is a switchboard.
\end{proposition}

\begin{example}
Figure~\ref{f intro} shows two different switchboards $\Gamma$
and~$\Gamma'$ (one disconnected, the other connected) which have the same set
of vertices~$W$; consequently $F_\Gamma=F_{\Gamma'}$.
This is an instance of a general phenomenon: if
$W$ contains four words of the form shown in Figure~\ref{f switches}
on the left, then one can choose either the two Knuth edges or the two
rotation edges.
A similar phenomenon occurs with 4-tuples of the form shown in Figure \ref{f switches} on the right (for $b=a+1$).
\end{example}

\begin{figure}[ht]
\ \\[-.1in]
\begin{tikzpicture}[xscale = 1.8,yscale = 1.6]
\tikzstyle{vertex}=[inner sep=0pt, outer sep=4pt]
\tikzstyle{aedge} = [draw, thin, ->,black]
\tikzstyle{edge} = [draw, thick, -,black]
\tikzstyle{dashededge} = [draw, very thick, dashed, black]
\tikzstyle{LabelStyleH} = [text=black, anchor=south]
\tikzstyle{LabelStyleV} = [text=black, anchor=east]
\tikzstyle{LabelStyleH2} = [text=black, anchor=north]
\tikzstyle{doubleedge} = [draw, thick, double distance=1pt, -,black]
\tikzstyle{hiddenedge} = [draw=none, thick, double distance=1pt, -,black]

\begin{scope} [xshift = -2cm]
 \node[vertex] (v1) at (1,2){\footnotesize$\e{2134}$};
 \node[vertex] (v2) at (2,2){\footnotesize$\e{2314}$};
 \node[vertex] (v3) at (3,2){\footnotesize$\e{2341}$};
 \node[vertex] (v4) at (2,1){\footnotesize$\e{2143}$};
 \node[vertex] (v5) at (3,1){\footnotesize$\e{2413}$};
\end{scope}
 \draw[edge] (v1) to node[LabelStyleH]{{\Tiny 2}} (v2);
 \draw[edge] (v2) to node[LabelStyleH]{{\Tiny 3}} (v3);
 \draw[doubleedge] (v4) to node[LabelStyleH]{{\Tiny 2}} (v5);
 \draw[hiddenedge] (v4) to node[LabelStyleH2]{{\Tiny 3}} (v5);

\begin{scope} [xshift = 2cm]
 \node[vertex] (v1) at (1,2){\footnotesize$\e{2134}$};
 \node[vertex] (v2) at (2,2){\footnotesize$\e{2314}$};
 \node[vertex] (v3) at (3,2){\footnotesize$\e{2341}$};
 \node[vertex] (v4) at (2,1){\footnotesize$\e{2143}$};
 \node[vertex] (v5) at (3,1){\footnotesize$\e{2413}$};
\end{scope}
 \draw[edge] (v1) to node[LabelStyleH]{{\Tiny 2}} (v2);
 \draw[edge] (v2) to node[LabelStyleV]{{\Tiny $\tilde{3}$}} (v4);
 \draw[edge] (v4) to node[LabelStyleH]{{\Tiny 2}} (v5);
 \draw[edge] (v3) to node[LabelStyleV]{{\Tiny $\tilde{3}$}} (v5);
 \end{tikzpicture}
\caption{\label{f intro} Different switchboards on the same set of
  vertices, with $N=n=4$. The switchboard $\Gamma$ on the left has two connected
  components $\Gamma_{\operatorname{top}}$
  and~$\Gamma_{\operatorname{bottom}}$, with
$F_{\Gamma_{\operatorname{top}}}=s_{31}$ and $F_{\Gamma_{\operatorname{bottom}}}=s_{22}$.
The switchboard~$\Gamma'$ on the right is connected,
with $F_{\Gamma'} = F_\Gamma = s_{31}+s_{22}$.
}
\end{figure}

\pagebreak[3]

The results of \cite{SamiOct13, BLamLLT, FG},
when translated into the language of switchboards,
show that in many important cases, the symmetric
functions~$F_\Gamma(\mathbf{x})$
associated with switchboards~$\Gamma$ are Schur positive.
Computer experiments supply an extensive list of additional instances
of Schur positivity.
Unfortunately, there exist switchboards whose symmetric functions are not Schur
positive.
They are not easy to find, but once a counterexample has been
discovered, its verification is a straightforward calculation.

\begin{proposition}
\label{pr:456213+...}
The symmetric function $F_\Gamma$ associated with the switchboard
shown in Figure~\ref{f schur pos counterexample} is not Schur
positive:
$F_\Gamma(\mathbf{x})=s_{321}(\mathbf{x})+s_{2211}(\mathbf{x})-s_{222}(\mathbf{x})$.
\end{proposition}

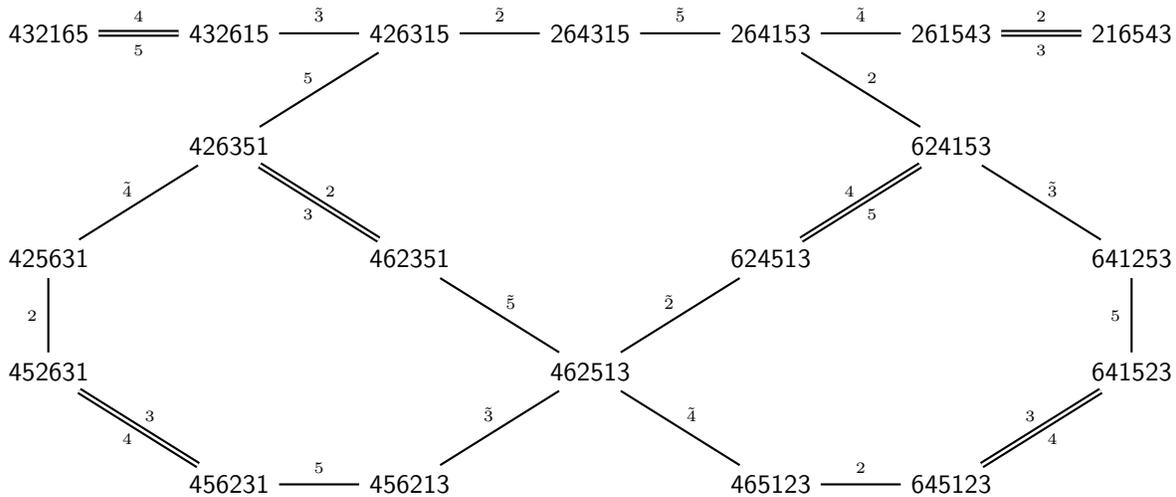
\begin{figure}[ht]
\centerfloat
\ \\[-.05in]
\begin{tikzpicture}[xscale = 2.4,yscale = 1.5]
\tikzstyle{vertex}=[inner sep=0pt, outer sep=4pt]
\tikzstyle{framedvertex}=[inner sep=3pt, outer sep=4pt, draw=gray]
\tikzstyle{aedge} = [draw, thin, ->,black]
\tikzstyle{edge} = [draw, thick, -,black]
\tikzstyle{doubleedge} = [draw, thick, double distance=1pt, -,black]
\tikzstyle{hiddenedge} = [draw=none, thick, double distance=1pt, -,black]
\tikzstyle{dashededge} = [draw, very thick, dashed, black]
\tikzstyle{LabelStyleH} = [text=black, anchor=south]
\tikzstyle{LabelStyleHn} = [text=black, anchor=north]
\tikzstyle{LabelStyleNE} = [text=black, anchor=south west, inner sep=2pt]
\tikzstyle{LabelStyleNEn} = [text=black, anchor=north east, inner sep=2pt]
\tikzstyle{LabelStyleV} = [text=black, anchor=east]
\tikzstyle{LabelStyleVn} = [text=black, anchor=west]
\tikzstyle{LabelStyleNW} = [text=black, anchor=south east, inner sep=2pt]
\tikzstyle{LabelStyleNWn} = [text=black, anchor=north west, inner sep=2pt]

\node[vertex] (v2) at (3,1){\footnotesize$\e{456213}$};
\node[vertex] (v17) at (5,1){\footnotesize$\e{465123}$};
\node[vertex] (v8) at (6,1){\footnotesize$\e{645123}$};
\node[vertex] (v6) at (7,2){\footnotesize$\e{641523}$};
\node[vertex] (v20) at (7,3){\footnotesize$\e{641253}$};
\node[vertex] (v13) at (2,1){\footnotesize$\e{456231}$};
\node[vertex] (v5) at (4,2){\footnotesize$\e{462513}$};
\node[vertex] (v16) at (5,3){\footnotesize$\e{624513}$};
\node[vertex] (v19) at (6,4){\footnotesize$\e{624153}$};
\node[vertex] (v12) at (1,2){\footnotesize$\e{452631}$};
\node[vertex] (v15) at (3,3){\footnotesize$\e{462351}$};
\node[vertex] (v10) at (5,5){\footnotesize$\e{264153}$};
\node[vertex] (v18) at (6,5){\footnotesize$\e{261543}$};
\node[vertex] (v9) at (7,5){\footnotesize$\e{216543}$};
\node[vertex] (v3) at (1,3){\footnotesize$\e{425631}$};
\node[vertex] (v7) at (2,4){\footnotesize$\e{426351}$};
\node[vertex] (v11) at (3,5){\footnotesize$\e{426315}$};
\node[vertex] (v1) at (4,5){\footnotesize$\e{264315}$};
\node[vertex] (v14) at (2,5){\footnotesize$\e{432615}$};
\node[vertex] (v4) at (1,5){\footnotesize$\e{432165}$};

\draw[edge] (v1) to node[LabelStyleH]{\Tiny$\tilde{2} $} (v11);
\draw[edge] (v1) to node[LabelStyleH]{\Tiny$\tilde{5} $} (v10);
\draw[edge] (v2) to node[LabelStyleH]{\Tiny$5 $} (v13);
\draw[edge] (v2) to node[LabelStyleNW]{\Tiny$\tilde{3} $} (v5);
\draw[edge] (v3) to node[LabelStyleNW]{\Tiny$\tilde{4} $} (v7);
\draw[edge] (v3) to node[LabelStyleV]{\Tiny$2 $} (v12);
\draw[doubleedge] (v4) to node[LabelStyleH]{\Tiny$4$} (v14);
\draw[hiddenedge] (v4) to node[LabelStyleHn]{\Tiny$5$} (v14);
\draw[edge] (v5) to node[LabelStyleNE]{\Tiny$\tilde{4} $} (v17);
\draw[edge] (v5) to node[LabelStyleNW]{\Tiny$\tilde{2} $} (v16);
\draw[edge] (v5) to node[LabelStyleNE]{\Tiny$\tilde{5} $} (v15);
\draw[edge] (v6) to node[LabelStyleV]{\Tiny$5 $} (v20);
\draw[doubleedge] (v6) to node[LabelStyleNW]{\Tiny$3$} (v8);
\draw[hiddenedge] (v6) to node[LabelStyleNWn]{\Tiny$4$} (v8);
\draw[edge] (v7) to node[LabelStyleNW]{\Tiny$5 $ } (v11);
\draw[doubleedge] (v7) to node[LabelStyleNE]{\Tiny$2$} (v15);
\draw[hiddenedge] (v7) to node[LabelStyleNEn]{\Tiny$3$} (v15);
\draw[edge] (v8) to node[LabelStyleH]{\Tiny$2 $} (v17);
\draw[doubleedge] (v9) to node[LabelStyleH]{\Tiny$2$} (v18);
\draw[hiddenedge] (v9) to node[LabelStyleHn]{\Tiny$3$} (v18);
\draw[edge] (v10) to node[LabelStyleH]{\Tiny$\tilde{4} $} (v18);
\draw[edge] (v10) to node[LabelStyleNE]{\Tiny$2 $} (v19);
\draw[edge] (v11) to node[LabelStyleH]{\Tiny$\tilde{3} $} (v14);
\draw[doubleedge] (v12) to node[LabelStyleNE]{\Tiny$3$} (v13);
\draw[hiddenedge] (v12) to node[LabelStyleNEn]{\Tiny$4$} (v13);
\draw[doubleedge] (v16) to node[LabelStyleNW]{\Tiny$4$} (v19);
\draw[hiddenedge] (v16) to node[LabelStyleNWn]{\Tiny$5$} (v19);
\draw[edge] (v19) to node[LabelStyleNE]{\Tiny$\tilde{3} $} (v20);
\end{tikzpicture}
\ \\[.15in]
\caption{\label{f schur pos counterexample}
Switchboard $\Gamma$ with $N=n=6$, $F_\Gamma=s_{321}+s_{2211}-s_{222}$.}
\end{figure}

\pagebreak[3]

Proposition~\ref{pr:456213+...} leads to the following counterexamples.

\begin{corollary}
\label{cor-nopos}
Conditions {\rm (i)-(ii)} in Theorem~\ref{t intro positivity} fail for
$I=\Irk$ (hence for $I=\Icomm$), for any $N\ge 6$.
Specifically:
\begin{itemize}
\item[{\rm (i)}]
the nonnegative vector
\begin{align}
\notag
\gamma =
\e{432165}+\e{432615}+\e{426315}+\e{264315}+\e{264153}+\e{261543}+\e{216543}&
\\
+\e{426351}+\e{624153}+\e{425631}+\e{462351}+\e{624513}+\e{641253} \qquad
\label{eq:gamma-counterexample}
\\
\notag
+\e{452631}+\e{462513}+\e{641523}+\e{456231}+\e{456213}+\e{465123}+\e{645123}
&
\end{align}
lies in $
\Irkperp$
but the symmetric function $F_\gamma=s_{321}+s_{2211}-s_{222}$ is not Schur positive;
\item[{\rm (ii)}]
for 
$\lambda=(2,2,2)$,
the noncommutative Schur function $\mathfrak{J}_{\lambda}(\mathbf{u})$
is not $\QQ$-monomial positive modulo~$\Irk$, as witnessed by the fact
that the coefficient of  $s_\lambda$ in  $F_\gamma$ is  $-1$
for the vector $\gamma\in\U^*_{\ge 0}\cap\Irkperp$
defined by~\eqref{eq:gamma-counterexample}.
\end{itemize}
\end{corollary}

\begin{proof}
The vector~$\gamma$ given in~\eqref{eq:gamma-counterexample}
is the sum of vertices of the switchboard~$\Gamma$ in
Figure~\ref{f schur pos counterexample}.
Propositions \ref{pr:01Irk=switchboard} and~\ref{pr:456213+...} imply that $\gamma\in\Irkperp$
and $F_\gamma = F_\Gamma=s_{321} + s_{2211} - s_{222}$.
Theorem~\ref{t intro positivity} then yields~{\rm (ii)}.
\end{proof}

Given that Schur positivity of $F_\Gamma$ fails for general
switchboards~$\Gamma$, one wonders which additional restrictions
imposed on~$\Gamma$ would ensure that $F_\Gamma$ is Schur positive.
One way to obtain results of this kind is to take an ideal~$I\supset\Irk$ such
that $\mathfrak{J}_\lambda(\mathbf{u})$ is known to be $\QQ$-monomial positive
modulo~$I$, and consider the corresponding class of switchboards.
Various instances of this approach are discussed in
subsequent sections of this paper.

An alternative strategy is to examine the
counterexamples of switchboards for which Schur positivity fails,
recognize their distinctive combinatorial features,
and impose restrictions which would rule them out.
The following remarks illustrate this approach.

\begin{remark}
\label{r axiom 5}
A natural subclass of switchboards is defined by the property that
``non-overlapping switches commute:''
\begin{equation}
\label{eq:axiom5}
\parbox{14.4cm}{if a switchboard has an $i$-edge $\{\e{v,w}\}$
and a $j$-edge $\{\e{v,w'}\}$, with $|i-j| \geq 3$, then
it must contain a vertex  $\e{v'}$ such that
 $\{\e{w, v'}\}$ is a $j$-edge and $\{\e{w',v'}\}$ is an $i$-edge.
}
\end{equation}
(This is the same as D graph axiom~5 from~\cite{SamiOct13}.)
The switchboard in Figure~\ref{f schur pos counterexample} does not
satisfy condition~\eqref{eq:axiom5}:
starting at $\e{v} \!=\! \e{426315}$ and following the $2$-edge and then
the $5$-edge,
we arrive at a different vertex \emph{vs.}\ following the  $5$-edge and then the  $2$-edge.
It is natural to ask whether all switchboards satisfying
condition~\eqref{eq:axiom5} have Schur positive symmetric functions.
Unfortunately this is also false by the example from \cite[Theorem~1.5]{BD0graph}.
\end{remark}

\begin{remark}
Another natural restriction is the following ``locality'' property:
\begin{equation}
\label{eq:axiom-locality}
\begin{array}{l}
\text{if two words $\e{v}$ and $\e{w}$ appearing in the same switchboard coincide in}\\
\text{positions $(i-1,i,i+1)$, then the $i$-edges incident to $\e{v}$ and $\e{w}$ (if present)}\\
\text{are of the same type (i.e., simultaneously Knuth or non-Knuth).}
\end{array}
\end{equation}
It is easy to see that \eqref{eq:axiom-locality} implies~\eqref{eq:axiom5}.
Hence the switchboard in Figure~\ref{f schur pos counterexample} must
violate~\eqref{eq:axiom-locality}. \linebreak[3]
%
(To see this directly, take $i = 5$, $\e{v} = \e{264315}$, and
$\e{w} = \e{426315}$.)
A~further strengthening of \eqref{eq:axiom-locality} is the following ``strong locality'' property:
\begin{equation}
\label{eq:axiom-strong-locality}
\text{for each $i$, all $i$-edges 
are of the same type, i.e., all Knuth or all non-Knuth.}
\end{equation}

\end{remark}

\begin{problem}
{\rm Is it true that for any switchboard~$\Gamma$
satisfying condition~\eqref{eq:axiom-locality}
(resp., the stronger condition~\eqref{eq:axiom-strong-locality}),
the symmetric function $F_\Gamma$ is Schur positive?}
\end{problem}

\begin{definition}
\label{def:dzero-graph}
A \emph{$\Dzero$~graph} is a switchboard
whose words have no repeated letters.
\end{definition}

$\Dzero$~graphs were~the main objects of study in~\cite{BD0graph},
and were studied, with slightly different conventions,
in~\cite{Sami2, SamiForum}.
Since switchboards arising in many applications are $\Dzero$~graphs,
it is tempting to speculate that passing from general switchboards
in a certain subclass
to $\Dzero$~graphs might help establish Schur positivity.
Our investigations show that typically this is not the case. \linebreak[3]
For example, the switchboard~$\Gamma$ in Figure~\ref{f
  schur pos counterexample} is a $\Dzero$~graph,
yet $F_\Gamma$ is not Schur positive.
The existence of $\Dzero$~graphs for which Schur positivity fails
was already shown in~\cite{BD0graph}, but the example given here is
smaller (20~vertices only).

\begin{definition}
\label{def:Ist}
Let  $\Ist$ denote the ideal in~$\U$ generated by the monomials
$\mathbf{u}_{\e{w}}$ whose associated word~$\e{w}$ has a repeated letter.
\end{definition}

Roughly speaking, passing from switchboards to
  $\Dzero$~graphs corresponds to passing from an ideal~$I$ to
  $I+\Ist$.
As noted above, this enlargement rarely affects Schur positivity.


\section{LLT polynomials}
\label{s LLT}
\vspace{-2mm}

LLT polynomials are certain  $q$-analogs of products of skew Schur
functions introduced by Lascoux, Leclerc, and Thibon~\cite{LLT}.
In this section,
by recasting the work of Lam~\cite{LamRibbon} and Assaf~\cite{SamiOct13},
we define a class of switchboards
whose symmetric functions are the coefficients of powers of~$q$ in an LLT polynomial.
We explain how one of the ideas of~\cite{SamiOct13} can be formulated as
a conjectural strengthening
of LLT positivity (Conjecture \ref{cj sami}).
We~conclude by reviewing related work done in~\cite{BLamLLT}.

There are two versions of LLT polynomials, which we distinguish
following the notation of~\cite{GH}: the combinatorial LLT polynomials
of \cite{LLT} defined using spin, and the new variant combinatorial
LLT polynomials of \cite{HHLRU} defined using inversion numbers. 
Although the theory of noncommutative Schur functions is well suited
to studying the former (see \cite{LamRibbon}), we prefer to work with
the latter, in order to be consistent with~\cite{SamiOct13,BLamLLT},
and  also because inversion numbers are easier to calculate
than~spin.

We begin by recalling from  \cite{SamiOct13} a formula defining LLT polynomials via expansions
in fundamental quasisymmetric functions.
This will require some preparation.

We adopt the English (matrix-style) convention for drawing
skew shapes (=skew Young diagrams) and tableaux, so that row (resp., column)
labels start with~1 and increase from north to south (resp., from west to
east).
For a cell~$z$ located in row~$a$ and column~$b$~in a skew
shape~$\beta$,
the \emph{content} $c(z)$ of~$z$ is defined by $c(z)=b-a$. \linebreak[3]
Here we view~$\beta$ as a subset of $\ZZ_{\ge 1}\times\ZZ_{\ge 1}$,
i.e., we do distinguish between skew shapes that differ by translations.

Throughout this section, $k$ denotes a positive integer.
Let
\[
\bm{\beta} =(\beta^{(0)},\dots,\beta^{(k-1)})
\]
be a $k$-tuple of skew shapes. 
The \emph{shifted content} $\tilde{c}(z)=\tilde{c}_{\bm{\beta}}(z)$
of a cell $z \in \beta^{(r)}$ is defined~by
\begin{align} \label{e shifted content}
\tilde{c}(z) = c(z)\,k+r.
\end{align}

\begin{example}
\label{ex:2-33-33/1-11-21}
Let $k=3$ and
$\bm{\beta}=\left(\partition{&~\\&}\,,
\partition{&~&~\\&~&~}\,,
\partition{&&~\\&~&~}\,\right) =
(2,33,33)/(1,11,21)$.
The only cell $z\in\beta^{(0)}$ is located in row~$1$ and
column~$2$, so $c(z)=1$ and $\tilde{c}(z)=3$.
The shifted contents of the cells in $\bm{\beta}$ are shown below:
\setlength{\cellsize}{2.1ex}
\vspace{-1mm}
\[
{\footnotesize
\left(\tableau{&3\\&}\,,\tableau{&4&7\\&1&4}\,,\tableau{&&8\\&2&5}\,\right)}.
\]
\end{example}

Let $\mathbb{Y}$ denote the set of all Young diagrams.
Define a right action of  $\U$ on  $\ZZ \mathbb{Y}$
(the free $\ZZ$-module with basis~$\mathbb{Y}$)
by
\[
\nu \cdot u_c = \begin{cases}
\mu \qquad & \text{if $\mu/\nu$ is a cell of content $c$;} \\
0 \qquad & \text{otherwise.}
\end{cases}
\]
Next define a right action of $\U$ on  $\ZZ \mathbb{Y}^k$ as follows:
for $\bm{\nu} =(\nu^{(0)},\dots,\nu^{(k-1)}) \in \mathbb{Y}^k$, set
\[
\bm{\nu} \cdot u_{ck+r}
= (\ldots, \nu^{(r-1)},\, \nu^{(r)} \cdot u_{c}\,, \, \nu^{(r+1)},\ldots),
\]
for $c\in\ZZ$ and $r\in\{0,1,\dots,k-1\}$.
Now, for a $k$-tuple $\bm{\beta} = \bm{\mu}/\bm{\nu}$ of skew shapes, set
\begin{equation}
\label{eq:Wik-beta}
\Wi{k}(\bm{\beta})  = \{\e{v} \text{ a word in } \U^*
\mid \bm{\nu} \cdot \mathbf{u}_{\e{v}}=  \bm{\mu}\}.
\end{equation}
It is not hard to see that every word in  $\Wi{k}(\bm{\beta})$ is a
rearrangement of the shifted contents of
the cells in~$\bm{\beta}$.

\pagebreak[3]

The last ingredient we shall need is the \emph{$k$-inversion
  statistic} $\invi{k}\,$.
This is the function on words  $\e{v}=\e{v}_1\e{v}_2\cdots\in\U^*$
defined as the cardinality of the set of inversions in~$\e{v}$
formed by pairs of entries which differ by less than~$k$:
\begin{align*}
\invi{k}(\e{v}) &= \big|\big\{(i,j)\mid \text{$i<j$ and $0<\e{v_i}-\e{v_j}<k$}\big\} \big|.
\end{align*}

\begin{example}
\label{ex LLT}
We continue with Example~\ref{ex:2-33-33/1-11-21}.
Some elements of  $\Wi{3}(\bm{\beta})$, together with their  $3$-inversion numbers, are:
\[
\begin{array}{l|ccccc}
\e{v} &
\e{42173845} & \e{34174285}& \e{83412745}& \e{48714235} & \e{28534174}\\[1.4mm]
\invi{3}(\e{v}) & 4 & 5 & 5 & 6 & 6
\end{array}
\]
\end{example}

\begin{definition}
\label{def:new-variant-LLT}
Let $q$ be a formal parameter.
Let  $\bm{\beta}$ be a  $k$-tuple of skew shapes 
in which the shifted contents of all cells lie in $\{1,\dots,N\}$.
The \emph{new variant combinatorial LLT polynomial} indexed by
$\bm{\beta}$ is defined by
\begin{align}
\label{eq:new-variant-LLT}
\mathcal{G}_{\bm{\beta}}(\mathbf{x};q) =
\sum_{\e{v}\in\Wi{k}(\bm{\beta})}q^{\invi{k}(\e{v})}Q_{\Des(\e{v})}(\mathbf{x}).
\end{align}
(We note that no generality is lost by the restriction on the shifted
contents of  $\bm{\beta}$
because $\mathcal{G}_{\bm{\beta}}(\mathbf{x};q)$ is unchanged by
translating all the $\beta^{(i)}$ horizontally by the same constant.)
\end{definition}

\begin{remark}
This family of polynomials was originally defined in
\cite{HHLRU}  as a sum over  $k$-tuples of semistandard tableaux with
an inversion statistic.
An expression in terms of quasisymmetric functions was given by Assaf
\cite[Corollary~4.3]{SamiOct13}.
Formula~\eqref{eq:new-variant-LLT}, which appeared in
\cite[Proposition 2.8]{BLamLLT},
is an adaptation of Assaf's description.
Be aware that the words~$\e{v}$ used herein are inverses of those used in
\cite{SamiOct13}; see \cite[\textsection2.5--2.6]{BLamLLT} for
details.
\end{remark}

\begin{remark}\label{r spin LLT}
The \emph{combinatorial LLT polynomials}
$G^{(k)}_{\mu/\nu}(\mathbf{x};\qlam)$ are defined as generating
functions for $k$-ribbon tableaux weighted by spin.
(To be precise, \cite{GH, LamRibbon, LT00} use spin whereas \cite{LLT} uses cospin.)
By a result of \cite{HHLRU}, setting $q = \qlam^{-2}$
identifies the two types of LLT polynomials
up to a power of~$\qlam$,
once the indexing shapes $\mu/\nu$ and~$\bm{\beta}$ are related
via the $k$-quotient correspondence.
See \cite[Proposition 6.17]{GH} for details.
\end{remark}

The main open problem in this subfield of algebraic combinatorics
concerns the coefficients $\mathfrak{c}_{\bm{\beta}}^\lambda(q)$
appearing in the Schur expansions of LLT polynomials:
\begin{equation}
\label{eq:LLT-LR}
\mathcal{G}_{\bm{\beta}}(\mathbf{x};q) =
\sum_{\lambda}\mathfrak{c}_{\bm{\beta}}^\lambda(q)s_\lambda(\mathbf{x}).
\end{equation}
Each $\mathfrak{c}_{\bm{\beta}}^\lambda(q)$ is known to be a polynomial in $q$
with nonnegative integer coefficients.
In the case that $\bm{\beta}$ is a tuple of partition shapes,
this was established by Leclerc-Thibon~\cite{LT00} and Kashiwara-Tanisaki~\cite{KT}; the former showed that
the coefficients $\mathfrak{c}_{\bm{\beta}}^\lambda(q)$ are essentially parabolic Kazhdan-Lusztig polynomials,
and the latter proved geometrically that these polynomials
have nonnegative integer coefficients.
The general case was established by Grojnowski and Haiman~\cite{GH}, also
using Kazhdan-Lusztig theory.
Unfortunately, none of the existing approaches
produces an explicit positive combinatorial interpretation of
$\mathfrak{c}_{\bm{\beta}}^\lambda(q)$, not even a conjectural one.
(The approach of \cite{SamiOct13}, though combinatorial, hinges on an
intricate algorithm for transforming a  D~graph into a dual equivalence
graph, and has yet to produce explicit formulas for  $k > 2$.)

\pagebreak[3]

The problem of finding a Littlewood-Richardson-type rule for
LLT polynomials is particularly important because its solution would
immediately yield a similar rule for transformed Macdonald polynomials:
as mentioned earlier, the
Haglund-Haiman-Loehr formula~\cite{HHL} expresses the transformed
Macdonald polynomials $\tilde{H}_\mu(\mathbf{x};q,t)$ as positive
sums of LLT polynomials.


In~\cite{LamRibbon}, Lam introduced his algebra of ribbon Schur
operators,
providing an elegant algebraic framework for LLT polynomials.
This set the stage for applying the theory of noncommutative Schur
functions \cite{FG}  to the study of Schur expansions of
LLT polynomials.
Following \cite{BLamLLT}, we work with the following slight variant
of Lam's construction.

\begin{definition} 
\label{d Lam algebra}
Let $\U_q = \QQ(q) \tsr_{\ZZ} \U$ where $\QQ(q)$ is the field of
rational functions in one variable~$q$.
\emph{Lam's ideal} $\Jlam{k}$ is the ideal in $\U_q$ generated by the elements
\begin{alignat}{3}
&u_a^2 &&  \text{for all $a$,} \\
&u_{a+k}u_au_{a+k} &&  \text{for all $a$,}  \\
&u_au_{a+k}u_a && \text{for all $a$,}  \label{e no repeat Lam}\\
&u_au_b - u_bu_a \qquad && \text{for $b-a > k$}, \label{e far commute} \\
&u_au_b - q^{-1} u_bu_a \qquad && \text{for $0<b-a<k$.} \label{e q commute}
\end{alignat}
We denote $\Ilam{k}\!=\!\Jlam{k} \cap\,\U$
(viewing $\U$ as the $\ZZ$-subalgebra of $\U_q$ generated by 1 and the~$u_i$).
In what follows, we mostly work with~$\Jlam{k}$ rather
than~$\Ilam{k}$.
Even though it is the latter that is contained in~$\U$,
we feel that the former is a more natural and important concept.
\end{definition}

Lam's ideal $\Jlam{k}$ contains $\QQ(q) \tsr_\ZZ \Irk$;
this will follow from Proposition \ref{p lam and assaf} below.

A word $\e{v}\in\U^*$ is called a \emph{nonzero $k$-word} if
for every pair $i<j$ such that $\e{v}_i = \e{v}_j$, there
exist  $s,t$
satisfying $i<s<t<j$ and $\{\e{v}_s,\e{v}_t\} =
\{\e{v}_i-k,\e{v}_i+k\}$.
The relevance of this notion in our current context
becomes clear from the following lemmas,
which can be found in \cite[Proposition-Definition~2.2 and
  Proposition~2.6(i)]{BLamLLT}.

\begin{lemma} 
\label{lem:nonzero-k-words}
A word $\e{v}$ is a nonzero $k$-word if and only if $\mathbf{u}_{\e{v}}\notin\Jlam{k}$.
\end{lemma}

\begin{lemma} 
The nonzero  $k$-words are precisely the words which appear in the
sets $\Wi{k}(\bm{\beta})$,
as $\bm{\beta}$ ranges over all $k$-tuples
of skew shapes.  
\end{lemma}

Lam \cite{LamRibbon} used the above construction
to obtain formulas (not manifestly positive) for the coefficients of
Schur expansions of LLT polynomials.
In our language, his result can be stated as follows.
\pagebreak[3]

\begin{proposition}
\label{p llt formula}
Let $\bm{\beta}$ be a  $k$-tuple of skew shapes 
with shifted contents in $\{1,\dots,N\}$.
Then the corresponding new variant combinatorial LLT polynomial can be written as
\begin{equation}
\label{eq:LLT=Fgamma}
\mathcal{G}_{\bm{\beta}}(\mathbf{x};q)
= F_\gamma(\mathbf{x})
\end{equation}
where
\[
\gamma= \sum_{\e{v}\in\Wi{k}(\bm{\beta})} q^{\invi{k}(\e{v})}\,\e{v}\in
\Jlamperp{k}\,.
\]
Consequently,
$\mathfrak{c}_{\bm{\beta}}^\lambda(q) = \langle
\mathfrak{J}_{\lambda}(\mathbf{u}), \gamma \rangle$.
(Here we use the obvious modifications of 
our basic notions for $\gamma\in \QQ(q)\otimes\U^*$.)
\end{proposition}

\pagebreak[3]

\begin{proof}
The proof of the fact that $\gamma\in \Jlamperp{k}$
is exactly the same as the proof of
the first statement of  \cite[Proposition 5.4]{BD0graph} (even though
the $\Wi{k}(\bm{\beta})$ defined here is more general than that
defined in~\cite{BD0graph}).
The formula~\eqref{eq:LLT=Fgamma} is immediate from
comparing~\eqref{eq:new-variant-LLT} with~\eqref{eq:Fgamma-via-Q}.
The last statement then follows from a straightforward modification of Proposition \ref{Icomm-Schur}
to the coefficient field~$\QQ(q)$ (as opposed to working over~$\ZZ$).
\end{proof}


The switchboards studied in this paper were partially motivated by the
D~graphs of Assaf~\cite{SamiOct13}.
These graphs were introduced to give a combinatorial explanation of LLT and Macdonald positivity.
We next explain some of Assaf's ideas using the language
developed in Sections~\ref{s main results} and~\ref{sec:switchboards},
and connect them to Lam's construction reviewed above.

\begin{definition}
\label{d I assaf}
The \emph{Assaf ideal}  $\Iassaf{k}$ is the ideal in~$\U$ generated by
all monomials $\mathbf{u}_{\e{w}}\in\Jlam{k}$
(cf.\ Lemma~\ref{lem:nonzero-k-words})
together with the elements listed in \eqref{e Iassaf1}--\eqref{e Iassaf4}.
\end{definition}

It is easy to see that
$\Iassaf{k}\supset \Iaba{k}\supset \Irk$.




\begin{proposition}
\label{p lam and assaf}
$\Jlam{k}\supset \QQ(q) \tsr_\ZZ \Iassaf{k}$.
\end{proposition}

\begin{proof}
To ease notation in our argument, we identify each element
$g\in\Iassaf{k}$ with 
$1 \tsr g \in \QQ(q) \tsr_\ZZ
\Iassaf{k}$.
It suffices to show that the generators of $\Iassaf{k}$
vanish
modulo~$\Jlam{k}$.
This clearly holds for \eqref{e Iassaf1}--\eqref{e
  Iassaf2}, cf.~\eqref{e far commute}.
To check that the generators \eqref{e Iassaf3} vanish modulo
$\Jlam{k}$, we compute, using \eqref{e q commute} twice:
\begin{align*}
&&u_bu_au_c &\equiv qu_au_bu_c \equiv u_au_cu_b \ \bmod \Jlam{k}
&&\quad \text{for $c-a \leq k$ and $a<b<c$.}
\end{align*}
The argument for the generators \eqref{e Iassaf4} is similar.
\end{proof}

Proposition~\ref{p lam and
  assaf} implies that $\Ilam{k}=\Jlam{k}\cap\U \supset \Iassaf{k}$.

\begin{definition} 
\label{def-assaf-llt-k}
An \emph{Assaf switchboard} of level~$k$
is a switchboard in which every vertex is a nonzero
$k$-word, and every $i$-edge with endpoints
$\e{w}$ and $\e{w}'$ corresponds to a switch where
$\{\e{w}_{i-1}\e{w}_i\e{w}_{i+1}, \e{w}'_{i-1}\e{w}'_i\e{w}'_{i+1}\}$
fits one of the patterns
\begin{align}
&\text{$\{\e{bac}, \e{bca}\}$ or  $\{\e{acb}, \e{cab}\}$, with $c-a > k$
    and $a<b<c$;} \label{e assaf switchboard 1}\\
&\text{$\{\e{bac}, \e{acb}\}$ or  $\{\e{bca}, \e{cab}\}$, with $c-a \leq
    k$ and $a < b < c$.} \label{e assaf switchboard 2}
\end{align}
\end{definition}

Assaf ideals $\Iassaf{k}$ and Assaf switchboards are closely related,
as the following proposition shows; its proof is straightforward.

\begin{proposition}
\label{p assaf ideal switchboards}
For a set of words $W$ of the same length, the following are equivalent:
\begin{itemize}
\item $\sum_{\e{w}\in W} \e{w} \in \Iassafperp{k}$;
\item $W$ is the vertex set of an Assaf switchboard of level~$k$.
\end{itemize}
If these conditions hold,
then there is a unique Assaf switchboard with vertex set $W$.
\end{proposition}

\pagebreak[3]



The next result and its corollary relate Assaf switchboards to LLT polynomials.

\begin{proposition}
\label{p Sami LLT graphs}
Let $t$ be a nonnegative integer, and $\bm{\beta}$ a  $k$-tuple of skew shapes
in which the shifted contents of all cells lie in $\{1,\dots,N\}$.
Then there is a unique level~$k$ Assaf switchboard
$\Gamma=\Gamma_k(\bm{\beta},t)$
whose vertex set is $\{\e{v} \in \Wi{k}(\bm{\beta}) \mid
\invi{k}(\e{v}) = t\}$.
\end{proposition}

\pagebreak[3]

\begin{proof}
Suppose that words $\e{w}$ and $\e{w}'$ are related
by a switch of the form \eqref{e assaf switchboard 1} or \eqref{e
  assaf switchboard 2}. We must show that
\begin{align}
&\e{w} \in \Wi{k}(\bm{\beta}) \text{ if and only if } \e{w}' \in
  \Wi{k}(\bm{\beta});  \label{e ww' Wik}\\
&\invi{k}(\e{w}) = \invi{k}(\e{w}'). \label{e ww' invi}
\end{align}
Statement \eqref{e ww' Wik} is checked directly
from the definition~\eqref{eq:Wik-beta}.
Statement \eqref{e ww' invi} is clear if $\e{w}$ and
$\e{w}'$ are related by a switch of the form \eqref{e assaf
  switchboard 1}.
Suppose that $\{\e{w}, \e{w}'\} = \{\e{ybacz}, \e{yacbz}\}$ for $c-a \le k$ and
$a<b<c$ and (sub)words $\e{y}, \e{z}$ (the second case of \eqref{e assaf
  switchboard 2} is similar). Then
$\invi{k}(\e{ybacz}) = \invi{k}(\e{yabcz})+1 = \invi{k}(\e{yacbz})$,
which verifies~\eqref{e ww' invi}.

The uniqueness follows from
Proposition~\ref{p assaf ideal switchboards}.
\end{proof}

\begin{definition}
\emph{LLT switchboards} are the special
Assaf switchboards $\Gamma_k(\bm{\beta},t)$ described in
Proposition~\ref{p Sami LLT graphs}.
This terminology is justified by Corollary~\ref{c Sami LLT graphs} below.
\end{definition}

\begin{corollary}
\label{c Sami LLT graphs}
Let $\bm{\beta}$ a  $k$-tuple of skew shapes
in which the shifted contents of cells lie in $\{1,\dots,N\}$.
Then the associated new variant combinatorial LLT polynomial
is the $q$-generating function for the LLT switchboards $\Gamma_k(\bm{\beta},t)$:
\begin{equation}
\label{eq:LLT-Assaf}
\mathcal{G}_{\bm{\beta}}(\mathbf{x};q)
= \sum_t q^tF_{\Gamma_k(\bm{\beta},t)}(\mathbf{x}).
\end{equation}
\end{corollary}

\begin{proof}
This is a direct consequence of Propositions~\ref{p llt formula} and~\ref{p
  Sami LLT graphs}.
\end{proof}


\begin{remark}
Corollary~\ref{c Sami LLT graphs} is a restatement
of the connection between certain D~graphs and LLT polynomials described
in \cite[\textsection4.2]{SamiOct13}.
The switchboard $\bigsqcup_t \Gamma_k(\bm{\beta},t)$ is, after relabeling vertices,
the graph $\G^{(k)}_{c,D}$ defined in
\cite[\textsection4.2]{SamiOct13}, where
$c$ is the content vector
and $D$ is the $k$-descent set corresponding to $\bm{\beta}$ as in
\cite[Equation~4.2]{SamiOct13}
(see also  \cite[Proposition 2.6]{BLamLLT}).
\end{remark}

\begin{example}
\label{example:Assaf-stronger-than-LLT}
Let $k\!=\!3$, $\bm{\beta} \!=\! \left(\partition{&~ }\,,\partition{~}\,,\partition{~&~}\,\right)
\! =\! (2/1, 1, 2)$.
Figure~\ref{f Assaf LLT graph2} shows all nonempty LLT switchboards
$\Gamma_3(\bm{\beta},t)$ 
and their symmetric functions $F_{\Gamma_3(\bm{\beta},t)}$.
Formula~\eqref{eq:LLT-Assaf} then
yields 
\[
\mathcal{G}_{\bm{\beta}}(\mathbf{x};q)=
s_{4} + q\, s_{31} + q^2\, (s_{31}+s_{22}) +  q^3\, s_{211}.
\]
\end{example}

\begin{figure}[ht]
\begin{tikzpicture}
\begin{scope}[xscale = 2.2,yscale = 1.5, yshift = .8in]
\tikzstyle{vertex}=[inner sep=0pt, outer sep=4pt]
\tikzstyle{framedvertex}=[inner sep=3pt, outer sep=4pt, draw=gray]
\tikzstyle{aedge} = [draw, thin, ->,black]
\tikzstyle{edge} = [draw, thick, -,black]
\tikzstyle{doubleedge} = [draw, thick, double distance=1pt, -,black]
\tikzstyle{hiddenedge} = [draw=none, thick, double distance=1pt, -,black]
\tikzstyle{dashededge} = [draw, very thick, dashed, black]
\tikzstyle{LabelStyleH} = [text=black, anchor=south]
\tikzstyle{LabelStyleHn} = [text=black, anchor=north]
\tikzstyle{LabelStyleV} = [text=black, anchor=east]

\node[vertex] (v1) at (0,0){\footnotesize$\e{1235} $};
\node[vertex, anchor=west] (inv0) at (2.4,0){$F_{\Gamma_3(\bm{\beta},0)}=s_4$};
\end{scope}
\begin{scope}[xscale = 2.2,yscale = 1.5, yshift = .4in]
\tikzstyle{vertex}=[inner sep=0pt, outer sep=4pt]
\tikzstyle{framedvertex}=[inner sep=3pt, outer sep=4pt, draw=gray]
\tikzstyle{aedge} = [draw, thin, ->,black]
\tikzstyle{edge} = [draw, thick, -,black]
\tikzstyle{doubleedge} = [draw, thick, double distance=1pt, -,black]
\tikzstyle{hiddenedge} = [draw=none, thick, double distance=1pt, -,black]
\tikzstyle{dashededge} = [draw, very thick, dashed, black]
\tikzstyle{LabelStyleH} = [text=black, anchor=south]
\tikzstyle{LabelStyleHn} = [text=black, anchor=north]
\tikzstyle{LabelStyleV} = [text=black, anchor=east]

\node[vertex] (v1) at (0,0){\footnotesize$\e{1325} $};
\node[vertex] (v2) at (-1,0){\footnotesize$\e{2135} $};
\node[vertex] (v3) at (1,0){\footnotesize$\e{1253} $};
\node[vertex, anchor=west] (inv1) at (2.4,0){$F_{\Gamma_3(\bm{\beta},1)}=s_{31}$};
\draw[edge] (v1) to node[LabelStyleH]{\Tiny$\tilde{2} $} (v2);
\draw[edge] (v1) to node[LabelStyleH]{\Tiny$\tilde{3} $} (v3);
\end{scope}

\begin{scope}[xscale = 2.2,yscale = 1.5]
\tikzstyle{vertex}=[inner sep=0pt, outer sep=4pt]
\tikzstyle{framedvertex}=[inner sep=3pt, outer sep=4pt, draw=gray]
\tikzstyle{aedge} = [draw, thin, ->,black]
\tikzstyle{edge} = [draw, thick, -,black]
\tikzstyle{doubleedge} = [draw, thick, double distance=1pt, -,black]
\tikzstyle{hiddenedge} = [draw=none, thick, double distance=1pt, -,black]
\tikzstyle{dashededge} = [draw, very thick, dashed, black]
\tikzstyle{LabelStyleH} = [text=black, anchor=south]
\tikzstyle{LabelStyleHn} = [text=black, anchor=north]
\tikzstyle{LabelStyleV} = [text=black, anchor=east]

\end{scope}

\begin{scope}[xscale = 2.2,yscale = 1.5, xshift = 0in]
\tikzstyle{vertex}=[inner sep=0pt, outer sep=4pt]
\tikzstyle{framedvertex}=[inner sep=3pt, outer sep=4pt, draw=gray]
\tikzstyle{aedge} = [draw, thin, ->,black]
\tikzstyle{edge} = [draw, thick, -,black]
\tikzstyle{doubleedge} = [draw, thick, double distance=1pt, -,black]
\tikzstyle{hiddenedge} = [draw=none, thick, double distance=1pt, -,black]
\tikzstyle{dashededge} = [draw, very thick, dashed, black]
\tikzstyle{LabelStyleH} = [text=black, anchor=south]
\tikzstyle{LabelStyleHn} = [text=black, anchor=north]
\tikzstyle{LabelStyleV} = [text=black, anchor=east]

\node[vertex] (v1) at (.5,0){\footnotesize$\e{2513} $};
\node[vertex] (v2) at (1.5,0){\footnotesize$\e{2153} $};
\node[vertex, anchor=west] (inv2) at (2.4,0){$F_{\Gamma_3(\bm{\beta},2)}=s_{31}+s_{22}$};

\draw[doubleedge] (v1) to node[LabelStyleH]{\Tiny$2 $} (v2);
\draw[hiddenedge] (v1) to node[LabelStyleHn]{\Tiny$3 $} (v2);

\node[vertex] (v01) at (-.5,0){\footnotesize$\e{2351} $};
\node[vertex] (v02) at (-1.5,0){\footnotesize$\e{2315} $};
\node[vertex] (v03) at (-2.5,0){\footnotesize$\e{3125} $};
\draw[edge] (v01) to node[LabelStyleH]{\Tiny$3 $} (v02);
\draw[edge] (v02) to node[LabelStyleH]{\Tiny$\tilde{2} $} (v03);

\end{scope}

\begin{scope}[xscale = 2.2,yscale = 1.5, yshift = -.4in]
\tikzstyle{vertex}=[inner sep=0pt, outer sep=4pt]
\tikzstyle{framedvertex}=[inner sep=3pt, outer sep=4pt, draw=gray]
\tikzstyle{aedge} = [draw, thin, ->,black]
\tikzstyle{edge} = [draw, thick, -,black]
\tikzstyle{doubleedge} = [draw, thick, double distance=1pt, -,black]
\tikzstyle{hiddenedge} = [draw=none, thick, double distance=1pt, -,black]
\tikzstyle{dashededge} = [draw, very thick, dashed, black]
\tikzstyle{LabelStyleH} = [text=black, anchor=south]
\tikzstyle{LabelStyleHn} = [text=black, anchor=north]
\tikzstyle{LabelStyleV} = [text=black, anchor=east]

\node[vertex] (v1) at (0,0){\footnotesize$\e{3251} $};
\node[vertex] (v2) at (-1,0){\footnotesize$\e{2531} $};
\node[vertex] (v3) at (1,0){\footnotesize$\e{3215} $};
\node[vertex, anchor=west] (inv3) at (2.4,0){$F_{\Gamma_3(\bm{\beta},3)}=s_{211}$};
\draw[edge] (v1) to node[LabelStyleH]{\Tiny$\tilde{2} $} (v2);
\draw[edge] (v1) to node[LabelStyleH]{\Tiny$3 $} (v3);
\end{scope}

\end{tikzpicture}
\caption{\label{f Assaf LLT graph2}
LLT switchboards $\Gamma_3(\bm{\beta}, t)$ for
$\bm{\beta} 
= (2/1, 1, 2)$,
and the corresponding symmetric functions $F_{\Gamma_3(\bm{\beta},t)}$.
The two connected
components of $\Gamma_3(\bm{\beta}, 2)$ have symmetric functions $s_{31}$ (on the left) and $s_{22}$ (on the right).
}
\end{figure}

\pagebreak[3]


\begin{example}
Let $k=3$, $\bm{\beta}
= \left(\partition{& &~\\& &}\,,
\partition{&~&~\\&~&~}\,,
\partition{&~\\ & }\,\right) = (3/2, 33/11, 2/1)$.
There are four non\-empty LLT switchboards
$\Gamma_3(\bm{\beta},t)$, for $t=2,3,4,5$.
(One of them is shown in Figure~\ref{f Assaf LLT graph}.)
The corresponding symmetric functions are:
\[
F_{\Gamma_3(\bm{\beta},2)} =s_{42}, \quad
F_{\Gamma_3(\bm{\beta},3)} =s_{33}+s_{321}, \quad
F_{\Gamma_3(\bm{\beta},4)} =s_{321} + s_{222}, \quad
F_{\Gamma_3(\bm{\beta},5)} =s_{2211}.
\]
Hence by Corollary \ref{c Sami LLT graphs},
the LLT polynomial $\mathcal{G}_{\bm{\beta}}(\mathbf{x};q)$ is given by
\[
\mathcal{G}_{\bm{\beta}}(\mathbf{x};q)=
q^2 s_{42} + q^3 (s_{33}+s_{321}) +  q^4 (s_{321} + s_{222}) + q^5 s_{2211}.
\]
\end{example}

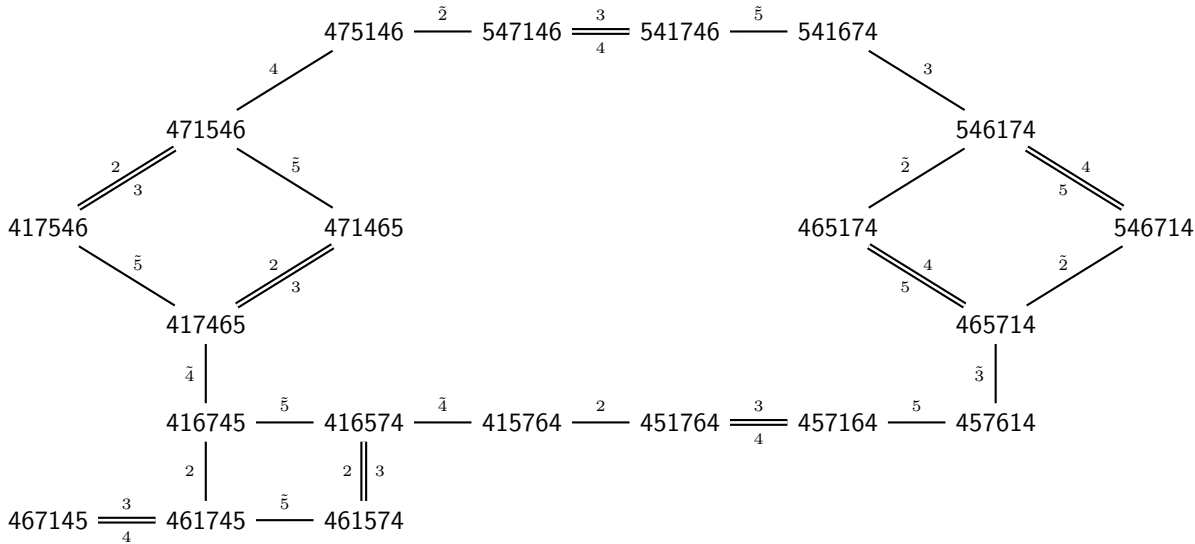
\begin{figure}[ht]
\centerfloat
\begin{tikzpicture}[xscale = 2.1,yscale = 1.3]
\tikzstyle{vertex}=[inner sep=0pt, outer sep=4pt]
\tikzstyle{framedvertex}=[inner sep=3pt, outer sep=4pt, draw=gray]
\tikzstyle{aedge} = [draw, thin, ->,black]
\tikzstyle{edge} = [draw, thick, -,black]
\tikzstyle{doubleedge} = [draw, thick, double distance=1pt, -,black]
\tikzstyle{hiddenedge} = [draw=none, thick, double distance=1pt, -,black]
\tikzstyle{dashededge} = [draw, very thick, dashed, black]
\tikzstyle{LabelStyleH} = [text=black, anchor=south]
\tikzstyle{LabelStyleHn} = [text=black, anchor=north]
\tikzstyle{LabelStyleV} = [text=black, anchor=east]
\tikzstyle{LabelStyleVn} = [text=black, anchor=west]
\tikzstyle{LabelStyleNEn} = [text=black, anchor=north east, inner sep=2pt]
\tikzstyle{LabelStyleNE} = [text=black, anchor=south west, inner sep=2pt]
\tikzstyle{LabelStyleNWn} = [text=black, anchor=north west, inner sep=2pt]
\tikzstyle{LabelStyleNW} = [text=black, anchor=south east, inner sep=2pt]

\begin{scope}
\node[vertex] (v1) at (7,3){\footnotesize$\e{465714} $};
\node[vertex] (v2) at (8,4){\footnotesize$\e{546714} $};
\node[vertex] (v3) at (7,2){\footnotesize$\e{457614} $};
\node[vertex] (v4) at (6,4){\footnotesize$\e{465174} $};
\node[vertex] (v5) at (7,5){\footnotesize$\e{546174} $};
\node[vertex] (v6) at (3,1){\footnotesize$\e{461574} $};
\node[vertex] (v7) at (3,2){\footnotesize$\e{416574} $};
\node[vertex] (v8) at (6,6){\footnotesize$\e{541674} $};
\node[vertex] (v9) at (6,2){\footnotesize$\e{457164} $};
\node[vertex] (v10) at (5,2){\footnotesize$\e{451764} $};
\node[vertex] (v11) at (4,2){\footnotesize$\e{415764} $};
\node[vertex] (v12) at (1,1){\footnotesize$\e{467145} $};
\node[vertex] (v13) at (2,1){\footnotesize$\e{461745} $};
\node[vertex] (v14) at (2,2){\footnotesize$\e{416745} $};
\node[vertex] (v15) at (3,4){\footnotesize$\e{471465} $};
\node[vertex] (v16) at (2,3){\footnotesize$\e{417465} $};
\node[vertex] (v17) at (4,6){\footnotesize$\e{547146} $};
\node[vertex] (v18) at (3,6){\footnotesize$\e{475146} $};
\node[vertex] (v19) at (5,6){\footnotesize$\e{541746} $};
\node[vertex] (v20) at (2,5){\footnotesize$\e{471546} $};
\node[vertex] (v21) at (1,4){\footnotesize$\e{417546} $};
\draw[edge] (v1) to node[LabelStyleNW]{\Tiny$\tilde{2} $} (v2);
\draw[doubleedge] (v1) to node[LabelStyleNE]{\Tiny$4 $} (v4);
\draw[hiddenedge] (v1) to node[LabelStyleNEn]{\Tiny$5 $} (v4);
\draw[edge] (v1) to node[LabelStyleV]{\Tiny$\tilde{3} $} (v3);
\draw[doubleedge] (v2) to node[LabelStyleNE]{\Tiny$4 $} (v5);
\draw[hiddenedge] (v2) to node[LabelStyleNEn]{\Tiny$5 $} (v5);
\draw[edge] (v3) to node[LabelStyleH]{\Tiny$5 $} (v9);
\draw[edge] (v4) to node[LabelStyleNW]{\Tiny$\tilde{2} $} (v5);
\draw[edge] (v5) to node[LabelStyleNE]{\Tiny$3 $} (v8);
\draw[doubleedge] (v6) to node[LabelStyleV]{\Tiny$2 $} (v7);
\draw[hiddenedge] (v6) to node[LabelStyleVn]{\Tiny$3 $} (v7);
\draw[edge] (v6) to node[LabelStyleH]{\Tiny$\tilde{5} $} (v13);
\draw[edge] (v7) to node[LabelStyleH]{\Tiny$\tilde{5} $} (v14);
\draw[edge] (v7) to node[LabelStyleH]{\Tiny$\tilde{4} $} (v11);
\draw[edge] (v8) to node[LabelStyleH]{\Tiny$\tilde{5} $} (v19);
\draw[doubleedge] (v9) to node[LabelStyleH]{\Tiny$3 $} (v10);
\draw[hiddenedge] (v9) to node[LabelStyleHn]{\Tiny$4 $} (v10);
\draw[edge] (v10) to node[LabelStyleH]{\Tiny$2 $} (v11);
\draw[doubleedge] (v12) to node[LabelStyleH]{\Tiny$3 $} (v13);
\draw[hiddenedge] (v12) to node[LabelStyleHn]{\Tiny$4 $} (v13);
\draw[edge] (v13) to node[LabelStyleV]{\Tiny$2 $} (v14);
\draw[edge] (v14) to node[LabelStyleV]{\Tiny$\tilde{4} $} (v16);
\draw[doubleedge] (v15) to node[LabelStyleNW]{\Tiny$2 $} (v16);
\draw[hiddenedge] (v15) to node[LabelStyleNWn]{\Tiny$3 $} (v16);
\draw[edge] (v15) to node[LabelStyleNE]{\Tiny$\tilde{5} $} (v20);
\draw[edge] (v16) to node[LabelStyleNE]{\Tiny$\tilde{5} $} (v21);
\draw[doubleedge] (v17) to node[LabelStyleH]{\Tiny$3 $} (v19);
\draw[hiddenedge] (v17) to node[LabelStyleHn]{\Tiny$4 $} (v19);
\draw[edge] (v17) to node[LabelStyleH]{\Tiny$\tilde{2} $} (v18);
\draw[edge] (v18) to node[LabelStyleNW]{\Tiny$4 $} (v20);
\draw[doubleedge] (v20) to node[LabelStyleNW]{\Tiny$2 $} (v21);
\draw[hiddenedge] (v20) to node[LabelStyleNWn]{\Tiny$3 $} (v21);
\end{scope}
\end{tikzpicture}
\caption{\label{f Assaf LLT graph}
The LLT switchboard 
$\Gamma_k(\bm{\beta}, t)$,
for $k=3$, $t=3$,
and $\bm{\beta} = (3/2, 33/11, 2/1)$.
The associated symmetric function is $F_{\Gamma_3(\bm{\beta}, 3)} = s_{33}+s_{321}$.
}
\end{figure}

\begin{definition}
An \emph{Assaf symmetric function} of level~$k$
is the symmetric function $F_\Gamma(\mathbf{x})$ associated to an 
Assaf switchboard~$\Gamma$ of level~$k$.
\end{definition}

\begin{conjecture}
\label{cj sami}
Assaf symmetric functions are Schur positive.
\end{conjecture}

Conjecture~\ref{cj sami} is implicit in the work of Assaf,
as is its slightly stronger variant
Conjecture~\ref{conj:monomial-positivity-modIk}.
(We opted for the latter variant in Section \ref{s main results}
primarily because it was much easier to formulate.)
By Corollary~\ref{c Sami LLT graphs}, the coefficient of $q^t$
in an LLT polynomial is an Assaf symmetric function associated to a special
Assaf switchboard~$\Gamma_k(\bm{\beta},t)$, so
Conjecture~\ref{cj sami}
is a stronger statement than Schur positivity of LLT polynomials.
In a sense, it is strictly stronger,
as Example~\ref{example:Assaf-stronger-than-LLT} illustrates:
the LLT switchboard $\Gamma_3(\bm{\beta},2)$ in Figure~\ref{f Assaf LLT graph2}
is disconnected,
so $F_{\Gamma_3(\bm{\beta},2)}$ is a sum of two nonzero Assaf
symmetric functions.

We conclude this section by recalling from  \cite{BLamLLT}
the main positive result of this approach to Schur positivity of LLT
and Macdonald polynomials:
a combinatorial description of the coefficients
$\mathfrak{c}^\lambda_{\bm{\beta}}(q)$ (see~\eqref{eq:LLT-LR})
in the special case $k=3$.
Before doing so, we review earlier related work.

The LLT polynomial for $k=1$ is nothing but the skew
Schur function for the corresponding shape.
For $k=2$, an explicit combinatorial rule for the coefficients
$\mathfrak{c}^\lambda_{\bm{\beta}}(q)$ was stated by Carr\'{e}
and Leclerc~\cite{CL95}; its proof was completed by van
Leeuwen~\cite{vL00}, cf.~\cite[\textsection9]{HHL}.
Relatedly, Fishel~\cite{Fis95} gave the first combinatorial interpretation for the
coefficient of  $s_\lambda(\mathbf{x})$ in the transformed Macdonald
polynomial $\tilde{H}_\mu(\mathbf{x};q,t)$ in the case when $\mu$ has 2 columns
(using rigged configurations). Alternative formulas for this case were
given in~\cite{LM03, Zab99}.

Conjecture \ref{cj sami} is known to hold in the following special
cases (the results below are stronger than their counterparts in the
previous paragraph, but came later):
for  $k=1$, the Assaf ideal  $\Iassaf{k}$ contains the plactic ideal
(cf.\ Definition \ref{def:Iplac}), hence for any connected component  $\Gamma$
of an Assaf switchboard of level 1, the symmetric function  $F_\Gamma$
is just a Schur function, cf.\ Proposition~\ref{p plactic}.
Assaf \cite[Theorem 4.9]{SamiOct13} showed that this is also true for
connected Assaf switchboards of level~2.
Hence Conjecture \ref{cj sami} holds for  $k\in\{1,2\}$.
Roberts \cite[Theorem 4.11]{RobertsDgraph} extended the work of Assaf
to a setting that contains the  $k=2$ case,
by proving that the symmetric function of any connected Assaf
switchboard contained in $\Gamma_k(\bm{\beta},t)$ is a Schur function
whenever
$\,\max_{i\in\ZZ}
|C(\bm{\beta}) \cap [i,i+k]| \le 3\,$
where $C(\bm{\beta})$~is the set of distinct shifted contents of the
cells of $\bm{\beta}$ (cf.~\eqref{e shifted content}).
Assaf \cite[Theorem~4.10]{SamiOct13} and Novelli-Schilling
\cite[Theorem~3]{NS} gave an explicit positive formula for the Schur
expansion of the Assaf symmetric functions in the case $k \ge N-1$;
note that in this case,
Assaf switchboards have only rotation switches.

To state the main result of~\cite{BLamLLT}, we will need a couple of auxiliary notions.
For a partition~$\lambda$, let $\RSST(\lambda; N)$ denote the set  of
semistandard Young tableaux~$T$ of shape  $\lambda$ and entries in
$\{1,\dots,N\}$
satisfying the following constraints:
\begin{itemize}
\item
the entries strictly increase across the rows and down the columns;
\item
the entries increase in increments of at least~$3$ along diagonals.
\end{itemize}
For $T\in\RSST(\lambda; N)$,
we produce a word~$\sqread(T)\in\U^*$ by reading the entries of~$T$
in the following order.
Let us circle each entry~$c$ of $T$ such that the entry immediately
west of it is $c-1$.
The word $\sqread(T)$ is then obtained by reading the diagonals of $T$
one by one, starting from the southwest
corner and finishing at the northeast corner. In each diagonal, we first
read the circled entries going northwest, then the uncircled entries
going southeast.\linebreak[3]
To illustrate, the tableau
{\setlength{\cellsize}{2.55ex}
\[
T=
{ \text{\footnotesize $\tableau{
1&2&4&6\\3&4&5&7\\ 5&6&8&9
}$}
}
\text{\ \ has circled entries \, }
{\text{\footnotesize  $\tableau{
1&\encircle{2}&4&6\\3&\encircle{4}&\encircle{5}&7\\ 5&\encircle{6}&8&\encircle{9}
}$}
}
\text{ \ \ and $\sqread(T)=\e{563418952476}$.}
\]
\vspace{-3mm}
}

\begin{theorem}[{\cite[Theorem 1.1]{BLamLLT}}] \label{t sqread}
For any partition $\lambda$, the noncommutative Schur function
$\mathfrak{J}_\lambda(\mathbf{u})$ is $\ZZ$-monomial positive modulo
$\Jlam{3}$.  A monomial expansion is given by
\begin{align}
\mathfrak{J}_\lambda(\mathbf{u}) \equiv \sum_{T \in \RSST(\lambda; N)} \mathbf{u}_{\sqread(T)} \ \bmod \Jlam{3}. \label{e sqread}
\end{align}
\end{theorem}

Theorem \ref{t sqread} combined with Proposition \ref{p llt formula}
yields the explicit combinatorial formula
\begin{align}
\label{e qLRcoefs}
\mathfrak{c}_{\bm{\beta}}^\lambda(q)
= \sum_{\substack{T\in\RSST(\lambda; N) \\ \sqread(T)\in\Wi{3}(\bm{\beta})}}
q^{\invi{3}(\sqread(T))}
\end{align}
for the coefficients in the Schur
expansion of an LLT polynomial indexed by a 3-tuple of skew shapes;
see \cite[Corollary 4.3]{BLamLLT} for further details.
Cf.\ also Conjecture~\ref{cj sqread}.

\pagebreak[3]

\section{Plactic, niplactic, and beyond}
\label{s bijectivizations and examples}

In this section, we recast some of the key results obtained in~\cite{FG}
in the language of ideals and switchboards.
To this end, we introduce some notation.
First, we set
\begin{equation}
\label{eq:Ifg}
\Ifg=\Ifgp+\Irk\,.
\end{equation}
Equivalently, $\Ifg$~is
the ideal in $\U$ generated by the elements
\eqref{i knuth1}--\eqref{i knuth2}
and {\eqref{i fg close}--\eqref{i fg far2}}.

Let $\SSYT(\lambda; N)$ denote the set of semistandard Young tableaux
of shape $\lambda$ and entries from $\{1,\dots,N\}$.
The \emph{column reading word} $\creading(T)\in\U^*$
of a tableau $T \in \SSYT(\lambda; N)$
is the word obtained by concatenating the columns of~$T$ (reading
each column bottom to top), starting with the leftmost column.
To illustrate, the tableau
\[
T=\tableau{1&1&2&2 \\2&2&3\\5&6&7}
\]
of shape $\lambda=433$ has column reading word $\creading(T)=\e{5216217322}$.

\begin{theorem}[\cite{FG}]
\label{t FG}
The noncommutative Schur functions $\mathfrak{J}_\lambda(\mathbf{u})$
are $\ZZ$-monomial positive modulo the ideal~$\Ifg\,$.
Explicitly,
\begin{equation}
\mathfrak{J}_\lambda(\mathbf{u}) \equiv \sum_{T \in \SSYT(\lambda; N)}
\mathbf{u}_{\creading(T)} \ \bmod \Ifg \,.
\end{equation}
\end{theorem}

Recall that Theorem~\ref{th:FG'-positivity} of this paper
strengthens Theorem~\ref{t FG} by replacing $\Ifg$ by~$\Ifgp$.



\begin{definition}
\label{d FG switchboards}
A switchboard~$\Gamma$ is called \emph{rotation-free}
if it has no rotation switches.
That is, every switch in~$\Gamma$ is Knuth, braid, or idempotent,
see Definition~\ref{def-switchboard} and Figure~\ref{f switches}.
\end{definition}

Rotation-free switchboards are directly related to the ideal~$\Ifg\,$:
it is not hard to check that a $(0,1)$-vector in $\Ifgperp$ is the same as
the sum of vertices of a rotation-free switchboard.
Combining this with Theorem \ref{t FG}
and Proposition~\ref{Icomm-Schur}, we obtain the following corollary.

\begin{corollary} \label{c FG}
For a rotation-free switchboard~$\Gamma$, the symmetric function
$F_\Gamma(\mathbf{x})$  is Schur positive.
The coefficient of $s_\lambda(\mathbf{x})$ in  $F_\Gamma(\mathbf{x})$ is
the number of tableaux $T\in \SSYT(\lambda;N)$ such that
the column reading word $\creading(T)$ appears as a vertex in~$\Gamma$.
\end{corollary}


Several important families of symmetric functions, some of which are reviewed below
(see \cite{FG} for additional examples),
arise from ideals containing~$\Ifg\,$.
For each of these families,
Corollary~\ref{c FG} provides a manifestly positive combinatorial rule
for the coefficients in the corresponding Schur expansions.

The following classical construction goes back to Lascoux and Sch\"utzenberger~\cite{LS}.

\begin{definition} 
\label{def:Iplac}
The \emph{plactic ideal} $\Iplac \subset \U$ is 
generated by the ``Knuth elements''
\begin{alignat}{2}
&u_au_cu_b - u_cu_au_b \qquad && \text{for $a\le b<c$;} \label{e plactic 1}\\
&u_bu_au_c - u_bu_cu_a && \text{for $a<b\le c$.}\label{e plactic 2}
\end{alignat}
It is easy to see that $\Iplac \supset \Ifg\,$.
\end{definition}

\begin{definition} 
\label{def:plactic-switchboards}
A switchboard is called \emph{plactic} if all its switches are Knuth.
A connected plactic switchboard has the words in a Knuth (=\,plactic) equivalence class as its vertices,
and Knuth switches as its edges.
For a semistandard Young tableau~$T$, we denote by $\Gamma_T$ the switchboard
whose vertex set consists of words with insertion tableau~$T$.
\end{definition}

The $(0,1)$-vectors in $\Iplacperp$ are the same as sums of distinct Knuth equivalence classes.

The switchboards  $\Gamma_T$ have been studied implicitly
in the huge body of work on the
Robinson-Schensted-Knuth correspondence and jeu de taquin.
They were studied explicitly in the context of $W$-graphs~\cite{KL,
  StembridgeWGraph}, and
their combinatorial structure was thoroughly investigated in~\cite{SamiOct13,RobertsDgraph}.
It is well known that the edge-labeled graph~$\Gamma_T$,
after forgetting vertex labels,
depends only on the shape of~$T$ and not on its entries;
these edge-labeled graphs (with slightly different conventions) are the
\emph{standard dual equivalence graphs} of~\cite{SamiOct13}.

The symmetric functions associated with connected plactic switchboards are
nothing but Schur functions.
This is just a restatement of the well known expansion of a Schur polynomial in terms of Gessel's
fundamental quasisymmetric functions:

\begin{proposition}[{\cite{GesselPPartition}, \cite[Chapter~7]{St}}]
\label{p plactic}
For $T\in\SSYT(\lambda;N)$,
we have $F_{\Gamma_T}(\mathbf{x}) \!=\!s_\lambda(\mathbf{x})$.
\end{proposition}

We next discuss the Stanley symmetric functions, also known as stable Schubert polynomials.
They are related to the following ideal~\cite{FS}.

\begin{definition}
\label{d nilCox}
The \emph{nilCoxeter ideal} $\Incox$ of  $\U$  is generated by the elements
\begin{alignat}{3}
&u_a^2 \qquad&& \text{for all~$a$;}\\
&u_au_c - u_cu_a \qquad&& \text{for $c - a \geq 2$;}\\
&u_au_{b}u_a - u_bu_au_b \qquad&& \text{for $b-a=1$.}
\end{alignat}
It is easy to see that $\Incox \supset \Ifg\,$.
\end{definition}

The monomials $\mathbf{u}_{\e{w}}\notin\Incox$ 
are labeled by \emph{reduced words} $\e{w}$ for the symmetric group~$\mathcal{S}_{N+1}$
(viewed as words in~$\U^*$).
Moreover $\mathbf{u}_{\e{v}}\equiv\mathbf{u}_{\e{w}}\bmod \Incox$ if and only if
$\e{v}$ and $\e{w}$ are reduced words for the same permutation~$\pi\in\mathcal{S}_{N+1}$.

We denote by $\operatorname{Red}(\pi)$ the set of reduced words for
$\pi\in\mathcal{S}_{N+1}$, and set
\[
\gamma(\pi) = \sum_{\e{w} \in \text{Red}(\pi)} \e{w} \in \U^*.
\]
The $(0,1)$-vectors  $\gamma(\pi)$ form a  $\ZZ$-basis of $\Incoxperp$.



The \emph{Stanley symmetric function} 
associated to the permutation $\pi\in\mathcal{S}_{N+1}$
is defined by
\begin{align}
F_{\gamma(\pi)}(\mathbf{x}) = \sum_{\e{w} \in \text{Red}(\pi)} Q_{\Des(\e{w})}(\mathbf{x}).
\end{align}
By Corollary \ref{c FG},  $F_{\gamma(\pi)}(\mathbf{x})$
is a Schur positive symmetric function~\cite{EdelmanGreene}:

\begin{corollary}
The coefficient of $s_\lambda(\mathbf{x})$ in $F_{\gamma(\pi)}(\mathbf{x})$
is the number of tableaux $T$ of shape~$\lambda$ 
whose column reading word $\creading(T)$ is a reduced word for~$\pi$.
\label{c Stanley answer}
\end{corollary}


\begin{definition}
\label{d I nilplactic}
The \emph{nilplactic ideal} $\Inplac$ of  $\U$ is generated by the elements
\begin{alignat}{2}
&u_a^2 \qquad&& \text{for all~$a$;} \\
&u_a u_c u_a \qquad&& \text{for $c-a\ge 2$;}\\
&u_c u_a u_c  \qquad&& \text{for $c-a\ge 2$;}\\
&u_bu_au_c - u_bu_cu_a \qquad&& \text{for $a<b<c$;}\\
&u_au_cu_b - u_cu_au_b \qquad&& \text{for $a<b<c$;}\\
&u_au_bu_a - u_bu_au_b \qquad&& \text{for $b-a=1$.}
\end{alignat}
It is easy to check that  $\Incox \supset \Inplac \supset \Ifg\,$.
A  switchboard $\Gamma$ is called \emph{nilplactic} if
\begin{itemize}
\item all switches in $\Gamma$ are either Knuth switches on 3 distinct letters, or braid~switches;
\item no word  in $\Gamma$ contains the same letter twice in succession;
\item no word  in $\Gamma$ contains 3 consecutive letters $\e{aca}$ or $\e{cac}$ with $c-a\ge 2$.
\end{itemize}
\end{definition}

The vertices of a connected nilplactic switchboard form a nilplactic equivalence class;
they are reduced words for the same permutation.
The $(0,1)$-vectors in $\Inplacperp$ are  sums of distinct niplactic classes.

It is known \cite{EdelmanGreene, LScohomologyring} that the symmetric
function  $F_\Gamma$ associated to any
connected nilplactic switchboard  $\Gamma$ is a
Schur function.
Moreover any connected nilplactic switchboard is isomorphic, as an edge-labeled graph,
to one of the standard dual equivalence graphs.

Nilplactic switchboards give  another way of understanding the Schur
expansions of Stanley symmetric functions.
The set  $\text{Red}(\pi)$ of reduced words for a
permutation~$\pi\in\mathcal{S}_{N+1}$ is the vertex set of a unique
(generally disconnected) nilplactic switchboard ~$\Gamma(\pi)$.
Hence the coefficient of  $s_\lambda(\mathbf{x})$ in
$F_{\gamma(\pi)}(\mathbf{x}) = F_{\Gamma(\pi)}(\mathbf{x})$
is the number of components of  $\Gamma(\pi)$ whose associated
symmetric function is  $s_\lambda(\mathbf{x})$.
An example is given in Figure~\ref{f stanley symmetric graph}.

%
%
%

\begin{figure}[H]
\centerfloat
\begin{tikzpicture}[xscale = 2.2,yscale = 1.5]
\tikzstyle{vertex}=[inner sep=0pt, outer sep=4pt]
\tikzstyle{framedvertex}=[inner sep=3pt, outer sep=2pt, draw=gray]
\tikzstyle{aedge} = [draw, thin, ->,black]
\tikzstyle{edge} = [draw, very thick, -,black]
\tikzstyle{doubleedge} = [draw, very thick, double distance=1.4pt, -,black]
\tikzstyle{hiddenedge} = [draw=none, thick, double distance=1.4pt, -,black]
\tikzstyle{dashededge} = [draw, thick, lightgray]
\tikzstyle{LabelStyleH} = [text=black, anchor=south]
\tikzstyle{LabelStyleHn} = [text=black, anchor=north]
\tikzstyle{LabelStyleV} = [text=black, anchor=east]

\node[framedvertex] (v1) at (-1,1){\footnotesize$\e{4212}$};
\node[vertex] (v2) at (0,1){\footnotesize$\e{4121}$};
\node[vertex] (v3) at (1,1){\footnotesize$\e{1421}$};
\node[vertex] (v4) at (1,0){\footnotesize$\e{1241}$};
\node[vertex] (v5) at (2,0){\footnotesize$\e{1214}$};
\node[framedvertex] (v6) at (3,0){\footnotesize$\e{2124}$};
\node[framedvertex] (v7) at (3,1){\footnotesize$\e{2142}$};
\node[vertex] (v8) at (4,1){\footnotesize$\e{2412}$};

\draw[edge] (v1) to node[LabelStyleH]{\Tiny$\tilde{3}$} (v2);
\draw[edge] (v2) to node[LabelStyleH]{\Tiny$2$} (v3);
\draw[dashededge] (v3) to node[LabelStyleV]{\Tiny$~$} (v4);
\draw[edge] (v4) to node[LabelStyleH]{\Tiny$3$} (v5);
\draw[edge] (v5) to node[LabelStyleH]{\Tiny$\tilde{2}$} (v6);
\draw[dashededge] (v6) to node[LabelStyleV]{\Tiny$~$} (v7);
\draw[doubleedge] (v7) to node[LabelStyleH]{\Tiny$3$} (v8);
\draw[hiddenedge] (v7) to node[LabelStyleHn]{\Tiny$2$} (v8);
\end{tikzpicture}

\caption{\label{f stanley symmetric graph} Let  $\pi = 32154 \in \S_5$.
Without the vertical edges, this is the nilplactic switchboard  $\Gamma(\pi)$
on the vertex set ~$\text{Red}(\pi)$.
Vertical edges correspond to relations  $\e{ac} \sim \e{ca}$ which are not switches.
The associated symmetric function is the Stanley symmetric function
$F_{\Gamma(\pi)} = s_{31}+s_{22}+s_{211}$.
By Corollary~\ref{c Stanley answer},
the Schur expansion of   $F_{\Gamma(\pi)}$ is determined
by the reduced $\boxed{\text{words}}$ which are column reading words of tableaux.}
\end{figure}
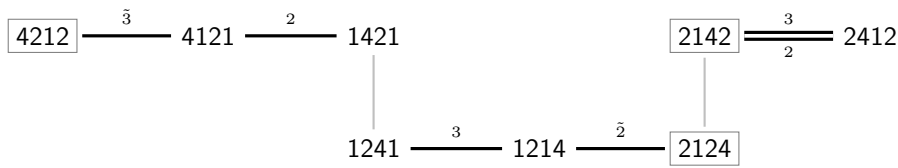

\begin{definition} 
\label{d weak 0 Hecke}
Let $\Iweakh$ be the ideal in $\U$ generated by the elements
\begin{alignat}{2}
&u_bu_au_c - u_bu_cu_a \qquad&& \text{for $a<b<c$;}\\
&u_au_cu_b - u_cu_au_b \qquad&& \text{for $a<b<c$;}\\
&u_au_bu_a - u_bu_au_b && \text{for  $b-a=1$;} \\
&u_bu_bu_a - u_bu_au_a \qquad && \text{for $b-a=1$.}
\end{alignat}
\end{definition}

\pagebreak[3]

It is easy to check that  $\Inplac\supset \Iweakh \supset \Ifg\,$.
Applying Corollary \ref{c FG} to certain switchboards associated with~$\Iweakh$ yields the
Schur expansions of \emph{stable Grothendieck polynomials},
up to a predictable sign; see \cite[Corollary~4.2]{FG} for details.
We note that connected components of these switchboards
do not have to be standard dual equivalence graphs;
this sets them apart from connected plactic or nilplactic switchboards.
An example is given in Figure~\ref{f weak 0 Hecke}.

\begin{figure}[ht]

\vspace{-.05in}

        \centerfloat
\begin{tikzpicture}[xscale = 2.2,yscale = 1.1]
\tikzstyle{vertex}=[inner sep=3pt, outer sep=4pt]
\tikzstyle{framedvertex}=[inner sep=3pt, outer sep=4pt, draw=gray]
\tikzstyle{aedge} = [draw, thin, ->,black]
\tikzstyle{edge} = [draw, thick, -,black]
\tikzstyle{doubleedge} = [draw, thick, double distance=1pt, -,black]
\tikzstyle{hiddenedge} = [draw=none, thick, double distance=1pt, -,black]
\tikzstyle{dashededge} = [draw, dashed, gray]
\tikzstyle{LabelStyleH} = [text=black, anchor=south]
\tikzstyle{LabelStyleHn} = [text=black, anchor=north]
\tikzstyle{LabelStyleV} = [text=black, anchor=east]

\node[framedvertex] (v1) at (0,0){\footnotesize$\e{2112}$};
\node[vertex] (v2) at (1,0){\footnotesize$\e{2212}$};
\node[framedvertex] (v3) at (2,0){\footnotesize$\e{2121}$};
\node[vertex] (v4) at (3,0){\footnotesize$\e{1211}$};
\node[vertex] (v5) at (4,0){\footnotesize$\e{1221}$};

\draw[edge] (v1) to node[LabelStyleH]{\Tiny$\tilde{2}$} (v2);
\draw[edge] (v2) to node[LabelStyleH]{\Tiny$\tilde{3}$} (v3);
\draw[edge] (v3) to node[LabelStyleH]{\Tiny$\tilde{2}$} (v4);
\draw[edge] (v4) to node[LabelStyleH]{\Tiny$\tilde{3}$} (v5);

\end{tikzpicture}

\vspace{.1in}

\begin{tikzpicture}[xscale = 2.2,yscale = 1.1]
\tikzstyle{vertex}=[inner sep=3pt, outer sep=4pt]
\tikzstyle{framedvertex}=[inner sep=3pt, outer sep=4pt, draw=gray]
\tikzstyle{aedge} = [draw, thin, ->,black]
\tikzstyle{edge} = [draw, thick, -,black]
\tikzstyle{doubleedge} = [draw, double, -,black]
\tikzstyle{hiddenedge} = [draw=none, double, -,black]
\tikzstyle{dashededge} = [draw, dashed, gray]
\tikzstyle{LabelStyleH} = [text=black, anchor=south]
\tikzstyle{LabelStyleHn} = [text=black, anchor=north]
\tikzstyle{LabelStyleV} = [text=black, anchor=east]

\node[framedvertex] (v1) at (0,0){\footnotesize$\e{2122}$};
\node[vertex] (v2) at (1,0){\footnotesize$\e{1212}$};
\node[vertex] (v3) at (2,0){\footnotesize$\e{1121}$};

\draw[edge] (v1) to node[LabelStyleH]{\Tiny$\tilde{2}$} (v2);
\draw[edge] (v2) to node[LabelStyleH]{\Tiny$\tilde{3}$} (v3);

\end{tikzpicture}

\vspace{.1in}

\begin{tikzpicture}[xscale = 2.2,yscale = 1.1]
\tikzstyle{vertex}=[inner sep=3pt, outer sep=4pt]
\tikzstyle{framedvertex}=[inner sep=3pt, outer sep=4pt, draw=gray]
\tikzstyle{aedge} = [draw, thin, ->,black]
\tikzstyle{edge} = [draw, thick, -,black]
\tikzstyle{doubleedge} = [draw, thick, double distance=1pt, -,black]
\tikzstyle{hiddenedge} = [draw=none, thick, double distance=1pt, -,black]
\tikzstyle{dashededge} = [draw, dashed, gray]
\tikzstyle{LabelStyleH} = [text=black, anchor=south]
\tikzstyle{LabelStyleHn} = [text=black, anchor=north]
\tikzstyle{LabelStyleV} = [text=black, anchor=east]

\node[framedvertex] (v1) at (0,0){\footnotesize$\e{2111}$};
\node[vertex] (v2) at (1,0){\footnotesize$\e{2211}$};
\node[vertex] (v3) at (2,0){\footnotesize$\e{2221}$};

\draw[edge] (v1) to node[LabelStyleH]{\Tiny$\tilde{2}$} (v2);
\draw[edge] (v2) to node[LabelStyleH]{\Tiny$\tilde{3}$} (v3);

\end{tikzpicture}

\caption{\label{f weak 0 Hecke}
Each component (hence their union) is a switchboard whose vertex sum lies in~$\Iweakhperp$.
Their symmetric functions are (top to bottom): $s_{31}+s_{22}$, $s_{31}$, $s_{31}$.
By Corollary~\ref{c FG}, these Schur expansions can be read off from the
$\boxed{\text{vertices}}$ that are column reading words.
}
\end{figure}
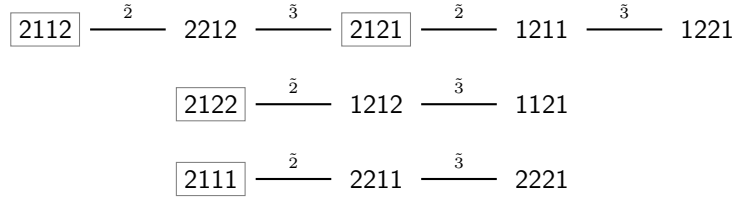

\section{Semimatched ideals}
\label{s semimatched ideals}

\begin{definition} 
\label{d semimatched ideal}
Let us partition the set of $3$-element subsets of $\{1,\dots,N\}$ into three categories,
by declaring each triple~$S\subset\{1,\dots,N\}$ to be
either a \emph{Knuth triple}, or a \emph{rotation triple}, or an \emph{undecided triple}.
Similarly, declare each subset  $\{a, a+1\} \subset \{1,\dots,N\}$ to be a \emph{Knuth pair}, a \emph{braid/idempotent pair}, or an \emph{undecided pair}.
Let $I_M$ be a monomial ideal in~$\U$.
The  \emph{semimatched ideal}~$I$ associated to this data is the ideal generated by~$I_M$,
the switchboard ideal~$\Irk$ (whose generators are
\eqref{i fg close}--\eqref{quad uuu 3vars again}),
and the elements listed below:
\begin{alignat}{2}
&u_a u_c u_b - u_c u_a u_b \qquad &&\text{ if $a < b < c$ is a Knuth triple,}  \\
&u_b u_a u_c - u_b u_c u_a \quad &&\text{ if $a < b < c$ is a Knuth triple,} \\
&u_a u_c u_b - u_b u_a u_c \quad &&\text{ if $a < b < c$ is a rotation triple,}  \\
&u_c u_a u_b - u_b u_c u_a \quad &&\text{ if $a < b < c$ is a rotation triple,} \\
&u_a u_b u_a - u_b u_a u_a \quad &&\text{ if  $b-a=1$ and $\{a, b\}$ is a Knuth pair,} \\
&u_b u_a u_b - u_b u_b u_a \quad &&\text{ if  $b-a=1$ and $\{a, b\}$ is a Knuth pair,} \\
&u_a u_b u_a - u_b u_a u_b \quad &&\text{ if  $b-a=1$ and $\{a, b\}$ is a braid/idempotent pair,} \\
&u_b u_a u_a - u_b u_b u_a \quad &&\text{ if  $b-a=1$ and $\{a, b\}$ is a braid/idempotent pair}.
\end{alignat}
%
%
%
%

For a semimatched ideal~$I$ as above,
an  \emph{$I$-switchboard} is a switchboard in which every vertex $\e{w}$ corresponds to
a monomial $u_\e{w} \notin I_M$,
and every edge corresponds to a switch
that fits one of the patterns listed below:
\begin{alignat}{3}
&\{\e{bac}, \e{bca}\} \ &&\text{or\ }  \{\e{acb}, \e{cab}\}\ \  && \text{for Knuth triples and undecided triples $a < b < c$;} \label{e I switchboard 1}\\
&\{\e{bac}, \e{acb}\} &&\text{or\ }  \{\e{bca}, \e{cab}\} && \text{for rotation triples and undecided triples $a < b < c$;}  \label{e I switchboard 2} \\
&\{\e{aba}, \e{baa}\} &&\text{or\ }  \{\e{bab}, \e{bba}\} && \text{for $b-a>1$;} \label{e I switchboard 3}\\
&\{\e{aba}, \e{baa}\} &&\text{or\ }  \{\e{bab}, \e{bba}\} && \text{for Knuth and undecided pairs $\{a, b\}$ ($b-a=1$);} \label{e I switchboard 4}\\
&\{\e{aba}, \e{bab}\} &&\text{or\ }  \{\e{baa}, \e{bba}\} && \text{for braid/idempotent and undecided pairs ($b-a=1$).}  \label{e I switchboard 5}
\end{alignat}
\end{definition}

\pagebreak[3]

In Figure~\ref{fig:some-ideals}, the following ideals are semimatched:
$\Irk$, $\Ifg$, $\Iaba{k}$, $\Iassaf{k}$, $\Iplac$, $\Iweakh$, and~$\Inplac$.
For example, the ideal  $\Ifg = \Ifgp+\Irk$ (cf.~\eqref{eq:Ifg}) is
the semimatched ideal in which all triples are Knuth, all pairs
undecided, and $I_M = \{0\}$.
The Assaf ideal $\Iassaf{k}$ is the semimatched
ideal with Knuth triples $a < b <c$ for $c-a  > k$,
rotation triples $a < b < c$ for  $c - a \le k$, no undecided triples,
all pairs Knuth, and $I_M$ generated by the monomials in~$\Jlam{k}$.

We observe that Assaf switchboards are the same as $\Iassaf{k}$-switchboards,
rotation-free switchboards are $\Ifg$-switchboards,
plactic switchboards are $\Iplac$-switchboards,
and nilplactic switchboards are $\Inplac$-switchboards.
Hence the notion of  $I$-switchboards allows us to recover
the classes of switchboards studied in Sections~\ref{s LLT}--\ref{s bijectivizations and examples}
from their corresponding ideals.


\begin{proposition}
\label{p semimatched ideal switchboards}
Let  $I$ be a semimatched ideal.
For a set of words $W$ of the same length, the following are equivalent:
\begin{itemize}
\item $\sum_{\e{w}\in W}\e{w} \in I^\perp$;
\item $W$ is the vertex set of an  $I$-switchboard.
\end{itemize}
\end{proposition}


\begin{definition}
\label{d triple biject}
A  \emph{triples ideal} is a semimatched ideal
in which there are no undecided triples, all pairs  $\{a, a+1\}$ are Knuth pairs, and $I_M=\{0\}$.
Hence the data defining a triples ideal is a map  from the set of triples in $\{1,\dots,N\}$
to the 2-element set $\{\text{Knuth, rotation}\}$.
A \emph{triples switchboard} is an $I$-switchboard associated to a triples ideal~$I$.
A \emph{triples $\Dzero$~graph} is a triples switchboard
whose words have no repeated letters;
these were studied in
\cite{BD0graph},
see \cite[Example~2.5 and Definition~7.13]{BD0graph}.
\end{definition}

\begin{example} \label{ex triples}
Let $\Gamma$ be the triples switchboard on the vertex set $\S_5$
corresponding to the triples ideal in which 124, 245, and 345 are
rotation triples and the rest are Knuth triples.
The components of~$\Gamma$ have symmetric functions
$s_{41} + s_{32} + s_{311} + s_{221}$ (see Figure~\ref{f triples}),
$s_{32} + s_{311} + s_{221} + s_{2111}$,
$s_{41} + s_{32}$,
$s_{221} + s_{2111}$,
$s_{5}$,
$s_{41}$ (two components), $s_{32}$~(two),
$s_{311}$~(four),
$s_{221}$~(two),
$s_{2111}$~(two),
and~$s_{11111}$.
\end{example}

\begin{figure}[ht]
\vspace{-.08in}
        \centerfloat
\begin{tikzpicture}[xscale = 1.9,yscale = 1.25]
\tikzstyle{vertex}=[inner sep=0pt, outer sep=4pt]
\tikzstyle{framedvertex}=[inner sep=3pt, outer sep=4pt, draw=gray]
\tikzstyle{aedge} = [draw, thin, ->,black]
\tikzstyle{edge} = [draw, thick, -,black]
\tikzstyle{doubleedge} = [draw, thick, double distance=1pt, -,black]
\tikzstyle{hiddenedge} = [draw=none, thick, double distance=1pt, -,black]
\tikzstyle{dashededge} = [draw, very thick, dashed, black]
\tikzstyle{LabelStyleH} = [text=black, anchor=south]
\tikzstyle{LabelStyleHn} = [text=black, anchor=north]
\tikzstyle{LabelStyleNE} = [text=black, anchor=south west, inner sep=2pt]
\tikzstyle{LabelStyleV} = [text=black, anchor=east]
\tikzstyle{LabelStyleVn} = [text=black, anchor=west]
\tikzstyle{LabelStyleNW} = [text=black, anchor=south east, inner sep=2pt]

\node[vertex] (v1) at (4,3){\footnotesize$\e{34521} $};
\node[vertex] (v2) at (4,2){\footnotesize$\e{35241} $};
\node[vertex] (v3) at (5,2){\footnotesize$\e{32541} $};
\node[vertex] (v4) at (6,3){\footnotesize$\e{34251} $};
\node[vertex] (v5) at (7,4){\footnotesize$\e{32451} $};
\node[vertex] (v6) at (3,2){\footnotesize$\e{35412} $};
\node[vertex] (v7) at (2,2){\footnotesize$\e{43512} $};
\node[vertex] (v8) at (1,2){\footnotesize$\e{43152} $};
\node[vertex] (v9) at (2,4){\footnotesize$\e{13452} $};
\node[vertex] (v10) at (2,5){\footnotesize$\e{13524} $};
\node[vertex] (v11) at (3,6){\footnotesize$\e{31254} $};
\node[vertex] (v12) at (2,6){\footnotesize$\e{13254} $};
\node[vertex] (v13) at (5,4){\footnotesize$\e{34215} $};
\node[vertex] (v14) at (6,5){\footnotesize$\e{32415} $};
\node[vertex] (v15) at (5,6){\footnotesize$\e{34125} $};
\node[vertex] (v16) at (4,6){\footnotesize$\e{31425} $};
\node[vertex] (v17) at (1,6){\footnotesize$\e{13425} $};
\node[vertex] (v18) at (4,5){\footnotesize$\e{32145} $};
\node[vertex] (v19) at (1,4){\footnotesize$\e{31245} $};
\node[vertex] (v20) at (1,5){\footnotesize$\e{13245} $};
\draw[edge] (v1) to node[LabelStyleV]{\Tiny$\tilde{3} $} (v2);
\draw[edge] (v2) to node[LabelStyleH]{\Tiny$2 $} (v3);
\draw[edge] (v2) to node[LabelStyleH]{\Tiny$\tilde{4} $} (v6);
\draw[edge] (v3) to node[LabelStyleNW]{\Tiny$\tilde{3} $} (v4);
\draw[edge] (v4) to node[LabelStyleNE]{\Tiny$4 $} (v13);
\draw[edge] (v4) to node[LabelStyleNW]{\Tiny$2 $} (v5);
\draw[edge] (v5) to node[LabelStyleNE]{\Tiny$4 $} (v14);
\draw[edge] (v6) to node[LabelStyleH]{\Tiny$\tilde{2} $} (v7);
\draw[doubleedge] (v7) to node[LabelStyleH]{\Tiny$3$} (v8);
\draw[hiddenedge] (v7) to node[LabelStyleHn]{\Tiny$4$} (v8);
\draw[edge] (v9) to node[LabelStyleV]{\Tiny$\tilde{4} $} (v10);
\draw[edge] (v10) to node[LabelStyleV]{\Tiny$3 $} (v12);
\draw[edge] (v11) to node[LabelStyleH]{\Tiny$\tilde{4} $} (v16);
\draw[edge] (v11) to node[LabelStyleH]{\Tiny$2 $} (v12);
\draw[edge] (v12) to node[LabelStyleH]{\Tiny$\tilde{4} $} (v17);
\draw[edge] (v13) to node[LabelStyleNW]{\Tiny$2 $} (v14);
\draw[edge] (v14) to node[LabelStyleNE]{\Tiny$\tilde{3} $} (v15);
\draw[edge] (v15) to node[LabelStyleH]{\Tiny$2 $} (v16);
\draw[edge] (v16) to node[LabelStyleV]{\Tiny$\tilde{3} $} (v18);
\draw[edge] (v17) to node[LabelStyleV]{\Tiny$3 $} (v20);
\draw[edge] (v19) to node[LabelStyleV]{\Tiny$2 $} (v20);

\end{tikzpicture}
\vspace{-2mm}
\caption{\label{f triples}
A triples switchboard for the triples ideal
in which 124, 245, and 345 are rotation triples and all other triples are Knuth.
See Example \ref{ex triples}.
}
\end{figure}
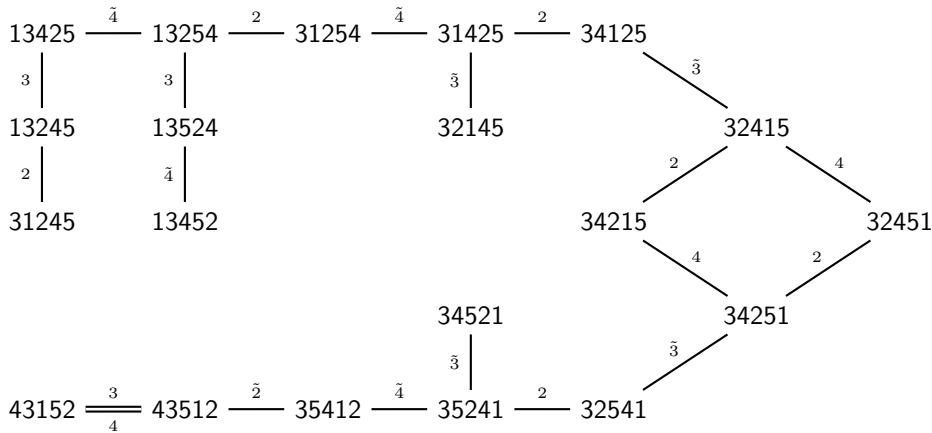
\vspace{-1mm}

\begin{problem}
\label{problem:triples}
Is the symmetric function $F_\Gamma$ Schur positive for any triples switchboard~$\Gamma$?
\end{problem}

If $\Gamma$ is a $\Dzero$~graph with  $N \le 5$, then $F_\Gamma$ is
Schur positive by \cite[Corollary 4.9]{BD0graph}.
Extensive computer tests provide supporting evidence for Schur positivity for $N\in\{6,7,8\}$.
\pagebreak[3]

\section{The Schur positivity threshold}
\label{s Schur positivity}

As of this writing, the following tantalizing questions remain unanswered:
For which ideals $I$ are the symmetric functions $F_\gamma(\mathbf{x})$ Schur positive
for all vectors \hbox{$\gamma \in \U_{\ge 0}^* \cap I^\perp$}?
For~which $I$ are the noncommutative Schur functions
$\mathfrak{J}_\lambda(\mathbf{u})$ monomial positive
modulo~$I$?
In this section, we assemble our best attempts towards answers,
including known special cases (cf.~Figure~\ref{fig:positivity-for-switchboards}) and
strategies for further exploration.
We believe it particularly instructive to examine the ``phase
transition'' between Schur positivity and non-positivity---for
instance, by looking at two ideals
which are ``as close as possible'' while requiring
$\mathfrak{J}_\lambda(\mathbf{u})$ be monomial positive modulo one
but
not the other.

Here is one specific line of inquiry.
Given ideals $I_1,\dots,I_m$ for which monomial positivity
of $\mathfrak{J}_\lambda(\mathbf{u})$ is known (or expected) to hold modulo each~$I_j$,
is there a monomial positive expression for $\mathfrak{J}_\lambda(\mathbf{u})$
which is true modulo each ideal~$I_j$, and
consequently modulo the intersection~$\bigcap I_j$?


To simplify calculations, we prefer to work with ideals which
contain both the switchboard ideal~$\Irk$
and the ideal $\Ist$ generated by monomials with repeated
entries. 
This translates into restricting from
switchboards to $\Dzero$~graphs, see
Definitions~\ref{def:dzero-graph}--\ref{def:Ist}.

We begin by looking more closely at the counterexample in Figure~\ref{f schur pos counterexample}.
For $N=6$, the noncommutative Schur function
$\mathfrak{J}_{(2,2,2)}(\mathbf{u})$
is given by (cf.~\eqref{eq:ncschur-via-e})
\[
\mathfrak{J}_{(2,2,2)}(\mathbf{u})
=e_3(\mathbf{u})^2 - e_4(\mathbf{u}) e_2(\mathbf{u})
\equiv \sum_{\substack{\e{w}\in\mathcal{S}_6\\ \Des(\e{w})=\{1,2,4,5\}}}  
\mathbf{u}_{\e{w}}
- \sum_{\substack{\e{w}\in\mathcal{S}_6\\ \Des(\e{w})=\{1,2,3,5\}}}
\mathbf{u}_{\e{w}}
\ \bmod \Ist.
\]
Note that in the last expression, the two sums have 20 and~15
monomials, respectively.
Let us see how this formula
simplifies modulo some ideals.

\begin{proposition}
\label{p four pos expressions}
Let $N=6$. Modulo the ideal $\Irk+\Ist$,
the noncommutative Schur~function $\mathfrak{J}_{(2,2,2)}(\mathbf{u})$
cannot be written as a sum of 5~monomials,
nor as a sum of 6 monomials minus a monomial.
It can be written as a sum of 7 monomials minus 2 monomials:
\begin{align}
\mathfrak{J}_{(2,2,2)}(\mathbf{u}) \equiv
    \mathbf{u}_\e{321654} + \mathbf{u}_\e{426513} + \mathbf{u}_\e{562143} + \mathbf{u}_\e{436512} +
  \mathbf{u}_\e{563412}
&+ \mathbf{u}_\e{462315} + \mathbf{u}_\e{452316}  \notag\\
&- \mathbf{u}_\e{462351} - \mathbf{u}_\e{452361} \,.
\label{e J222}
\end{align}
Modulo the ideal 
\,$\bigcap_k (\Iassaf{k}+\Ist)$,
the noncommutative Schur function
$\mathfrak{J}_{(2,2,2)}(\mathbf{u})$
has exactly four  $\ZZ$-monomial positive expressions:
\begin{align}
\mathfrak{J}_{(2,2,2)}(\mathbf{u})
&\equiv \mathbf{u}_\e{321654} + \mathbf{u}_\e{426513} +
\mathbf{u}_\e{562143} + \mathbf{u}_\e{431652} +
\mathbf{u}_\e{563412} \label{e J222 all k}
\\
    &\equiv \mathbf{u}_\e{321654} + \mathbf{u}_\e{462153} +
\mathbf{u}_\e{521643} + \mathbf{u}_\e{436512} +
\mathbf{u}_\e{563412} \notag 
\\
    &\equiv \mathbf{u}_\e{321654} + \mathbf{u}_\e{462513} +
\mathbf{u}_\e{521643} + \mathbf{u}_\e{436152} +
\mathbf{u}_\e{563412} \notag 
\\
    &\equiv \mathbf{u}_\e{321654} + \mathbf{u}_\e{462513} + \mathbf{u}_\e{526143} + \mathbf{u}_\e{431652} + \mathbf{u}_\e{563412} \,.
\notag 
\end{align}
\end{proposition}


\begin{remark}
We feel that in this example,
the ideals $\Irk+\Ist$ and $\bigcap_k (\Iassaf{k}+\Ist)$
are very close to (but on different sides of) the threshold
where Schur positivity breaks.
Formula \eqref{e J222}, with only two minus signs, suggests that
$\mathfrak{J}_{(2,2,2)}(\mathbf{u})$ is
quite close to being  monomial positive modulo \hbox{$(\Irk+\Ist)$}. 
At the same time, the fact that there are only
four monomial positive expressions for $\mathfrak{J}_{(2,2,2)}(\mathbf{u})$
modulo $\bigcap_k (\Iassaf{k}+\Ist)$ suggests that monomial positivity just barely holds here.
Indeed, for an ideal $I$ generated by monomials and binomials of the form
$\mathbf{u}_{\e{v}}-\mathbf{u}_{\e{w}}$,
the number of monomial positive expressions
for~$\mathfrak{J}_\lambda(\mathbf{u})$ modulo~$I$
(assuming they exist)
is the product of the sizes of the equivalence classes $\bmod \,I$
to which these monomials belong
(assuming those equivalence classes are distinct).
This number is typically quite large,
e.g., it is equal to~$5^5$ for the plactic ideal in this example, cf.\ the proof
below.
\end{remark}

\begin{proof}[Proof of Proposition~\ref{p four pos expressions}]
We first describe a (somewhat brute force) alternative way of showing that
$\mathfrak{J}_{(2,2,2)}(\mathbf{u})$ is not $\ZZ$-monomial positive
modulo  $\Irk+\Ist$ (which implies the same for~$\Irk$,
cf.\ Corollary~\ref{cor-nopos}{\rm (ii)}).
Assume the contrary, i.e.,
\begin{equation}
\label{eq:J222=sum-of-five-monomials}
\mathfrak{J}_{(2,2,2)}(\mathbf{u})\equiv
\mathbf{u}_{\e{w}^1} + \mathbf{u}_{\e{w}^2}+\mathbf{u}_{\e{w}^3}+
\mathbf{u}_{\e{w}^4}+ \mathbf{u}_{\e{w}^5}
\ \bmod (\Irk+\Ist),
\end{equation}
for some monomials
$\mathbf{u}_{\e{w^1}},\dots,\mathbf{u}_{\e{w^5}}\in\U$.
Then the same congruence also holds
modulo $\Iplac+\Ist\supset \Irk+\Ist$, where  $\Iplac$ is the
plactic ideal from Definition~\ref{def:Iplac}.
Since $\Iplac+\Ist \supset \Ifg\,$, Theorem \ref{t FG} says that
\begin{equation}
\label{eq:J222=sum-of-column-words}
\mathfrak{J}_{(2,2,2)}(\mathbf{u})
\equiv \sum_{T \in \SSYT(\lambda; 6)} \mathbf{u}_{\creading(T)}
\equiv \sum_{T \in \SYT(\lambda; 6)} \mathbf{u}_{\creading(T)}
\ \bmod (\Iplac+\Ist),
\end{equation}
where  $\SYT(\lambda; N)$ is the set of standard Young tableaux
of shape $\lambda=(2,2,2)$. 
In order for the right-hand sides of
\eqref{eq:J222=sum-of-five-monomials} and
\eqref{eq:J222=sum-of-column-words}
to be congruent to each other $\bmod(\Iplac+\Ist)$,
the five words $\e{w}^1,\dots,\e{w}^5$ must belong (in some order)
to the Knuth equivalence classes of the five standard Young tableaux
of shape~$(2,2,2)$.
Each of these equivalence classes consists of $5$~words (indeed,
permutations of $\e{1,2,3,4,5,6}$).
We then checked, using a noncommutative Gr\"obner basis calculation in \texttt{Magma}~\cite{Magma},
that none of these $5^5$ possibilities is congruent to
$\mathfrak{J}_{(2,2,2)}(\mathbf{u})$ modulo~$\Irk+\Ist$.

A similar brute force calculation shows that there is no expression
for $\mathfrak{J}_{(2,2,2)}(\mathbf{u})$ as a sum of six monomials
minus another monomial that holds modulo $\Irk+\Ist$.
Further calculations produce expressions with only two minus signs,
including~\eqref{e J222}.


The second part of the proposition
is established by a noncommutative Gr\"obner basis calculation similar
to that above: for each of the  $5^5$ expressions as
in~\eqref{eq:J222=sum-of-five-monomials}, we checked
whether it holds modulo $\bigcap_k (\Iassaf{k}+\Ist)$.
Just as the argument above used $\Iplac+\Ist \supset \Irk+\Ist$,
this one relies on the inclusion
$\Iplac+\Ist = \Iassaf{1}+\Ist \supset \bigcap_k (\Iassaf{k}+\Ist)$.
\end{proof}


\begin{remark}
To better understand the transition from monomial positivity to non-positivity,
it is instructive to see how the expression in \eqref{e J222}
collapses to \eqref{e J222 all k} once we impose the relations
corresponding to $\bigcap_k (\Iassaf{k}+\Ist)$.
Let us rewrite~\eqref{e J222} as
\[
\mathfrak{J}_{(2,2,2)}(\mathbf{u}) \equiv
  \mathbf{u}_\e{321654} + \mathbf{u}_\e{426513} + \mathbf{u}_\e{562143} + \mathbf{u}_\e{436512} +
  \mathbf{u}_\e{563412} + \mathbf{u}_\e{4623[1,5]} + \mathbf{u}_\e{4523[1,6]} \bmod (\Irk+\Ist)
\]
where we used the notation
\[
\mathbf{u}_\e{v[a,b]w}
= \mathbf{u}_\e{v} u_a u_b \mathbf{u}_\e{w}
 -\mathbf{u}_\e{v} u_b u_a \mathbf{u}_\e{w}\,.
\]
Consider the term $\mathbf{u}_\e{4523[1,6]}$.
If $k \geq 5$, then $\mathbf{u}_\e{4523[1,6]} \equiv
\mathbf{u}_\e{4352[1,6]} \bmod (\Iassaf{k}\!+\!\Ist)$.
If $k < 5$, then $\mathbf{u}_\e{4523[1,6]} \equiv 0 \equiv
\mathbf{u}_\e{4352[1,6]} \bmod (\Iassaf{k}\!+\!\Ist)$.
Hence, computing here and below modulo $\bigcap_k (\Iassaf{k}\!+\!\Ist)$, we~get
$\mathbf{u}_\e{4523[1,6]} \equiv \mathbf{u}_\e{4352[1,6]} \equiv
\mathbf{u}_\e{43[1,6]52}$. 
Similar arguments yield
$\mathbf{u}_\e{4623[1,5]} \equiv \mathbf{u}_\e{4362[1,5]} \equiv
\mathbf{u}_\e{436[1,5]2}$.
Consequently the difference of \eqref{e J222} and \eqref{e J222 all k}
is equal to
\begin{align*}
\mathbf{u}_\e{436512}  + \mathbf{u}_\e{4623[1,5]} &+ \mathbf{u}_\e{4523[1,6]}
- \mathbf{u}_\e{431652}\\
&\equiv  \mathbf{u}_\e{436512} + \mathbf{u}_\e{436[1,5]2} +
\mathbf{u}_\e{43[1,6]52} - \mathbf{u}_\e{431652}
=\mathbf{u}_\e{436152} - \mathbf{u}_\e{436152}
=0.
\end{align*}
\end{remark}

In light of the above discussion, it is tempting to conjecture that
$\mathfrak{J}_\lambda(\mathbf{u})$ is always monomial positive modulo the
intersection $\bigcap_k (\Iassaf{k}\!+\!\Ist)$.
Unfortunately this turns out to be false,
as Proposition~\ref{p D0graph counterex} below shows.

Recall the ideals $\Ilam{k}=\Jlam{k}\cap\U$ from Definition~\ref{d Lam
  algebra}.

\pagebreak[3]

\begin{proposition}[{\cite[Corollaries 7.2--7.3]{BD0graph}}]
\label{p D0graph counterex}
Let $N=8$.
There exists a $\Dzero$~graph~$\Gamma$ 
on a vertex set~$W$ of cardinality 39696
such that
\begin{itemize}
\item
$\sum_{\e{w}\in W} \e{w}
\in  \big((\Iplac+\Ist) \cap (\Ilam{3}+\Ist) \cap (\Ilam{7}+\Ist) \big)^{\!\perp}
\subset  \big(\bigcap_k (\Iassaf{k}+\Ist)\big)^{\!\perp}
\subset  \big(\bigcap_k \Iassaf{k}\big)^{\!\perp}$;
\vspace{-3mm}
\item
the coefficient of
$s_{(2,2,2,2)}(\mathbf{x})$ in $F_\Gamma(\mathbf{x})$ 
is equal to~$-1$.
\end{itemize}
%
%
Consequently, $\mathfrak{J}_{(2,2,2,2)}(\mathbf{u})$ is not
$\QQ$-monomial positive  modulo $\bigcap_k (\Iassaf{k}+\Ist)$.
\end{proposition}


\begin{remark}
It is conceivable that $\mathfrak{J}_\lambda(\mathbf{u})$ is
$\ZZ$-monomial positive modulo $(\Iassaf{k} \cap \Iassaf{k+1})$ for
any~$k$,
and that the number of such monomial positive expressions
is reasonably small.
Then this may be a good route, at least
experimentally, towards finding
a monomial positive expression for
$\mathfrak{J}_\lambda(\mathbf{u})$
modulo $\Iassaf{k}$ for each particular~$k$.
\end{remark}

Let $\mathcal{I}_\lambda$ be the poset of all ideals $I\supset\Icomm$, ordered by containment,
such that $\mathfrak{J}_\lambda(\mathbf{u})$ is $\ZZ$-monomial positive
modulo~$I$.
What are the minimal elements of $\mathcal{I}_\lambda$?
Are there finitely many such minimal~ideals?
If $I$ is such a minimal ideal, is there a unique monomial positive
expression for $\mathfrak{J}_\lambda(\mathbf{u})$ modulo~$I$,
or can we at least control the number of such expressions?
%
%
Our investigations suggest that these questions are difficult.
One concrete reason is the following surprising
corollary of Proposition~\ref{p D0graph counterex}.

\begin{corollary}
\label{c no unique smallest}
There is no unique smallest ideal $I$ containing $\Icomm$ such that
$F_\gamma(\mathbf{x})$ is Schur positive for every $\gamma \in
\U^*_{\ge 0} \cap I^\perp$.
Similarly, there is no unique smallest ideal $I$ containing $\Icomm$ such that,
for all   $\lambda$,
$\mathfrak{J}_\lambda(\mathbf{u})$ is  $\ZZ$-monomial positive modulo  $I$.
\end{corollary}

\begin{proof}
Let $\mathcal{I}$ be the poset of all ideals $I\supset\Icomm$, ordered by containment,
such that $F_\gamma(\mathbf{x})$ is Schur positive for every $\gamma \in \U^*_{\ge 0} \cap I^\perp$.
We know that  $\Ilam{k} \in \mathcal{I}$ for all  $k$ by the Schur
positivity of LLT polynomials.
We also know that $\Iplac \in \mathcal{I}$.  If  $\mathcal{I}$ has a
unique smallest element  $J$, then
$J \subset  \bigcap_k \Ilam{k} \cap \Iplac \subset (\Iplac+\Ist) \cap
(\Ilam{3}+\Ist) \cap (\Ilam{7}+\Ist)$,
hence the latter intersection must lie in~$\mathcal{I}$.
This however contradicts Proposition~\ref{p D0graph counterex}.
%

The proof of the second statement is similar. We only need the additional fact that
$\mathfrak{J}_\lambda(\mathbf{u})$ is  $\ZZ$-monomial positive modulo  $\Ilam{k}$.
This follows from the easy fact that for an ideal  $I$ generated by monomials
and binomials of the form  $\mathbf{u}_\e{v}- \mathbf{u}_\e{w}$,  $\ZZ$- and $\QQ$-monomial positivity of $\mathfrak{J}_\lambda(\mathbf{u})$ modulo  $I$ are equivalent.
\end{proof}

One may hope to gain insight into the Schur
positivity phenomenon by studying \emph{nested} sequences of ideals,
and identifying the places where a transition from positivity to
non-positivity occurs.
Two examples are discussed below.

\begin{definition}
Let $\Ilam{\le k}$ denote the ideal in  $\U$ generated by the elements
\begin{alignat}{3}
& &&u_au_c - u_cu_a \qquad &&\text{for $c-a > k$,} \label{e Ilam le A}\\
& &&u_bu_au_c+u_cu_au_b-u_bu_cu_a - u_au_cu_b \qquad &&\text{for $c-a
    \le k$ and $a < b < c$,} \label{e Ilam le B}
\end{alignat}
together with all monomials $\mathbf{u}_{\e{w}}\in\Jlam{k}$.
Then $\Ilam{k} \supset \Ilam{\le k} \supset \Irk$ and $\Jlam{k}
\supset \QQ(q) \tsr_{\ZZ} \Ilam{\le k}$.
\end{definition}

The ideals $\Ilam{\le k}$ are nested: 
$\Ilam{\le k}\subset \Ilam{\le m}$ for $k\ge m$.
Monomial positivity of $\mathfrak{J}_{(2,2,2)}(\mathbf{u})$
fails modulo~$\Ilam{\le 5}$,
since it fails modulo $\Irk+\Ist$ by Corollary~\ref{cor-nopos}, and
$\Ilam{\le 5} \subset \Ilam{\le 5}+\Ist=\Irk+\Ist$ for $N=6$.
On~the other hand, Theorem~\ref{t sqread} can be strengthened as follows.

\begin{theorem}[{\cite[\textsection 5.2]{BLamLLT}}] \label{t sqread gen}
For any partition $\lambda$, the noncommutative Schur function $\mathfrak{J}_\lambda(\mathbf{u})$
is $\ZZ$-monomial positive modulo $\Ilam{\le 3}$.
A monomial expansion is given by
\begin{align}
\mathfrak{J}_\lambda(\mathbf{u}) \equiv \sum_{T \in \RSST(\lambda; N)}
\mathbf{u}_{\sqread(T)} \ \bmod \Ilam{\le 3}. \label{e sqread gen}
\end{align}
\end{theorem}

As a parallel to our discussion of the ideals $\Jlam{k}$ and  $\Iassaf{k}$ in Section~\ref{s LLT},
it is natural to also study the nested family of ideals
$\Ilam{\le k} \cap \Iassaf{k}$ (as $k$ varies).
The ideal  $\Ilam{\le k} \cap \Iassaf{k}$
contains the elements \eqref{e Iassaf1}--\eqref{e Iassaf2}, \eqref{e Ilam le B}, and all
monomials $\mathbf{u}_{\e{w}}\in\Jlam{k}$.
Just as above, monomial positivity of $\mathfrak{J}_{(2,2,2)}(\mathbf{u})$
fails modulo~$\Ilam{\le 5} \cap \Iassaf{5}$.
However, we conjecture the following strengthening of Theorem \ref{t
  sqread gen} (this is a slight variant of a
conjecture in {\cite[\textsection 5.2]{BLamLLT}}).

\begin{conjecture} \label{cj sqread}
For any partition $\lambda$, the noncommutative Schur function
$\mathfrak{J}_\lambda(\mathbf{u})$ is \linebreak[3] $\ZZ$-monomial positive modulo
$\Ilam{\le 3} \cap \Iassaf{3}$.  A monomial expansion is given by
\begin{align}
\mathfrak{J}_\lambda(\mathbf{u}) \equiv \sum_{T \in \RSST(\lambda; N)}
\mathbf{u}_{\sqread(T)} \ \bmod \Ilam{\le 3} \cap \Iassaf{3}. \label{e cj sqread}
\end{align}
\end{conjecture}

Conjecture \ref{cj sqread} is stronger than the same statement
modulo~$\Iassaf{3}$,
so it would imply that the Assaf symmetric functions of level 3 are
Schur positive (cf.\ Conjecture \ref{cj sami}).

We conclude this section with a summary of our current knowledge of
Schur positivity for symmetric functions associated to switchboards,
see Figure~\ref{fig:positivity-for-switchboards}.

\newcommand{\Tstrut}{\rule{0pt}{2.6ex}}
\begin{figure}[ht]
\begin{tabular}{llll}
Class of switchboards~$\Gamma$ & All $F_\Gamma$ Schur positive?
                                         & Explanation/proof \\[.5mm]
\hline
$\Dzero$ graphs                & No      &
Proposition~\ref{pr:456213+...} \Tstrut\\[.5mm] 
Switchboards
with $\gamma\in \big(\bigcap_k \Iassaf{k}\big)^\perp$
                               & No      & Proposition~\ref{p D0graph counterex} \\[.5mm]
Triples switchboards           & Unknown & Problem \ref{problem:triples}\\[.5mm]
Assaf switchboards             & Unknown & Conjecture~\ref{cj sami}\\[.5mm]
Switchboards with $\gamma\in (\Ilam{\le 3} \cap \Iassaf{3})^\perp$ & \text{Unknown}             & \text{Conjecture \ref{cj sqread}}  \\[.5mm]
Switchboards with $\gamma\in \Ilamleperp{3}$ & Yes & Theorem~\ref{t sqread gen} \cite{BLamLLT} \\[.5mm]
LLT switchboards               & Yes  & LLT positivity \cite{LT00, GH}\\[.5mm]
Rotation-free switchboards     & Yes  & Corollary~\ref{c FG} \cite{FG}\\[.5mm]
\end{tabular}
\caption{\label{fig:positivity-for-switchboards}
Schur positivity of symmetric functions~$F_\Gamma$ for various classes of
switchboards~$\Gamma$.
Here $\gamma
$
denotes the sum of the words appearing in~$\Gamma$.
}
\end{figure}


\section{Proof of Theorem~\ref{t intro positivity}
}
\label{s monomial positivity}

We will need the following basic convex geometry lemma 
(cf., e.g., \cite[Theorem~4.1]{GRMonoids}),
which we restate in the form convenient for our purposes.

\begin{lemma}[Minkowski-Farkas' Lemma]
\label{lem:farkas}
Suppose $\mathbf{a}^1, \mathbf{a}^2, \ldots, \mathbf{a}^r \in \QQ^m$
and $\mathfrak{J}\in \QQ^m$.
Let $\langle  \cdot , \cdot \rangle$ be the
standard
inner product on $\QQ^m$. Then the following are equivalent:
\begin{itemize}
\item $\mathfrak{J}$ lies in the cone $\QQ_{\geq 0}\{\mathbf{a}^1, \ldots, \mathbf{a}^r\}$.
\item
for any $\gamma \in \QQ^m$ satisfying $\langle \mathbf{a}^i, \gamma \rangle
\geq 0$ for all~$i$, we have $\langle \mathfrak{J}, \gamma \rangle \ge 0$.
\end{itemize}
Moreover, in the last line, we can replace $\gamma\in\QQ^m$ by
$\gamma\in\ZZ^m$.
\end{lemma}

Let $\U_d$ denote the degree $d$ homogeneous component of $\U$.
For a homogeneous ideal $I$ of $\U$, let $I_d = \U_d \cap I$ be the
degree $d$ homogeneous component of $I$.

\begin{theorem}
\label{t linear program}
Let $I$ be a homogeneous ideal in~$\U$.
For an element $\mathfrak{J} \in \U_d$, the following are equivalent:
\begin{itemize}
\item[{\rm (a)}] $\mathfrak{J}$ is $\QQ$-monomial positive modulo~$I$;
\item[{\rm (b)}] for any $\gamma \in \U_{\ge 0}^* \cap
  I_d^\perp$, there holds $ \langle \mathfrak{J}, \gamma \rangle
  \ge 0$.
\end{itemize}
\end{theorem}

\begin{proof}
Let $\mathbf{u}^1, \ldots, \mathbf{u}^m$ be the complete list of
degree~$d$ monomials in~$\U_d$.
Let $\mathbf{g}^1, \ldots, \mathbf{g}^t$ be a $\QQ$-basis of
  $\QQ I_d$.
We start by observing that $\U_{\ge 0}^* \cap
  I_d^\perp$ is precisely the set of vectors $\gamma\in\U^*$
satisfying $\langle\mathbf{u}^i,\gamma\rangle\ge0$,
$\langle\mathbf{g}^j,\gamma\rangle\ge0$, and
$\langle-\mathbf{g}^j,\gamma\rangle\ge0$ for all $i\le m$ and $j\le
t$.
Hence, by Lemma~\ref{lem:farkas}, condition~(b) is equivalent to
$\mathfrak{J} \in \QQ_{\geq 0}\{\mathbf{u}^1, \ldots,
\mathbf{u}^m,
\mathbf{g}^1, \ldots, \mathbf{g}^t, -\mathbf{g}^1, \ldots,
-\mathbf{g}^t\}$,
which is easily seen to be equivalent to condition~(a).
\end{proof}

We now prove Theorem~\ref{t intro positivity},
the stronger statement where $\lambda$ is fixed.
Let $\lambda$ be a partition of~$d$.
By Proposition~\ref{Icomm-Schur}, the following conditions are equivalent:
\begin{itemize}
\item[{\rm (i)}] for any $\gamma\in\U_{\ge 0}^*\cap I^\perp$,
the coefficient of $s_\lambda(\mathbf{x})$ in 
$F_\gamma(\mathbf{x})$ is nonnegative;
\item[{\rm (i${}'$)}] for any $\gamma\in\U_{\ge 0}^*\cap I^\perp$, there holds
$ \langle \mathfrak{J}_\lambda(\mathbf{u}), \gamma \rangle \ge 0$.
\end{itemize}
Furthermore, by Theorem~\ref{t linear program}, (i${}'$) is equivalent to
\begin{itemize}
\item[{\rm (ii)}] $\mathfrak{J}_\lambda(\mathbf{u})$ is $\QQ$-monomial positive modulo~$I$,
\end{itemize}
establishing the desired equivalence (i)$\Leftrightarrow$(ii). \hfill \qed

\section{Proof of Theorem \ref{th:FG'-positivity}}
\label{sec:proof-of-FG'-positivity}

We use the shorthand $[m]=\{1,\dots,m\}$ throughout.
By~convention, $[0]=\varnothing$,
so for example  $e_k(\mathbf{u}_{[0]})=0$ for $k\neq 0$, and $e_0(\mathbf{u}_{[0]})=1$.
Recall  (cf. \eqref{e ek def}) that  $e_k(\mathbf{u}_S)$ denotes the noncommutative elementary symmetric function in the variables $\{u_i\}_{i \in S}$.

We will prove Theorem \ref{th:FG'-positivity} by an inductive
computation of $\mathfrak{J}_\lambda(\mathbf{u})$
which  involves peeling off the diagonals of the shape~$\lambda$
and incorporating them into a tableau.
This requires the following flagged generalization of the noncommutative Schur functions.

\begin{definition}
For 
$\alpha\!=\!(\alpha_1,\dots,\alpha_l)\!\in\!\ZZ_{\ge 0}^l$
and 
$\mathbf{n}\!=\!(n_1,\dots,n_l) \!\in\! \{0,1,\ldots, N\}^l$,
we define the \emph{noncommutative column-flagged Schur function}
$J_{\alpha}(\mathbf{n})=J_{\alpha}(\mathbf{n})(\mathbf{u})$ by
\begin{align}
J_{\alpha}(\mathbf{n})
=\sum_{\pi\in \S_{l}}
\sgn(\pi) \, e_{\alpha_1+\pi(1)-1}(\mathbf{u}_{[n_1]})
\,e_{\alpha_2+\pi(2)-2}(\mathbf{u}_{[n_2]})\cdots
e_{\alpha_{l}+\pi(l)-l}(\mathbf{u}_{[n_l]}). \label{e flag schur}
\end{align}
These specialize to the noncommutative Schur functions $\mathfrak{J}_{\lambda}(\mathbf{u})$
via
$\mathfrak{J}_{\lambda}(\mathbf{u}) = J_{\lambda'}(N,\ldots,N)$.
\end{definition}

\begin{remark}
When the $u_i$ commute, 
$J_{\lambda}(\mathbf{n})$ becomes the column-flagged Schur function $S_\lambda^*(\mathbf{1},\mathbf{n})$ studied in \cite{Wachs}, where $\mathbf{1}=(1,\dots,1)\in\ZZ^l$.
\end{remark}

\begin{definition}
\label{d diagread}
Let $\mathbf{\alpha}=(\alpha_1,\ldots,\alpha_l)\in \ZZ_{\ge 0}^l$
be a composition satisfying $\alpha_{j+1}\le\alpha_j+1$ for all~$j$.
We denote by $\alpha'$ the
diagram whose  $j$th column contains the cells in rows $1,\dots,\alpha_j$.
Formally, $\alpha'=\{(i,j)|j\in [l], i\in [\alpha_j]\}\subset
\ZZ_{\ge 1} \times \ZZ_{\ge 1}$.
%
Now, for $\mathbf{n}=(n_1,\ldots,n_l)\in\ZZ_{\ge 0}^l$,
let $\SSYT(\alpha')^{\mathbf{n}}$ denote the set of
fillings \linebreak[3]
$T: \ZZ_{\ge 1} \times \ZZ_{\ge 1} \to [N]$ of the diagram~$\alpha'$ such that
\begin{itemize}
\item the entries weakly increase across the rows and strictly increase down the columns;
\item the entries in column  $j$ lie in $[n_j]$;
\item if $\alpha_j < \alpha_{j+1}$, then $T(\alpha_{j+1},j+1) > n_j$.
\end{itemize}
The \emph{diagonal reading word} of $T \in \SSYT(\alpha')^\mathbf{n}$, denoted $\diagread(T)$,
is obtained by concatenating the diagonals of~$T$
(reading each diagonal in the southeast  direction), starting with the
southwesternmost diagonal of~$\alpha'$.
For example,
\[
\diagread\Bigg( \, \text{\footnotesize
$\tableau{1&1&2&2 \\2&2 & 3\\ &  & 7}$
}
\, \Bigg) = \e{21271322}.
\]
\end{definition}

Theorem~\ref{th:FG'-positivity-technical} below is a generalization of
Theorem~\ref{th:FG'-positivity}.
To obtain Theorem~\ref{th:FG'-positivity}
from it, set  $n_1 = \cdots = n_l = N$ and $\alpha = \lambda'$.

\begin{theorem}
\label{th:FG'-positivity-technical}
Let  $\alpha=(\alpha_1,\dots, \alpha_l)\in\ZZ_{\ge 0}^l$ satisfy
$\alpha_{j+1}\le \alpha_j+1$ for all~$j$,
and let $0\leq n_1\leq \cdots\leq n_l \leq N$.
Then
\begin{align}\label{et:FG'-positivity-technical}
J_\alpha(\mathbf{n})
=J_{\alpha}(n_1, \ldots, n_l) \equiv
\sum_{ T \in \SSYT(\alpha')^{\mathbf{n}}} \mathbf{u}_{\diagread(T)} \bmod \Ifgp.
\end{align}
\end{theorem}

We will prove Theorem~\ref{th:FG'-positivity-technical} by
establishing its more technical variant
Theorem~\ref{t new statement fgprime}, see below.
To state the latter, we will need the following definition.

\begin{definition}
\label{def:Jalpha-augmented}
The \emph{augmented noncommutative column-flagged Schur function}
labeled by a weak composition $\alpha\!=\!(\alpha_1,\dots,\alpha_l)\!\in\!\ZZ_{\ge
  0}^l$,
an $l$-tuple $\mathbf{n}\!=\!(n_1,\dots,n_l) \!\in\! \{0,1,\ldots, N\}^l$,
and words
$\e{w}^1, \ldots, \e{w}^{l-1} \in \U^*$
is defined by
\begin{align}
&J_{\alpha}(n_1\Jnot{\e{w}^1}n_2\Jnot{\e{w}^2}\cdots\Jnot{\e{w}^{l-1}}n_l)  \label{e
    augmented J1} \\
&=\sum_{\pi\in \S_{l}}
\sgn(\pi) \ e_{\alpha_1+\pi(1)-1}(\mathbf{u}_{[n_1]}) \, \mathbf{u}_{\e{w}^1} \,
e_{\alpha_2+\pi(2)-2}(\mathbf{u}_{[n_2]})\, \mathbf{u}_{\e{w}^2}\cdots
\mathbf{u}_{\e{w}^{l-1}} \, e_{\alpha_{l}+\pi(l)-l}(\mathbf{u}_{[n_l]}). \label{e
  augmented J2}
\end{align}
We omit  $\e{w}^j$ from the notation in \eqref{e augmented J1} if the word $\e{w}^j$ is empty.
%
\end{definition}

The \emph{peeling index} of $\alpha\!=\!(\alpha_1,\dots, \alpha_l)\in\ZZ_{\ge 0}^l$
is the smallest $j$ such that $\alpha_j\geq \alpha_{j+1}$.
Here $\alpha_{l+1}=0$ by convention.

For a word  $\e{w}$ and a filling  $T$ of a diagram  $\alpha'$ (as in Definition \ref{d diagread}),
define the word $\diagread(\e{w};T)$ as follows:
if  $\alpha_1 = \alpha_2 = 0$, set $\diagread(\e{w};T) = \e{w}\,\diagread(T)$.
 Otherwise, $\diagread(\e{w};T) = \e{d^1 w d^2  \cdots d^t}$, where
$\e{d^1, d^2, \ldots, d^t}$ are the strictly increasing words
 corresponding to the (nonempty) diagonals of~$T$,
numbered starting with the southwesternmost diagonal.
(Hence $\diagread(\e{w};T) = \diagread(T)$ when  $\e{w}$ is empty.)

\begin{theorem}
\label{t new statement fgprime}
Let  $\alpha=(\alpha_1,\dots, \alpha_l)\in\ZZ_{\ge 0}^l$ be such that
$\alpha_{i+1}\le \alpha_i+1$ for all~$i$.
Let $j$ be the peeling index of~$\alpha$.
If $0 \le n_1\leq n_2\leq \cdots\leq n_l$, and $\e{w}$ is a strictly
increasing word in letters $> n_j$, then
\begin{align}
\label{e claim diagread}
 J_{\alpha}(n_1, \dots, n_{j-1}, n_j \Jnot{\e{w}}n_{j+1},\dots) \equiv
 \sum_{ T \in \SSYT(\alpha')^\mathbf{n}} \mathbf{u}_{\diagread(\e{w};T)}
 \bmod \Ifgp.
\end{align}
\end{theorem}

\begin{proof}
Write $J$ for $J_{\alpha}(n_1,\dots,n_j\Jnot{\e{w}}n_{j+1},\dots)$.
The proof is by induction on  $l$ and the~$n_i$.

We first address the base cases $(\alpha_1 = 0, j=1)$ and $(n_1 = 0, j=1)$.
Suppose $\alpha_1 = 0$ and $j=1$.
Since $\alpha_{i+1} \leq \alpha_{i}+1$ for all  $i$
and $\alpha_1 = \alpha_2 = 0$, $J_{\alpha}(n_1,\dots, n_l)$ is  a noncommutative version of
the determinant of a matrix whose first column, read downwards, is a 1 followed by~0's.
Then by Definition~\ref{def:Jalpha-augmented} and induction,
\begin{align}\label{e j1 base case}
J= \mathbf{u}_{\e{w}}J_{\hat{\alpha}}(\mathbf{\hat{n}}) \equiv
\mathbf{u}_{\e{w}}\sum_{ T \in \SSYT(\hat{\alpha}')^\mathbf{\hat{n}}} \mathbf{u}_{\diagread(T)}
 = \sum_{ T \in \SSYT(\alpha')^\mathbf{n}} \mathbf{u}_{\diagread(\e{w};T)} \ \bmod \Ifgp,
\end{align}
where  $\hat{\alpha} = (\alpha_2,\dots, \alpha_l)$ and $\mathbf{\hat{n}} = (n_2,\dots, n_l)$.
Now suppose $n_1 = 0$ and  $j=1$.
By the previous case, we may assume  $\alpha_1 > 0$.
Then  $J_{\alpha}(n_1,\dots, n_l)$ is a noncommutative version of the
determinant of a matrix whose first row is~0, hence $J=0$.
Since  $n_1 =0$ implies that $\SSYT(\alpha')^\mathbf{n}$ is empty, the right-hand side of \eqref{e claim diagread} is also 0.

The argument for the induction step requires the following three lemmas.

\begin{lemma}
\label{l elem sym J0}
If   $\alpha_{j+1}=\alpha_j+1 $ and  $n_{j+1}=n_j$, then
$J_\alpha(\mathbf{n}) \equiv 0 \bmod \Ifgp\,$.
More generally, 
this holds in the augmented case provided $\e{w}^j$ is empty.
\end{lemma}

\begin{proof}
This follows from the fact that the noncommutative elementary symmetric functions commute modulo~$\Ifgp$.
\end{proof}

\begin{lemma}
\label{l plactic inc word}
If the word $\e{v} \!= \! \e{v}_1 \cdots \e{v}_s$  is strictly decreasing
(i.e., $v_1>\cdots>v_s$) and
$\e{w}  \!= \! \e{w}_1 \cdots \e{w}_t$ is strictly increasing,
and moreover $v_1 < b < w_1$, then
$u_b \mathbf{u}_\e{v} \mathbf{u}_\e{w} \equiv u_b \mathbf{u}_\e{w} \mathbf{u}_\e{v}  \bmod \Ifgp$.
\end{lemma}

\begin{proof}
The case  $t=1$ follows immediately from  $s$ applications of the
plactic relation with three distinct letters.
Applying this special case $t$ times gives the result:
$u_b \mathbf{u}_\e{v} u_{w_1}\cdots u_{w_t} \equiv  u_b  u_{w_1} \mathbf{u}_\e{v}
u_{w_2}\cdots u_{w_t} \equiv u_b u_{w_1} u_{w_2} \mathbf{u}_\e{v} u_{w_3}\cdots
u_{w_t}
\equiv \cdots \equiv  u_b \mathbf{u}_\e{w} \mathbf{u}_\e{v}   \bmod
\Ifgp$.
\end{proof}

\begin{lemma}
Let $m\in\ZZ_{>0}$ and $k \in \ZZ$. Then
\begin{equation}
\label{e ek induction}
e_k(\mathbf{u}_{[m]}) =u_{m} e_{k-1}(\mathbf{u}_{[m-1]})+e_k(\mathbf{u}_{[m-1]}).
\end{equation}
\end{lemma}

\begin{proof}
This is immediate from the definition of $e_k(\mathbf{u}_{[m]})$.
\end{proof}

For an index  $j \in [l]$, we will apply \eqref{e ek induction} to  $J_\alpha(n_1, \ldots, n_l)$
and its variants by expanding
$e_{\alpha_j+\pi(j)-j}(\mathbf{u}_{[n_j]})$ in \eqref{e flag schur}
using \eqref{e ek induction} (so that \eqref{e ek induction} is
applied once to each of the  $l!$ terms in the sum in \eqref{e flag
  schur}).
We refer to this as a \emph{$j$-expansion of  $J_\alpha(n_1,\ldots, n_l)$
using~\eqref{e ek induction}}.

Set $\alpha_- = (\alpha_1,\dots, \alpha_{j-1}, \alpha_j-1, \alpha_{j+1},\dots)$ and
$\mathbf{n}_- = (n_1,\ldots,n_{j-1},n_j-1,n_{j+1},\ldots)$.
By the base cases above, we may assume that either ($\alpha_1 > 0$ and  $n_1 > 0$ and  $j=1$) or  $j >1$.
We first handle the former case.
A $1$-expansion of $J$ using \eqref{e ek induction} yields
\begin{align}
J&=u_{n_1}J_{\alpha_-}(n_1-1\Jnot{\e{w}}n_2,\dots, n_l)
+J_{\alpha}(n_1-1\Jnot{\e{w}}n_2,\dots, n_l) \notag 
\\
&\equiv
u_{n_1}\mathbf{u}_{\e{w}}J_{\alpha_-}(n_1-1,n_2,\dots, n_l)
+J_{\alpha}(n_1-1\Jnot{\e{w}}n_2,\dots, n_l)
\label{e un2}\\
&\equiv  ~u_{n_1}\mathbf{u}_{\e{w}}\sum_{ T \in \SSYT(\alpha_-')^{\mathbf{n}_-}}  \mathbf{u}_{\diagread(T)}
+ \sum_{ T \in \SSYT(\alpha')^{\mathbf{n}_-}}  \mathbf{u}_{\diagread(\e{w};T)} \ \bmod \Ifgp. \label{e diag rec0}
\end{align}
The first congruence is by Lemma~\ref{l plactic inc word} and the second is by induction.
Now \eqref{e diag rec0} is equal to the right-hand side of
\eqref{e claim diagread} because \eqref{e diag rec0} is the
result of partitioning the set of tableaux appearing in
\eqref{e claim diagread} into two, depending on whether or not  $T(\alpha_j,j)$
 is equal to or less than  $n_j$.
Note that if  $\alpha_j = \alpha_{j+1}$, then for $T \in \SSYT(\alpha')^{\mathbf{n}}$ with $T(\alpha_j,j) = n_j$,
the condition $T(\alpha_j,j) \le T(\alpha_{j+1},j+1)$ (which holds since the rows of  $T$ are weakly increasing)
corresponds  to the condition $n_j -1 = (n_-)_j < T'((\alpha_-)_{j+1},j+1) = T'(\alpha_{j+1},j+1)$ required of any $T' \in \SSYT(\alpha'_-)^{\mathbf{n}_-}$.

Now let $j > 1$.
By definition of the peeling index,
 $\alpha_{j}=\alpha_{j-1}+1$ and
$\alpha_j\geq \alpha_{j+1}$.
If $n_{j-1}=n_{j}$, then
$J\equiv 0 \bmod \Ifgp$ by Lemma~\ref{l elem sym J0};
in this case $\SSYT(\alpha')^\mathbf{n}$ is empty, so the right-hand side of \eqref{e claim diagread} is also 0.
So we may assume that $n_{j-1}<n_j$.
Now by a  $j$-expansion of  $J$ using \eqref{e ek induction}, we have
\begin{align}
J=&~J_{\alpha_-}(n_1,\dots, n_{j-1} \Jnot{\e{n_j}} n_j-1 \Jnot{\e{w}} n_{j+1},\dots) \notag\\
&+J_{\alpha}(n_1, \dots, n_{j-1}, n_j-1 \Jnot{\e{w}}n_{j+1},\dots)\notag\\
\equiv &~J_{\alpha_-}(n_1,\dots, n_{j-1}\Jnot{\e{n_jw}}n_j-1,n_{j+1},\dots) \label{e unj3}\\
&~+J_{\alpha}(n_1,\dots, n_{j-1}, n_j-1\Jnot{\e{w}}n_{j+1},\dots)
\label{e unj4}\\
\equiv  &~\sum_{ T \in \SSYT(\alpha_-')^{\mathbf{n}_-}}  \mathbf{u}_{\diagread(\e{n_jw};T)}
+ \sum_{ T \in \SSYT(\alpha')^{\mathbf{n}_-}}  \mathbf{u}_{\diagread(\e{w};T)} \ \bmod \Ifgp. \label{e diag rec2}
\end{align}
The first congruence is by Lemma~\ref{l plactic inc word}.
We now justify the second congruence.
Since the tuple $\mathbf{n}_-$ is weakly increasing, $j-1$ is the peeling index of
$\alpha_-$, and $\e{n}_j\e{w}$ is a strictly increasing word in letters $> n_{j-1}$,
the summand \eqref{e unj3} satisfies the conditions of
Theorem~\ref{t new statement fgprime}, hence by induction, it is congruent ($\bmod \ \Ifgp$) to the first sum in \eqref{e diag rec2}.
Since the tuple $\mathbf{n}_-$ is weakly increasing, the summand \eqref{e unj4} is congruent to
the second sum in \eqref{e diag rec2} by induction.
Finally, \eqref{e diag rec2} is equal to the right-hand side of
\eqref{e claim diagread} by the same argument given in the  $j=1$ case.
\end{proof}

We conclude by reconciling the difference between the reading words in
Theorems \ref{t FG} and~\ref{th:FG'-positivity-technical}.

\begin{proposition}
For any SSYT $T$, $\mathbf{u}_{\creading(T)} \equiv \mathbf{u}_{\diagread(T)} \bmod \, \Ifgp.$
Hence either the diagonal reading word or the column reading word can
be used in Theorems \ref{t FG} and \ref{th:FG'-positivity-technical}.
\end{proposition}

\begin{proof}
The proof is by induction on the number of cells of $T$.
Let $\e{d^1, d^2, \ldots, d^t}$ be the strictly increasing words corresponding to the diagonals of  $T$,
so that $\diagread(T) = \e{d^1d^2\cdots d^t}$.
Let $r$ be the number of rows of $T$.
For each $i=1,2,\ldots,r$, set $\e{d^i = c_i\hat{d}^i}$ so that $\e{c_i}$ is the first letter of $\e{d^i}$ and  $\e{\hat{d}^i}$ is the remainder of the word $\e{d^i}$.
The key computation is
\begin{align}
\mathbf{u}_{\e{d^1d^2\cdots d^r}} &= \mathbf{u}_{\e{c_1\hat{d}^1c_2\hat{d}^2\cdots c_r\hat{d}^r}} \equiv \mathbf{u}_{\e{c_1\hat{d}^1c_2\hat{d}^2\cdots c_{r-1}c_r\hat{d}^{r-1}\hat{d}^r }} \equiv \cdots \equiv \mathbf{u}_{\e{c_1c_2\cdots c_r\hat{d}^1\cdots\hat{d}^r}} \ \bmod \Ifgp, \label{e column diag}
\end{align}
where each congruence is by Lemma \ref{l plactic inc word} (whose hypotheses are satisfied since the $\e{d^i}$ are strictly increasing and $c_1 > c_2 > \cdots > c_r$).
Now let $T'$ be the tableau obtained from $T$ by removing the first column.
Note that $\e{c_1\cdots c_r}$ is the first column of $T$ read bottom to top, and $\e{\hat{d}^1\cdots\hat{d}^rd^{r+1}\cdots d^t}$ is the diagonal reading word of $T'$.
Hence by \eqref{e column diag} and induction,
we have
$\mathbf{u}_{\diagread(T)}\equiv \mathbf{u}_{\e{c_1\cdots c_r}}\mathbf{u}_{\diagread(T')} \equiv \mathbf{u}_{\e{c_1\cdots c_r}}\mathbf{u}_{\creading(T')} \equiv \mathbf{u}_{\creading(T)}
\bmod \Ifgp$.
\end{proof}



\end{document}